\documentclass[10pt,reqno]{amsart}
\usepackage{graphicx}

\title[]{Lagrangian  Flows, Maslov index zero and Special Lagrangians}
\author{Andrew A.~Cooper}
\address{Department of Mathematics\\
Box 8205, North Carolina State University\\ 
Raleigh, NC 27695}
\author{Jon Wolfson}
\address{Department of Mathematics \\ 
                 Michigan State University \\
                 East Lansing, MI 48824}

\date{\today}

\newcommand{\intprod}{\mbox{\rule{5pt}{.6pt}\rule{.6pt}{6pt}}~}

\newtheorem{thm}{Theorem}[section]
\newtheorem{lem}[thm]{Lemma}
\newtheorem{cor}[thm]{Corollary}
\newtheorem{prop}[thm]{Proposition}
\newtheorem{conj}[thm]{Conjecture}
\theoremstyle{defn}
\newtheorem{rem}[thm]{Remark}
\newtheorem{ex}[thm]{Example}

\newtheorem{defn}[thm]{Definition}

\numberwithin{equation}{section}

\renewcommand{\a}{\alpha}
\renewcommand{\b}{\beta}
\renewcommand{\d}{\delta}
\newcommand{\D}{\Delta}
\newcommand{\e}{\varepsilon}
\newcommand{\g}{\gamma}

\renewcommand{\l}{\lambda}
\renewcommand{\L}{\Lambda}
\newcommand{\n}{\nabla}
\newcommand{\p}{\partial}
\newcommand{\var}{\varphi}
\newcommand{\s}{\sigma}

\renewcommand{\th}{\theta}
\renewcommand{\t}{\tau}

\renewcommand{\o}{\omega}

\newcommand{\tr}{\operatorname{tr}}
\renewcommand{\div}{\operatorname{div}}
\newcommand{\Id}{\operatorname{Id}}
\newcommand{\supp}{\operatorname{supp}}
\newcommand{\II}{\operatorname{II}}

\newcommand{\cal}{\mathcal}

\def\Pb{\ifmmode{\Bbb P}\else{$\Bbb P$}\fi}
\def\N{\ifmmode{\Bbb N}\else{$\Bbb N$}\fi}
\def\Z {\ifmmode{\Bbb Z }\else{$\Bbb  Z$}\fi}
\def\C{\ifmmode{\Bbb C}\else{$\Bbb C$}\fi}
\def\R{\ifmmode{\Bbb R}\else{$\Bbb R$}\fi}

\makeatletter

\newcommand{\Rmnum}[1]{\expandafter\@slowromancap\romannumeral #1@}
\makeatother

\def\dist{\operatorname{dist}}

\def\dist{\operatorname{dist}}

\def\supp{\operatorname{supp}}
\def\Ric{\operatorname{Ric}}
\def\div{\operatorname{div}}
\def\proj{\operatorname{proj}}

\begin{document}

\begin{abstract}

	We introduce a notion of vanishing Maslov index for lagrangian varifolds and lagrangian integral cycles in a Calabi-Yau manifold. We construct mass-decreasing flows of lagrangian varifolds and lagrangian cycles which satisfy this condition. The flow of cycles converges, at infinite time, to a sum of special lagrangian cycles.
	
	We use the flow of cycles to obtain the fact that special lagrangian cycles generate the part of the lagrangian homology which lies in the image of the Hurewicz homomorphism. We also establish a weak version of a conjecture of Thomas-Yau regarding lagrangian mean curvature flow.

\end{abstract}

\maketitle

\setcounter{secnumdepth}{2}

\setcounter{section}{-1}

\section{Introduction}

In a K\"ahler manifold which is Calabi-Yau or, more generally, K\"ahler-Einstein, the mean curvature vector of a lagrangian submanifold is an infinitesimal symplectic motion and therefore preserves the lagrangian constraint. This implies that the problem of minimizing volume among lagrangian submanifolds in a  K\"ahler-Einstein manifold is formally possible. It also implies that mean curvature flow in a  K\"ahler-Einstein manifold preserves the lagrangian constraint for as long as the flow exists. This property inspired Thomas and Yau to conjecture \cite{ty}  that a stable embedded  lagrangian submanifold in a Calabi-Yau manifold with vanishing Maslov class will,  under classical mean curvature flow,  converge to an embedded lagrangian submanifold with mean curvature zero. Such a limiting submanifold would then be a special lagrangian submanifold calibrated by a parallel section of the canonical line bundle, hence minimizing. The precise statement of Thomas-Yau's conjecture allows for degeneration resulting from connect sums and other simple singularities, but the spirit of the conjecture is that classical mean curvature flow leads to special lagrangian submanifolds.  A. Neves \cite{nzm} has shown that  there are embedded  lagrangian submanifolds in a Calabi-Yau manifolds with vanishing Maslov class that,  under classical mean curvature flow, develop singularities in finite time. These examples give counterexamples to the Thomas-Yau conjecture if the assumption of stability is removed.

On the other hand, mean curvature of any codimension for singular submanifolds (rectifiable varifolds) has been defined and studied by Brakke \cite{b}. It is natural to try to extend Brakke's work to the lagrangian setting. However this is not possible. By work of the second author \cite{w1} there are lagrangian surfaces in a Calabi-Yau surface with non-zero mean curvature which minimize area among lagrangian representing a homology class. Under Brakke's flow, indeed under any generalized mean curvature flow, the area of such a surface will decrease. But since the surface minimizes among lagrangians, the flow cannot then preserve the lagrangian constraint. Taken together these two observations seem to imply that neither classical nor weak mean curvature flow is useful in the construction of special lagrangians. \\

In this paper we introduce and study a {\em lagrangian varifold flow}. This is a flow of certain lagrangian rectifiable varifolds in euclidean space, or, more generally, in a Calabi-Yau manifold, which we call {\em Maslov index zero rectifiable varifolds}. We give the precise definition below, but note here that an embedded lagrangian submanifold with zero Maslov index in the usual sense satisfies our  definition and an embedded lagrangian submanifold with non-vanishing Maslov class in the usual sense does not.

Our first result concerns a lagrangian rectifiable varifold in euclidean space $\R^{2n}$.

\begin{thm}
	Let $V$ be a compactly-supported Maslov index zero varifold with $H\in L^2(V)$ in $\R^{2n}$. There is a one-parameter family $V(t)$, $t>0$, of zero-Maslov varifolds with $H\in L^2(V)$ with:
		\begin{enumerate}
			\item $V=V(0)$.
			\item $\lVert V(t)\rVert$ is nonincreasing and lower-semicontinuous in $t$.
			\item On spacetime regions where $V(t)$ is a $C^2$-submanifold, $V(t)$ moves by lagrangian mean curvature flow.
		\end{enumerate}
\end{thm}

 We next consider a lagrangian homology class $\a \neq 0$ in a closed Calabi-Yau manifold.
 
 \begin{thm}
	Let $I$ be a rectifiable lagrangian cycle representing $\a$ whose associated varifold is a Maslov index zero has $H\in L^2(V)$. There is a one-parameter family $I(t)$ of rectifiable cycles with:
		\begin{enumerate}
			\item $I=I(0)$.
			\item $M(I(t))$ is nonincreasing and lower-semicontinuous in $t$.
		\end{enumerate}
	Moreover, $I(t)$ converges sequentially, as $t\rightarrow\infty$, to $I(\infty)$, which can be written as a sum of finitely many special lagrangian cycles (possibly with different phases).
\end{thm}

 As a consequence of our results, we have:

 \begin{cor}
 \label{cor:Thomas-Yau}
 Let $N$ be a closed Calabi-Yau manifold. 
If an integral lagrangian homology class $\a \in H_n(N; \Z)$ can be represented by an imbedded lagrangian submanifold with zero Maslov index then $\a = \a_1 + \dots + \a_k$ where each $\a_i \in H_n(N; \Z)$ is a lagrangian homology class that can be represented by a special lagrangian current. The phases of the calibrating $n$-forms may be different for each $i=1, \dots, k$. 
 \end{cor} 
 
 The special lagrangian varifolds are volume minimizing in their homology class and therefore by the work of Almgren \cite{al} have the regularity of volume minimizers, namely, they are regular except on a set of Hausdorff codimension two. We say they are {\it special lagrangian varieties}.
 
 Let $N$ be a closed Calabi-Yau manifold of complex dimension $n$. A class in $H_n(N, \Z)$ is called {\it a lagrangian homology class} if it can be represented by a simplex consisting of simplices with all $n$-simplices lagrangian. Denote the subspace of lagrangian homology classes by $LH_n(N, \Z)$. Let ${\cal S} \subset LH_n(N, \Z)$ be the subspace of the lagrangian homology that is generated by the special lagrangian cycles. 
 
 \begin{thm}
Let $N$ be a closed Calabi-Yau manifold of complex dimension $n$. 
If an integral lagrangian homology class $\a \in H_n(N; \Z)$ can be represented by the image of a smooth map $f: M \to N$, where $M$ is a simply connected closed $n$-manifold then $\a \in {\cal S}$. In particular, the image of the Hurewicz homomorphism $\pi_n(N) \to H_n(N, \Z)$ in the lagrangian homology lies in ${\cal S}$.
 \end{thm} 
 
  \begin{cor}
Let $N$ be a simply connected closed Calabi-Yau manifold of complex dimension $n$. Then for $n = 2, 3, 4, 5, 6$ all lagrangian homology classes lie in  ${\cal S}$.
   \end{cor}

\bigskip

A central idea is the definition of a lagrangian integer-rectifiable varifold with Maslov index zero. We do not define the Maslov class on such a varifold, only the more restricted notion of Maslov index zero. We exploit the geometry of the ambient Calabi-Yau manifold to make the definition. On any symplectic $n$-manifold $N$ consider the fiber bundle whose fiber at $x \in N$ consists of the Grassmann bundle $LGr(n)$ of oriented lagrangian $n$planes at $x$. On a Calabi-Yau manifold there is a  complex $(n,0)$-form $\sigma$ that is a parallel section of the canonical line bundle. Evaluating this form on a lagrangian $n$-plane defines a map:
$$
\ell: LGr(n) \to S^1
$$
that we call the $S^1$-valued lagrangian angle. It determines a map:
$$
\b: LGr(n) \to \R
$$
where $\b$ is defined mod $\pi$.  We call this the lagrangian angle. A rectifiable $n$-varifold $V$ has an approximate tangent $n$-plane almost everywhere and therefore $\b$ is well defined mod $\pi$ almost everywhere on $V$. 

If $V$ is a $C^1$ embedded lagrangian submanifold then $\b$ is well defined mod $2\pi$ everywhere, and the classical notion of Maslov index zero is equivalent to the statement that $\b$ has a continuous lift to a $\R$-valued function. On a varifold, we cannot require that $\b$ has a continuous lift to a scalar-valued function. We will require that $\b$ has a lift to a scalar valued function in $L^\infty(V)$. Abusing notation we will denote this lift $\b$ as well. This condition is not itself adequate, since on a $C^1$ embedded lagrangian submanifold selecting such a lift is always possible by allowing $\b$ to have jump discontinuities across codimension one sets.

To eliminate the possibility of such jump discontinuities we introduce a notion of  the weak derivative and require that the weak derivative $B$ of $\b \in L^\infty(V)$ lies in $L^2(V)$ with values $TN$. In sum, a lagrangian rectifiable varifold has Maslov index zero if the $S^1$ lagrangian angle admits a lift to a scalar function $\b \in L^\infty(V)$ and $\b$ has a weak derivative in $L^2(V)$.

If $V$ is a lagrangian rectifiable varifold with Maslov index zero then on $V$ we have a function $\b$.  We would like to use $\b$ as a hamiltonian potential function to generate a flow; however, $\b$ is merely a function on $V$. For each $\e>0$, we smooth $\b$ to $\b_\e$, a smooth function on $N$ given by mollification with respect to $V$. Using $\b_\e$ a hamiltonian potential to define a hamiltonian flow on $N$ and flowing for time $\D t$ yields a new lagrangian rectifiable varifold with Maslov index zero. Iterating and letting $\D t \to 0$ we construct a one-parameter family of lagrangian varifolds for each $\e > 0$. Then we let $\e \to 0$ to construct the ``hamiltonian flow'' referred to above. This ``hamiltonian flow'' consists of integral lagrangian rectifiable varifolds with Maslov index zero. The  mass of the varifolds is non-increasing and decreasing if the lagrangian angle is not a constant.

\bigskip

We are indebted to a number of excellent works on varifolds. We mention the work of Brakke \cite{b} on generalized mean curvature flow, the work of Allard \cite{a} on varifolds, the recent work of Menne \cite{m2} on rectifiable varifolds with locally bounded first variation and finally the book of Simon \cite{s} on the whole subject of geometric measure theory.

\section{Background and Motivation}

Let $(N, \o)$ be a K\"ahler manifold of complex dimension $n$ and let $\imath: L \to N$ be a real $n$-dimension imbedded (or immersed) submanifold. We say $L$ is {\it lagrangian} if $\imath^* \o = 0$. Let $H$ be the mean curvature vector field along $L$ and $\Ric$ denote the Ricci two-form of $M$. A simple geometric computation yields
$$
d(H \intprod \; \o) = \Ric. 
$$
In particular if the K\"ahler metric is K\"ahler-Einstein so that $\Ric = R \o$ then  $\imath^*d(H \intprod \; \o) =0$. From Cartan's formula it follows that 
$$
{\cal L}_H \o =0.
$$
Thus, $H$ is an infinitesimal symplectic motion. It then follows that:

\begin{thm}
If $(N, \o)$ is a  K\"ahler-Einstein manifold and $\imath: L \to N$ is an immersed lagrangian submanifold then classical mean curvature flow exists for some time interval $[0, T)$ and for all $t \in [0, T)$ the submanifold $L_t$ is lagrangian.
\end{thm}

Classical mean curvature flow cannot be defined when the flow develops a singularity and so, in general, classical mean curvature flows exist only finite intervals. Brakke \cite{b} was able to define a ``generalized mean curvature flow'' for a class of singular submanifolds known as rectifiable varifolds. Rectifiable varifolds have a tangent plane almost everywhere but are otherwise highly singular. To define a lagrangian rectifiable varifold we require the tangent plane to be a lagrangian plane. Then it is reasonable to expect that the Brakke flow preserves the lagrangian condition if the ambient manifold if K\"ahler-Einstein. However, this is false due to a result of the second author.

\begin{thm}
In general the Brakke flow does not preserve the lagrangian constraint.
\end{thm}

To explain this result we must find a lagrangian rectifiable varifold in a  K\"ahler-Einstein manifold for which the Brakke flow exists but does not preserve the lagrangian constraint. We consider a lagrangian homology class represented by an immersed two-sphere in a K3 surface. Use the minimization procedure of \cite{sw}
to construct a lagrangian two-sphere that minimizes area in its homology class. By careful choice of the lagrangian homology class it can be shown that the class cannot be represented by a holomorphic curve for any complex structure compatible with the metric. In particular the lagrangian minimizer is not a holomorphic curve for any compatible complex structure. Applying the regularity theory of \cite{sw}, the minimizer is regular (smooth) except at finitely many points which are (i) branch points, or (ii) SW singularities. The SW singularities are modeled by a countable family of non-trivial lagrangian cones that can be written explicitly. Around the vertex of each cone there is non-zero Maslov index. If there are no SW singularities then the minimizer is a classical (branched) minimal surface that is lagrangian. For a compatible complex structure such a surface is a holomorphic curve. But this has been ruled out and therefore there must be SW singularities and the minimizer cannot be a classical minimal surface. In particular the lagrangian minimizer is a lagrangian rectificable varifold with non-vanishing generalized mean curvature. The Brakke flow decreases the area of the lagrangian minimizer and preserves the homology class. But since the lagrangian minimizer minimizes area among lagrangians the Brakke flow cannot preserve the lagrangian constraint.

The SW singularities lie at the core of these examples. 
We observe that the link of any SW singularity has non-zero Maslov index. Because our definition of Maslov index zero precludes even local Maslov index, the present paper is a demonstration that it is the non-zero {\em local} Maslov index which obstructs the existence of any weak lagrangian mass-decreasing flow.

\bigskip

\section{Lagrangian geometry}

Let $(N, \o)$ be a K\"ahler $n$-manifold. At each point $x \in N$ a real dimension $n$ subspace $P \in T_xN$ is called {\it lagrangian} if $\omega_{|_P} = 0$. At each point $x \in N$ we consider the Grassmann manifold of lagrangian $n$-planes $L \subset T_xN$. Denote this manifold $LGr(x)$ and note that $LGr(x) $ can be identified with the homogeneous space $U(n)/O(n)$. These manifolds are the fibers of a bundle over $N$ that we call the lagrangian Grassmann bundle and denote $LGr$. (We can instead consider the Grassmann manifold of oriented lagrangian $n$-planes $L \subset T_xN$. In this case the relevant Grassmann manifold is identified with the homogeneous space $U(n)/SO(n)$.)

Suppose next that $(N, \o)$ has a Calabi-Yau metric. This implies that the canonical line bundle is geometrically trivial, in particular, that it has a non-zero parallel section. Such a section is a nowhere vanishing closed $(n,0)$ form that we denote $\ell$. Restricting $\ell$ to a lagrangian $n$-plane determines a unit complex number and therefore $\ell$ defines a smooth map:
\begin{equation}
\label{equ:def-lagr-angle}
\ell: LGr\to S^1
\end{equation}
We call this map the lagrangian angle. We will often write $\ell(x, L) = e^{i\b(x, L)}$, where the function $\b$ is well-defined mod $\pi$. (In the oriented case, $\b$ is well-defined mod $2\pi$.) Along a lagrangian submanifold $\Sigma$, tangent plane $T_x \Sigma$ is lagrangian and we can write $\ell(x)=\ell(T_x\Sigma)= e^{i\b(x)}$. On $\Sigma$ the function $\b$ is well-defined mod $\pi$. (In the oriented case, $\b$ is well-defined mod $2\pi$). In both cases the tangential gradient of $\b$ along $\Sigma$, $\n \b$, is well defined.

\subsection{The Maslov index and the Maslov class}
If $L:\Sigma\to N$ is a lagrangian immersion, we can consider the assignment $p\mapsto \beta(i_*T_p\Sigma)$, which gives a mod-$2\pi$ smooth map $\Sigma\to {\mathbb R}$ (which we continue to denote by $\beta$). It is immediate that $d\beta$ is a well-defined closed one-form so $[d\beta]\in H^1(\Sigma;{\mathbb R})$, called the {\em Maslov class} of the immersion $L$.

A simple computation (see e.g. \cite{hl}) shows that, if $H$ is the mean curvature of $L$, then
$$
d \b = h = i^*(H \intprod  \; \o).
$$
The {\em Maslov index} around a one-cycle $\a$ in $\Sigma$ is given by
$$
\mbox{Mas}(\a) = \frac{1}{2\pi}\int_\a d\b.
$$ 
which computes the winding of $\sigma(T_x\Sigma)$ as $x$ traverses $\a$. In particular, $[h]=[d\beta]$ lies in the $2\pi{\mathbb Z}$ lattice of $H^1(\Sigma;{\mathbb R})$.

\begin{thm}
\label{thm:Maslov-zero} 
If $\mbox{Mas}(\a) = 0$ for all one-cycles $\a$ then the lagrangian angle $\b$ admits a smooth lift to a smooth scalar valued function; conversely if the lagrangian angle $\b$ admits a smooth lift to a smooth scalar valued function then $\mbox{Mas}(\a) = 0$ for all one-cycles $\a$.
\end{thm}

\subsection{ Variation of the lagrangian angle}\label{sec:variationofbeta}
Our approach below will be to define a flow using the lagrangian angle $\beta$. As this function is central to our construction, we need to understand how $\beta$ behaves under smooth deformations which preserve the lagrangian condition. We will use the following convenient way to compute the lagrangian angle.

\begin{defn}
	Given a lagrangian $n$-plane $S$ at $x\in N$ with oriented orthonormal basis $\{e_1,\ldots, e_n\}$,the associated orienting form is $\xi=e_1\wedge\cdots\wedge e_n$. The associated complexified orienting form is $\xi_{\C}^S=\bigwedge_k\left(e_k-iJe_k\right)$
\end{defn}

It is elementary that the lagrangian condition implies $\xi_\C^S$ has the properties: 
	\begin{align}
		J\xi_\C^S&=i\xi_\C^S\label{eqn.n0form}\\
		\lvert\xi_\C^S\rvert&=1
	\end{align}
Note that the space of $n$-vectors on $S\otimes\C$ satisfying (\ref{eqn.n0form}) is of complex dimension 1. The $n$ vector $\tilde\ell$ dual to the parallel section $\ell(x)$, when restricted to $S\otimes \C$, also satisfies (\ref{eqn.n0form}), so it must be a unit complex multiple of $\xi_\C^S$. In fact\begin{equation}
			\label{eqn.betadef}\tilde{\ell}(x)|_S=e^{-i\beta(S)}\xi_\C^S
		\end{equation}

\begin{thm}
	Let $X$ be an infinitesimal symplectic motion. Then the lagrangian angle $\beta$ of the plane $S$ has first variation
		$$\delta_X\beta(S)=-\tr_S DJX$$
	where $D$ is the Levi-Civita connection and $J$ is the almost-complex structure coming from the Calabi-Yau manifold $(N,\omega)$.
\end{thm}

\begin{proof}
Let $\psi_t$ be a one-parameter family of diffeomorphisms $N\rightarrow N$, generated by the vector field $X$ (which we assume is $C_c^1$). By assumption $\psi_t$ preserves the lagrangian condition, so that $\psi$ induces a diffeomorphism $\psi_\sharp $ on $LGr $ which covers $\psi$.

Then for any $x\in N$, there are coordinates for $N$ at $x$ in which, for small $t$, \begin{align}
		\psi_t(x)&=x+tX(x)+O(t^2)\\
		d\psi_t|_x(\tau)&=\tau + tD_\tau X|_x+O(t^2)
	\end{align}

If $S$ is a lagrangian plane with orthonormal frame $e_1,\ldots,e_n$, we compute
\begin{equation}
\begin{aligned}
	&\left(\psi_t\right)_\sharp \xi=\bigwedge_k d\psi_te_k=\bigwedge_k\left(e_k+tD_{e_k}X+O(t^2)\right)\\
	&=\bigwedge_ke_k+t\sum_{k}(-1)^{k-1}D_{e_k}X\wedge\bigwedge_{s\neq k}e_s+O(t^2)\\
	&=\bigwedge_ke_k + t\sum_{k}(-1)^{k-1}(D_{e_k}X\cdot e_p)(\ e_p\wedge\bigwedge_{s\neq k}e_s)\\ \nonumber
         & \hspace{1.2cm} + t\sum_{k}(-1)^{k-1}(D_{e_k}X\cdot Je_p)(\ Je_p\wedge\bigwedge_{s\neq k}e_s) +O(t^2)\\
	&=\bigwedge_ke_k+t\tr_SDX \bigwedge_ke_k-t\sum_{k}(-1)^{k-1}(D_{e_k}JX\cdot e_p)(Je_p\wedge\bigwedge_{s\neq k}e_s)+O(t^2)
\end{aligned}
\end{equation}

Here we have used the fact that $S$ is lagrangian to decompose $D_{e_k}X$ into normal and tangential components, and write everything in terms of $e_1,\ldots,e_n$.

To compute the variation of the lagrangian angle, we will consider the function $e^{-i\beta}:LGr \rightarrow S^1$. Then we have
\begin{equation} \label{eqn.deltabeta}\delta_Xe^{-i\beta(x,S)}=\frac{d}{dt}|_{t=0}e^{-i\beta\left(\left(\psi_t\right)_\sharp (x,S)\right)}={\mathbb D}e^{-i\beta(S)}\cdot \frac{d}{dt}|_{t=0}(\psi_t)_\sharp  (x,S)
	\end{equation}
where ${\mathbb D}$ is the Levi-Civita connection coming from the bundle metric on $LGr$ induced by the Calabi-Yau metric on $N$ and the homogeneous metric on each fiber.\bigskip

To compute ${\mathbb D}e^{-i\beta(S)}$, note that the tangent plane $T_{(x,S)}LGr $ splits as $T_x N\oplus T_SG_{\text{Lag}}$. In this splitting, $\frac{d}{dt}|_{t=0}(\psi_t)_\sharp  (x,S)=\langle X,\delta_XS\rangle$. Because $\sigma$ is parallel, so is $\tilde{\sigma}$ and we have ${\mathbb D}e^{-i\beta(S)}=\langle 0,*\rangle$.

If we choose an orthonormal frame $\{e_1,\ldots,e_n\}$ for a lagrangian plane $S$ and identify $S$ with the orienting form $\xi=e_1\wedge\cdots\wedge e_n$,  we can give a basis for $T_SG_{\text{Lag}}$:
	\begin{equation}
		\theta_{ij}^S=Je_i\wedge\bigwedge_{s\neq j}e_s
	\end{equation}
for $i\leq j$. 

To compute the $\theta_{ij}^S$ component of ${\mathbb D}e^{i\beta(S)}$, let $S_{ij}(t)$ be a path of planes with $S_{ij}(0)=S$ and $S_{ij}'(0)=\theta_{ij}^S$. Then we compute
	\begin{equation}
	\begin{aligned}
		{\mathbb D}e^{-i\beta(S)}\cdot \theta_{ij}^S=\frac{d}{dt}|_{t=0}e^{-i\beta(S_{ij}(t))}=&\frac{d}{dt}|_{t=0}\langle \tilde{\ell},\xi_{ij}(t)\rangle\\
		=&\langle \tilde{\ell}, \theta_{ij}^S\rangle\\
		=&e^{-i\beta(S)}\langle \xi_\C^S,\theta_{ij}^S\rangle\\
		=&e^{-i\beta(S)}\langle \bigwedge_k (e_k-iJe_k),Je_i\wedge\bigwedge_{s\neq j}e_s\rangle\\
		=&-ie^{-i\beta(S)}\sum_i(-1)^{i-1}\delta^{ij}
	\end{aligned}
	\end{equation}

So we have the formula
	\begin{equation}\label{eqn.grassgradbeta}
	\begin{aligned}
		{\mathbb D}e^{-i\beta(S)}=&-ie^{-i\beta(S)}\sum_i (-1)^{i-1}\delta^{ij}\theta_{ij}^S=-ie^{-i\beta(S)}\sum_j(-1)^{j-1}Je_j\wedge\bigwedge_{s\neq j}e_j
	\end{aligned}
	\end{equation}

Applying (\ref{eqn.grassgradbeta}) to (\ref{eqn.deltabeta}), we have
	\begin{align} \nonumber
		\label{eqn.deltabeta2}\delta_Xe^{-i\beta(S)}=&-ie^{-i\beta(S)}\langle \sum_j(-1)^{j-1}Je_j\wedge\bigwedge_{s\neq j}e_s,\tr_SDX \bigwedge_ke_k\\ \nonumber
&-\sum_{k}(-1)^{k-1}D_{e_k}JX\cdot e_pJe_p\wedge\bigwedge_{s\neq k}e_s\rangle\\ 
		=&ie^{-i\beta(S)}\tr_SDJX 
	\end{align}
	
Since $\beta$ is only ambiguous up to addition of a multiple of $2\pi$, $\mathbb{D}\beta$ (and therefore $\delta_X\beta$) is well-defined. (\ref{eqn.deltabeta2}) yields
	\begin{equation}
		\label{eqn.deltabeta3}\delta_X\beta(S)=-\tr_SDJX
	\end{equation}
	
To carry out this computation, we identified $S$ with an orienting form $\xi^S$. Had we chosen the opposite orientation, we would have computed $\beta(S)$ using $-\xi^S$, which would result in an angle of $\beta(S)+\pi$. The formula \eqref{eqn.deltabeta3} is not affected by the addition of this constant.

\end{proof}

\bigskip

\section{Varifolds and lagrangian varifolds}
We will be concerned mostly with one-parameter families of integer-multiplicity rectifiable varifolds. For details on the theory of rectifiable varifolds, we refer the reader to Leon Simon's book \cite{s}; we briefly recall the salient points here as formulated in \cite{s}, \cite{a}, and \cite{m2} for general dimension and codimension.

Following Simon's notation, if $M$ is a countably $n$-rectifiable subset of ${\mathbb R}^N$ with multiplicity $\theta$, we write $V={\mathbf v}(M,\theta)$. To each such $V={\mathbf v}(M,\theta)$ we may associate a Radon measure on ${\mathbb R}^{N}$,  $\lVert V\rVert$, given by
$$\lVert V\rVert={\cal H}^n \; {\mbox{\rule{.6pt}{6pt}\rule{5pt}{.6pt}}~}  \; \theta$$
In particular, if $A$ is ${\cal H}^n$-measurable, then
$$\lVert V\rVert(A)=\int_{A\cap M}\theta d{\cal H}^n$$

The rectifiable set $M$ has an approximate tangent $n$-plane with respect to $\th$ for ${\cal H}^n$ a.e.~$x$, denoted $T_x M$. We then define the approximate tangent plane to the varifold $V$ by:
$$
T_x V = T_x M,
$$
for ${\cal H}^n$ a.e.~$x$.

If $X$ is an ambient $C^1$ vector field the divergence of $X$ along $V$, written $\div_V X$, is given by
$$\div_VX(x) = \operatorname{tr}_{T_xV}DX(x)$$

The first variation formula for the rectifiable varifold $V$ with respect to the variation $X$ is:
$$
\d V(X) = \int \div_V X d\lVert V\rVert 
$$

We will also use a more general notion, following Allard \cite{a}, simply called a {\em $n$-varifold}, which is a Radon measure on the Grassmann bundle $Gr(n)$ of $n$ planes in ${\mathbb R}^N$. It is elementary that to each integer-rectifiable $n$-varifold $V$ there is an associated $n$-varifold $\tilde{V}$ given by
$$\tilde{V}(A)=\lVert V\rVert(\pi(A))$$
where $\pi:Gr(n)\rightarrow {\mathbb R}^N$ is the projection map. The first variation formula for varifolds in Allard's sense is:
$$
\d V (X) = \int_{Gr(n)} \div_S X dV(x, S),
$$
We also note that each $n$-varifold induces a Radon measure $\lVert V\rVert=\pi_*V$ on ${\mathbb R}^N$, and a probability measure $V_x$ on each fiber $Gr(n)_x$, so that
$$\int_{Gr(n)} \psi(x,S)\ dV(x,S)=\int_{\R^N} \int_{Gr(n)_x} \psi(x,S)\ dV_x(S)\ d\lVert V\rVert_x$$

We say that $V$ has {\it locally bounded first variation in U} if for each $W$ with ${\overline W} \subset U$ there is a constant $c$ such that for all $C^1$ vector fields $X$ in $N$ with $\supp \lvert X\rvert \subset W$:
$$
\lvert \d V(X) \rvert \leq c \sup_U |X|.
$$
Define the total variation measure of $V$ on $U$ by:
$$
\lVert \d V\rVert(W) = \sup_{ \{ X: \lvert X\rvert \leq 1, \supp \lvert X\rvert \subset W \}} \lvert \d V(X)\rvert
$$
for any open $W$ with ${\overline W} \subset U$. Then $\lVert \d V \rVert$ being a Radon measure on $U$ is equivalent to $V$ having {\it locally bounded first variation in U}. Integer-rectifiable varifolds with locally bounded first variation satisfy a compactness result. For the proof see \cite{s} or \cite{a}.

\begin{thm}\label{thm:compactnessvarifolds}
Suppose $\{ V_i \}$ is a sequence of integer-rectifiable varifolds in $U$ that have locally bounded first variation in $U$ and suppose
$$
\sup_{i \geq 1} ( \lVert V_i \rVert(W) + \lVert\d V_i\rVert(W))
$$
is bounded for all $W$ with ${\overline W} \subset U$. Then there is a subsequence $\{ V_{i_j} \}$ and an integer-rectifiable varifold $V$ of 
locally of bounded first variation in $U$ such that $V_{i_j} \to V$ in the sense of Radon measures on $Gr(n)$. Moreover,  for all $W$ with ${\overline W} \subset U$,
$$
\lVert \d V \rVert(W) \leq \liminf_{j \to \infty} ||\d V_{i_j} ||(W)).
$$
\end{thm}

\bigskip

Let $V$ be an integer-rectifiable varifold with locally bounded first variation. Write
\begin{equation}
\d V (X) = \int_{Gr(n)} \div_S X dV(x, S) \equiv - \int \nu \cdot X d \lVert \d V \rVert,
\end{equation}
where $\nu$ is $\lVert \d V \rVert$-measurable with $| \nu| = 1$ $\lVert \d V \rVert$ a.e. 
Using the differentiation theory of Radon measures, the function
\begin{equation}
D_{\lVert V \rVert} \lVert \d V \rVert(x) = \lim_{\rho \to 0} \frac{\lVert \d V \rVert(B_\rho(x))}{\lVert V \rVert(B_\rho(x))}
\end{equation}
exists $\lVert V \rVert$ almost everywhere and is $\lVert V \rVert$ measurable. For any Borel set $W \subset U$,
\begin{equation}
\lVert \d V \rVert(W) = \int_W D_{\lVert V \rVert} \lVert \d V \rVert d\lVert V \rVert + \lVert \d V \rVert_{sing}(W),
\end{equation}
where 
$$
\lVert \d V \rVert_{sing} = \lVert \d V \rVert \;\; \lfloor \;\; Z
$$
where $Z$ is a Borel set of $\lVert V \rVert$-measure zero. Set $H(x) = D_{\lVert V \rVert} \lVert \d V \rVert(x) \nu(x)$ and call $H(V; x)$ the {\it generalized mean curvature} of $V$. $Z$ is called the {\it generalized boundary} of $V$. See \cite{s} for more details. It follows that we can write:
\begin{eqnarray}
\d V (X) &=& \int_{Gr(n)(U)} \div_S X dV(x, S) \\
\label{equ:HandSing}
&=& -\int_U H(V; \cdot) \cdot X d\lVert V \rVert + \int_U \nu \cdot X d\lVert \d V \rVert_{sing}.
\end{eqnarray}

Continuing to assume that $V$ is an integer-rectifiable varifold with locally bounded first variation following Allard \cite{a} we impose an additional condition. Let $p > 1$ and $q$ be its H\"older conjugate. Let $\a > 0$ be a constant. 
\begin{eqnarray}
\label{equ:Allard-bound}
\d V (X) \leq \a \lVert X\rVert_{L^q(V)}  \mbox{ for all smooth vector fields} \; X \; \mbox{with compact support}.
\end{eqnarray}

The following Proposition is essentially due to Allard \cite{a}, \S8.1.

\begin{prop}
\label{prop:no-sing}
If $V$ is an integer-rectifiable varifold with locally bounded first variation satisfying (\ref{equ:Allard-bound}) for $p > 1$ then $\lVert \d V\rVert_{sing} = 0$, $H(V; \cdot) \in L^p(V)$ with $\lVert H(V; \cdot)\rVert_{L^p(V)} \leq \a$  and for any compactly-supported $C^1$ test vector field $X$,
\begin{equation}
\d V (X) =  -\int H(V; \cdot) \cdot X d\lVert V \rVert.
\end{equation}
\end{prop}

\begin{proof}
Recall (\ref{equ:HandSing}) 
$$
\lVert \d V \rVert_{sing} = \lVert \d V \rVert \;\; \lfloor \;\; Z
$$
where $Z$ is a Borel set of $\lVert V \rVert$-measure zero. Let $W_\e$ be an open set with $\lVert V\rVert(W_\e) < \e$ and  $Z \subset W_\e$. Let $\chi_\e$ be the function that is identically one on $U \setminus W_\e$ and zero on $W_\e$. Set $X_\e = \chi_\e X$. Then $|| X_\e ||_{L^q(V)} \leq \lVert X\rVert _{L^q(V)}$ and $||X_\e||_{L^q(V)} \to || X ||_{L^q(V)}$ as $\e \to 0$. We have:
$$
\int_U H(V; \cdot) \cdot X_\e d\lVert V \rVert \to \int H(V; \cdot) \cdot X d\lVert V \rVert.
$$
as  $\e \to 0$ and
$$
\int_U \nu \cdot X_\e d\lVert \d V \rVert_{sing} = 0
$$
for all $\e > 0$. Therefore, letting $\e \to 0$, we have
\begin{equation}
\label{equ:Allard-ineq}
\int_U H(V; \cdot) \cdot X d\lVert V \rVert  \leq \a \lVert X\rVert _{L^q(V)}.
\end{equation}
It follows from (\ref{equ:Allard-ineq}) that $H(V; \cdot) \in L^p(V)$ and $||H(V; \cdot)||_{L^p(V)} \leq \a$.  \\

For any $i\in{\mathbb N}$ there is $\nu_i$, a smooth vector field with compact support, so that $\lvert \nu_i\rvert\leq 1+\frac{1}{i}$ and $\lVert\delta V\rVert\left\{x\middle\vert \nu_i(x)-\nu(x)\rvert\geq \frac{1}{i}\right\}<\frac{1}{i}$. For each $j\in{\mathbb N}$, let $\chi_j$ be a smooth function so that $0\leq \chi_j\leq 1$, $\chi_j\equiv 1$ on $\supp\lVert \delta V\rVert_{sing}$, and $\int \chi_j d\lVert V\rVert<\frac{1}{j}$. Set $\overline{\nu}_j=\chi_j\nu_j$. Then we have
$$\lVert \overline{\nu}_j\rVert_{L^q(V)}\leq \left(1+\frac{1}{j}\right)\frac{1}{j^{\frac{1}{q}}}\rightarrow 0$$ so that 
$$\left\lvert\int H(V)\cdot \overline{\nu}_j d\lVert V\rVert\right\rvert \leq \alpha \left(1+\frac{1}{j}\right)\frac{1}{j^{\frac{1}{q}}}\rightarrow 0$$and
$$\delta V(\overline{\nu}_j)\leq \alpha \lVert \overline{\nu}_j\rVert_{L^q(V)}\rightarrow 0.$$
On the other hand, $$\delta V(\overline{\nu}_j)=-\int H\cdot \overline{\nu_j} d\lVert V\rVert + \int \nu\cdot \overline{\nu}_j d\lVert \delta V\rVert_{sing}\rightarrow \lVert \delta V\rVert_{sing}$$

Therefore  $\lVert \d V\rVert_{sing} = 0$. The result follows.
\end{proof}

For $V$ is an integer-rectifiable varifold with locally bounded first variation, Menne \cite{m2} defines condition $(H_p)$ by:
\begin{equation}\label{eq:Hp}
\begin{cases}
H(V; \cdot) \in L^p(V).\\
\d V (X) =  -\int_U H(V; \cdot) \cdot X\ d\lVert V\rVert  \\
\, \, \, \, \, \, \,  \mbox{ for all smooth vector fields} \; X \; \mbox{with compact support}.
\end{cases}
\end{equation}

Proposition \ref{prop:no-sing} shows for $p>1$, Allard's condition (\ref{equ:Allard-bound}) implies $(H_p)$; the converse implication is immediate. The advantage of Allard's condition is the following compactness result. We will refer to the following result as Allard's compactness theorem.

\begin{thm}\label{thm:compact-Hp-varifolds}
Suppose $\{ V_i \}$ is a sequence of integer-rectifiable varifolds in $U$ that have locally bounded first variation in $U$ and satisfy $(H_p)$ with $p > 1$ and $\lVert H(V_i)\rVert_{L^p(V_i)}\leq \alpha$ for some $\alpha$ independent of $i$. Suppose
$$
\sup_{i \geq 1} ( \lVert V_i \rVert(W) + \lVert\d V_i\rVert(W))
$$
is bounded for all $W$ with ${\overline W} \subset U$. Then there is a subsequence of $\{ V_{i_j} \}$ and an integer-rectifiable varifold $V$ of 
locally of bounded first variation in $U$ also satisfying $(H_p)$ such that $V_{i_j} \to V$ in the sense of Radon measures on $Gr(n)$. 
\end{thm}

\begin{proof}
Using Theorem \ref{thm:compactnessvarifolds}  there is a subsequence $\{ V_{i_j} \}$ and an integer-rectifiable varifold $V$ of 
locally of bounded first variation in $U$ such that $V_{i_j} \to V$ in the sense of Radon measures on $Gr(n)$. 

Given any smooth test vector field $X$, observe that 
$$\delta V_{i_j}(X)=\int_{Gr(n)(U)}\div_S X\ dV_{i_j}\rightarrow \int_{Gr(n)(U)}\div_S X\ dV=\delta V(X)$$
and $$\lVert X\rVert_{L^q(V_{i_j})}\rightarrow \lVert X\rVert_{L^q(V)}.$$Thus condition \eqref{equ:Allard-bound} holds for $V$. By Proposition \ref{prop:no-sing}, the theorem follows.

\end{proof}

\bigskip

The  integer-rectifiable varifolds that are locally of bounded first variation and that satisfy $(H_p)$ with $p=2$ are of particular importance in this paper, in part because of the previous compactness result and in part because of the following partial regularity results proved by Menne \cite{m2}.

\begin{thm}
\label{thm:weak-regularity-M2}
Suppose $U$ is an open subset of $\R^{N}$ and $V$ is an integer rectifiable $n$-varifold with locally bounded first variation (equivalently, with $\lVert \d V \rVert$ a Radon measure). Then there exists a countable collection $C$ of $n$-dimensional $C^2$ submanifolds of $\R^{N}$  such that $\lVert V\rVert(U \setminus \cup C) = 0$ and each member $M \in C$ satisfies:
$$
T_zV=T_zM
$$
and
$$
H(V; z) = H(M; z),
$$
for $\lVert V\rVert$ almost all $z \in U \cap M$.
\end{thm}

\begin{thm}
\label{thm:Menne-tilt-excess}
Suppose $U$ is an open set in $\R^{N}$, $1 \leq p \leq \infty$ and $V$ is an integer-rectifiable $n$-varifold  satisfying $(H_p)$. 
Let $R(z) \in \operatorname{Hom}( \R^{N}, \R^{N})$ be the orthogonal projection to $T_zV$. If $n = 1,2$ and $p > 1$ or $n > 2$ and $p \geq \frac{2n}{n +2}$ then for $\lVert V\rVert $ a.e. $a$:
$$
\lim_{r \to 0 } \frac{1}{\lVert V\rVert ({B_r(a)})} \int_{B_r(a)} \Big( \frac{|R(z) - R(a) - R(a)(z-a) \cdot DR(a)|}{|z-a|} \Big)^2 d\lVert V\rVert_z = 0.
$$
\end{thm}

In particular the theorem applies to $p=2$; we will henceforth make use almost exclusively of the $p=2$ case.

\begin{defn}
We will refer to $a$ satisfying the conclusions of Theorems \ref{thm:weak-regularity-M2} and \ref{thm:Menne-tilt-excess} as {\em points of $C^2$ rectifiability}.
\end{defn}

Theorems \ref{thm:weak-regularity-M2} and \ref{thm:Menne-tilt-excess} depend in turn on the following tilt-excess decay estimate proved by Brakke \cite{b} and Menne \cite{m1}:

\begin{thm}
\label{thm:tilt-excess-B-M}
Suppose $U$ is an open subset of $\R^{N}$ and $V$ is an integer-rectifiable $n$-varifold with locally bounded first variation (equivalently, with $\lVert \d V \rVert$ a Radon measure). For $\lVert V\rVert$-a.e. $x \in U$, the approximate tangent space $T_xV \in Gr(n)$ exists and
$$
\lim_{r \to 0+} r^{-1/2-n/2} \left( \int_{\pi^{-1}(B_r(x))} \lvert S -T_xV\rvert^2 d\tilde{V} \right)^{1/2} = 0
$$
\end{thm}

\bigskip

In fact we will use these notions as they apply to lagrangian subspaces. Hence we assume $N=2n$. 

We say that the integer rectifiable $n$-varifold $V$ is {\em a lagrangian integer rectifiable varifold} if for ${\cal H}^n$ a.e. $x$ the approximate tangent plane  $T_x V$ is a lagrangian $n$-plane. A {\em lagrangian varifold} is an $n$-varifold which has support contained in the Grassmann bundle $LGr\subset Gr(n)$ of lagrangian $n$-planes. Because the Grassmannian of Lagrangian planes at $x$, $LGr_x$, is closed as a subset of all the Grassmannian of all $n$-planes $Gr(n)_x$, and $LGr$ is closed in $Gr(n)$, the compactness Theorem \ref{thm:compactnessvarifolds} yields a compactness theorem for lagrangian integer-rectifiable varifolds as well.

\bigskip

We remark also that we may replace ${\mathbb R}^{2n}$ with a Calabi-Yau manifold $(N,\omega)$ in a straightforward way to obtain versions of all Theorems in this section which apply to lagrangian integer-rectifiable varifolds and lagrangian varifolds in $(N,\omega)$, though for exposition we will remain mostly in the Euclidean context.

\bigskip

For reference, we include the following table of terminology for the various classes of varifolds we will consider in the remainder of this paper, ordered by inclusion:

\begin{itemize}
	\item {\em lagrangian varifold}, a Radon measure on the bundle $LGr$.
	\item {\em lagrangian integer varifold} or {\em lagrangian integer-rectifiable varifold}, an integer-rectifiable varifold whose tangent planes are almost all lagrangian.
	\item {\em varifold with Maslov index zero}, a lagrangian integer varifold with a weakly-differentiable lift of the lagrangian angle, as described in \S \ref{sec:maslovindexzero}.
	\item {\em varifold with Maslov index zero and mean curvature in $L^p$}, a varifold with Maslov index zero which in addition satisfies the hypotheses of Theorem \ref{thm:Menne-tilt-excess}. We will choose $p=2$ for definiteness.
\end{itemize}

Because we will be interested in applications to lagrangian homology classes in a closed Calabi-Yau manifold, we will henceforth assume that all varifolds have finite mass and compact support.

\section{Mollification along a Varifold}\label{sec:mollification}

In what follows, inspired by Brakke's construction, we use mollification, but rather than mollify the varifold $V$ as Brakke does, we will mollify functions with respect to the varifold measure $\lVert V\rVert$. Our exposition follows that of Evans \cite{e}. For exposition we state results in this section for varifolds in ${\mathbb R}^{2n}$; the adaptation to the manifold case is straightforward and discussed in \S\ref{sec:manifoldcase}.\\

Consider the function $g:{\mathbb R}_+\rightarrow{\mathbb R}_+$ given by $$g(t)=\begin{cases}\exp\left(\frac{1}{t^2-1}\right)&t\in[0,1)\\0&t\in[1,\infty)\end{cases}$$ Observe that $0\leq g(t)\leq 1$ for all $t$.

\begin{defn}
	For any $\e>0$, the mollifier at scale $\e$ is $\phi_\e\in C^\infty_c({\mathbb R}^{2n})$ given by $$\phi_\e(x)=C\e^{-n}g(\e^{-1}\lvert x\rvert)$$where $C$ is chosen so that $\int_{{\mathbb R}^n}\phi_\e(x)dx=1$.
\end{defn}

Observe that the support of $\phi_\e$ is $B(0,\e)$.

\begin{defn}
	If $f$ is $\lVert V\rVert$-integrable, its $\e$-mollification with respect to $V$, $f_\e$, is
		$$f_\e(x)=\frac{\int f(y)\phi_\e(x-y)d\lVert V\rVert_y}{\int \phi_\e(x-z)d\lVert V\rVert_z+ \e \lVert V\rVert}.$$
\end{defn}\bigskip

\begin{rem}
We choose the mollifying function $\phi$ to have compact support because this choice is easily adapted to the manifold case.

To ensure that pointwise $f_\e(x) \to f(x)$ as $\e \to 0$ we need to normalize the mollification. In the case of mollification with respect to Lebesgue measure on ${\mathbb R}^n$, the normalization would be $\e^{-n}$, but for us, this quantity must depend on $x$. The most natural choice of normalization is $\left( \int \phi_\e(x-z)  d\lVert V\rVert_z \right)^{-1}$. However because $\phi$ is compactly supported, the normalization vanishes for $x$ sufficiently far from $\supp V$ and this may introduce singularities into $f_\e$. We use the additional term $\e \lVert V\rVert$ to ensure that the denominator in $f_\e$ does not vanish, and thus that $f_\e(x)$ is smooth for all $x\in{\mathbb R}^{2n}$.

\end{rem}

Mollification can be considered as a linear map $L^p(V) \to C^\infty_c(\R^{2n})$; that is, for $f \in L^p(V) $ define 
\begin{equation}
\label{equ:define-moll-operator}
L_\e(f) = \phi_\e \star f = f_\e.
\end{equation}
Then $L_\e$ is a smoothing operator.

\bigskip

The theory of mollification with respect to a varifold is largely similar to the standard theory of mollification with respect to Lebesgue measure on ${\mathbb R}^n$. For completeness we state the following basic facts about mollification. Our exposition follows that of Evans \cite{e}, mutatis mutandis.

\begin{lem}
	$f_\e\in C^\infty_c({\mathbb R}^{2n})$.
\end{lem}

\begin{proof}
Clear.
\end{proof}

For classical mollification with respect to Lebesgue measure (see for example \cite{e}), because $\int  \phi_\e (x-y) dy=1$,  trivially 
	\begin{equation}
		\int D\phi_\e (x-y) dy=0
	\end{equation}
In our case we have the weaker statement:

\begin{lem}\label{lem:Dphiepsilonbounded}
For  $x$ in the support of $\lVert V\rVert$ where $V$ has an approximate tangent plane, 
$$\int D\phi_\e (x-y)d\lVert V\rVert_y$$ 
is bounded independent of $\e$.
\end{lem}

\begin{proof}
	We compute $D\phi_\e(x-y)$:
		\begin{equation}\label{eqn:Dphiepsilon}
			D\phi_\e(x) = C \e^{-n}\e^{-1}g'(\e^{-1}\lvert x\rvert) \frac{x}{\lvert x\rvert}
		\end{equation} 
	Therefore setting $\rho=\e^{\frac{n+1}{n}}$, 
		\begin{equation}
			\begin{aligned}
				\int D_x\phi_\e(x-y)d\lVert V\rVert_y&=C\rho^{-n}\int g'\left(\e^{-1}\lvert x-y\rvert\right)\frac{x-y}{\lvert x-y\rvert}d\lVert V\rVert_y\\
				&=C\int g'\left(\rho^{\frac{1}{n+1}}\lvert t\rvert\right)\frac{t}{\lvert t\rvert}d\lVert \rho^{-1}(V-x)\rVert t\\
			\end{aligned}
		\end{equation}
	As $\e\rightarrow 0$, $\rho\rightarrow 0$; the integration converges to integration on $T_xV$. Moreover, the integrand is uniformly bounded. 
\end{proof}

We remark that this lemma does not assert that $D\phi_\e$ is uniformly bounded in $L^1(V)$. 

\begin{lem}\label{lem:mollifiertomultiplicity}
	If $x$ is a point where $V$ has an approximate tangent plane, $$\int \phi_\e(x-y)d\lVert V\rVert_y\rightarrow \theta(x)$$ as $\e\rightarrow 0$, where $\theta(x)$ is the multiplicity of $V$ at $x$.
\end{lem}
\begin{proof}
	\begin{equation}
		\begin{aligned}
			\int \phi_\e(x-y)d\lVert V\rVert_y&=C\e^{-n}\int g\left(\e^{-1}\lvert x-y\rvert\right)d\lVert V\rVert_x\\
			&\rightarrow \theta(x) C\int_{T_xV}g(\lvert t\rvert)dt=\theta(x)
		\end{aligned}
	\end{equation}
	where we have used the fact that $C\int_{{\mathbb R}^n}g(\lvert t\rvert )dt=1$.
\end{proof}

We  have the following explicit bounds on the derivatives of $f_\e$, which degenerate as $\e\rightarrow 0$ but will nonetheless be useful in \S \ref{sec:epsilonflows}.\bigskip

\begin{lem}
	There is a universal constant $C$ so that for any $f\in L^{\infty}(V)$,
		\begin{equation*}
			\lvert Df_\e(x)\rvert\leq C\e^{-n-2}\lVert f\rVert_{L^\infty(V)}
		\end{equation*}
	for each $x\in {\mathbb R}^{2n}$.
\end{lem}

\begin{proof}
	First observe that 
		\begin{equation}\label{eqn:Linftybound}
			\begin{aligned}
			\lvert f_\e(x)\rvert&=\left\lvert\frac{\int f(y) \phi_\e(x-y) d\lVert V\rVert_y}{\int \phi_\e(x-z) d\lVert V\rVert_z+\e \lVert V\rVert}\right\rvert\\
			&\leq \lVert f\rVert_{L^\infty(V)}\frac{\int \phi_\e(x-y) d\lVert V\rVert_y}{\int \phi_\e(x-z) d\lVert V\rVert_z+ \e \lVert V\rVert}\leq \lVert f\rVert_{L^\infty(V)}
			\end{aligned}
		\end{equation}
	
	Observe that because
	\begin{equation}
		\lvert D\phi_\e(x)\rvert = \left\lvert C\e^{-n}\e^{-1} g'(\e^{-1}\lvert x\rvert)\frac{x}{\lvert x\rvert}\right\rvert
	\end{equation}
	we have $\lvert D\phi_\e(x)\rvert\leq C(1)\e^{-n-1}$, where $C(1)$ is the maximum of $C\lvert g'\rvert$ on $[0,1)$. 
		
	Now compute $Df_\e(x)$:
		\begin{equation}
			\begin{aligned}
			\lvert Df_\e(x)\rvert&=\left\rvert\frac{\int f(y)D\phi_\e(x-y)d\lVert V\rVert_y}{\int\phi_\e(x-z) d\lVert V\rVert_z + \e \lVert V\rVert}- f_\e(x)\frac{\int D\phi_\e(x-y)d\lVert V\rVert_y}{\int \phi_\e(x-z)d\lVert V\rVert_z+ \e \lVert V\rVert}\right\rvert\\
			&\leq2\lVert f\rVert_{L^\infty(V)} \frac{\int \lvert D\phi_\e(x-y)\rvert d\lVert V\rVert_y}{{\int \phi_\e(x-z)d\lVert V\rVert_z + \e \lVert V\rVert}} \\
			&\leq 2\lVert f\rVert_{L^\infty(V)} \frac{C(1)\e^{-n-1}\lVert V\rVert}{ \e \lVert V\rVert} 
			\end{aligned}
		\end{equation}
	
	\end{proof}

\begin{lem}\label{lem:derivativesofbetaepsilon}
	For any $k$, there is a universal constant $C=C(k)$ so that for any $f\in L^{\infty}(V)$, we have the estimate		\begin{equation*}
			\lvert D^kf_\e(x)\rvert\leq C\e^{-k-n-1}\lVert f\rVert_{L^\infty(V)}
		\end{equation*}	
	for each $x\in {\mathbb R}^{2n}$.
\end{lem}

\begin{proof}
The proof for the higher derivatives of $f_\e$ is similar to the proof for the first derivative and is left to the reader.
\end{proof}

\bigskip

\begin{thm}
\label{thm:strong-conv-Lp}
For $1 \leq p \leq \infty$,	if $f \in L^p(V)$  then $f_\e\rightarrow f$ strongly in $L^p(V)$.
\end{thm}

\begin{lem}\label{lem:continuousmollifier}
	If $f$ is continuous, then $f_\e(x)\rightarrow f(x)$ for $\lVert V\rVert$-a.e.~$x$.
\end{lem}

\begin{proof}
	Let $x\in \operatorname{supp}\lVert V\rVert$. Consider $\lvert f_\e(x)-f(x)\rvert$. We have
		\begin{equation}\label{eqn:contsmollifierbound}
			\begin{aligned}
			\lvert f_\e(x)-f(x)\rvert& =\left\lvert\frac{\int \phi_\e(x-y)f(y)d\lVert V\rVert_y}{\int \phi_\e(x-z) d\lVert V\rVert_z + \e \lVert V\rVert}-f(x)\right\lvert\\
			&=\left\lvert\frac{\int \phi_\e(x-y)(f(y)-f(x))d\lVert V\rVert_y}{\int \phi_\e(x-z) d\lVert V\rVert_z+ \e \lVert V\rVert}\right\lvert\\
			&\leq \frac{\int \phi_\e(x-y)\lvert f(y)-f(x)\rvert d\lVert V\rVert_y}{\int \phi_\e(x-z) d\lVert V\rVert_z+\e \lVert V\rVert}\\
			\end{aligned}
		\end{equation}
	 Since $f$ is continuous for any $\eta$ there is $\e$ so small that $\lvert f(x)-f(y)\rvert\leq \eta$ on $B(x,\e)$. For such $\e$,  \eqref{eqn:contsmollifierbound} is estimated by 
		\begin{equation}
			\begin{aligned}
			\frac{\int\phi_\e(x-y)\lvert f(x)-f(y)\rvert \eta d\lVert V\rVert_y}{\int \phi_\e(x-z) d\lVert V\rVert_z+\e \lVert V\rVert}\leq \frac{\int\phi_\e(x-y)\eta  d\lVert V\rVert_y}{\int \phi_\e(x-z) d\lVert V\rVert_z+\e \lVert V\rVert} \leq \eta
			\end{aligned}
		\end{equation}
\end{proof}

\begin{cor}
	Suppose $V$ has compact support. If $f$ is continuous, then $f_\e\rightarrow f$ strongly in $L^p(V)$ for each $p\geq 1$.
\end{cor}
\begin{proof}
	Since $f$ is continuous and $V$ has compact support, $f\in L^\infty(V)$. So by \eqref{eqn:Linftybound}, $f_\e^p$ is dominated by $f^p$, so the Dominated Convergence Theorem together with Lemma \ref{lem:continuousmollifier} gives $f_\e\rightarrow f$ in $L^p$.
\end{proof}

Recall that $L_\e: L^p(V)\rightarrow C^\infty({\mathbb R}^{2n})$ is the mollification operator, $L_\e(f) = f_\e$.

\begin{prop}
\label{prop:bound-on-moll-op}
Consider $L_\e:L^p(V)\rightarrow L^p(V)$ Then there is a constant $C$, independent of $\e$, $p$, and $f$ so that $\lVert L_\e f\rVert_{L^p(V)}\leq C\lVert f\rVert_{L^p(V)}$
\end{prop}

\begin{proof}
	Consider some $x\in \operatorname{supp}(\lVert V\rVert)$. We have
		\begin{equation}
			\begin{aligned}
			\lvert f_\e(x)\rvert^p &=\left\lvert \frac{\int f(y)\phi_\e(x-y)d\lVert V\rVert_y}{\int \phi_\e(x-z) d\lVert V\rVert_z+\e \lVert V\rVert}\right\rvert^p\\
			&\leq \left( \frac{\int \lvert f(y)\rvert \phi_\e(x-y)d\lVert V\rVert_y}{\int \phi_\e(x-z) d\lVert V\rVert_z+\e \lVert V\rVert}\right)^p\\
			&=\left( \frac{\int \lvert f(y)\rvert \phi_\e^{\frac{1}{p}}(x-y)\phi_\e^{1-\frac{1}{p}}(x-y) d\lVert V\rVert_y}{\int \phi_\e(x-z) d\lVert V\rVert_z+\e \lVert V\rVert}\right)^p\\
			&\leq \left( \frac{\left(\int \lvert f(y)\rvert^p \phi_\e(x-y)d\lVert V\rVert_y \right)^{\frac{1}{p}} \left(\int \phi_\e(x-y) d\lVert V\rVert_y\right)^{\frac{p-1}{p}}} {\int \phi_\e(x-z) d\lVert V\rVert_z+\e \lVert V\rVert}\right)^p\\
			&\leq \frac{\int \lvert f(y)\rvert^p\phi_\e(x-y)d\lVert V\rVert_y}{\int \phi_\e(x-z) d\lVert V\rVert_z+\e \lVert V\rVert}
			\end{aligned}
		\end{equation}
	Therefore
		\begin{equation}\label{eqn:fepsLp}
			\int \lvert f_\e(x)\rvert^p d\lVert V\rVert_x\leq \int \int \lvert f(y)\rvert^p \frac{\phi_\e(x-y)}{\int \phi_\e(x-z) d\lVert V\rVert_z+\e \lVert V\rVert}d\lVert V\rVert_x d\lVert V\rVert_y
		\end{equation}
		
	 Observe that as $\e\rightarrow 0$, the measure $$\int\frac{\phi_\e(x-y)}{\int \phi_\e(x-z) d\lVert V\rVert_z+\e \lVert V\rVert}d\lVert V\rVert_x d\lVert V\rVert_y$$ tends to $\lVert V\rVert\delta_y$; thus the right-hand side of \eqref{eqn:fepsLp} tends to $\lVert f\rVert_{L^p(V)}^p$. On the other hand at $\e=1$ the quantity $\displaystyle\sup_y \frac{\int \phi_\e(x-y)d\lVert V\rVert_x }{\int \phi_\e(x-z) d\lVert V\rVert_z+\e \lVert V\rVert}$ is finite and independent of both $f$ and $p$. The claim follows.
\end{proof}

\begin{proof}[Proof of Theorem \ref{thm:strong-conv-Lp}]
In the case $p<\infty$, given any $f\in L^p(V)$, since $\lVert V\rVert$ is a Radon measure, there is a sequence $f_j$ of continuous functions with $f_j\rightarrow f$ in $L^p$ (see, e.g., \cite{eg}). We have by Proposition \ref{prop:bound-on-moll-op}
	\begin{equation}
		\begin{aligned}
			\lVert f_\e- f\rVert_{L^p(V)} &\leq \lVert f_\e - \left(f_j\right)_\e\rVert_{L^p(V)} + \lVert \left(f_j\right)_\e-f_j\rVert_{L^p(V)} + \lVert f_j-f\rVert_{L^p(V)}\\
			&=\lVert L_\e(f-f_j)\rVert_{L^p(V)} +\lVert \left(f_j\right)_\e-f_j\rVert_{L^p(V)} + \lVert f_j-f\rVert_{L^p(V)}\\
			&\leq C\lVert f-f_j\rVert_{L^p(V)} + \lVert \left(f_j\right)_\e-f_j\rVert_{L^p(V)} + \lVert f_j-f\rVert_{L^p(V)}
		\end{aligned}
	\end{equation}
Now for any $\eta>0$, since $f_j\rightarrow f$ in $L^p(V)$, the first and third terms can each be made smaller than $\frac{\eta}{3}$; since each $f_j$ is continuous, Lemma \ref{lem:continuousmollifier} allows us to make the second term smaller than $\frac{\eta}{3}$.

For the case $p=\infty$, we may as well assume the support of each function involved is compact, apply the $L^p$ case, and let $p$ tend to $\infty$. \end{proof}

For use later we observe that mollification with respect to $V$ gives a proof of the following proposition.
\begin{prop}\label{prop:mollific}
	Let $f\in L^p(V)$. Then there is a sequence $f_j\in C^\infty_0({\mathbb R}^{2n})$ with $f_j\rightarrow f$ strongly in $L^p(V)$.
\end{prop}
\begin{proof}
	Set $f_j=f_{\frac{1}{j}}$.
\end{proof}

\bigskip

\section{Varifolds with Maslov Index Zero}\label{sec:weak-derivative}

We wish to extend the notion of vanishing Maslov class to varifolds.  We do this by observing that in the smooth case, the vanishing of the Maslov class is equivalent to the existence of a continuous lift of the lagrangian angle. In the varifold setting, the idea of a `continuous lift' of the lagrangian angle is not sensible, because the varifold itself is not a continuous object. We therefore consider varifolds which admit a weakly differentiable lift of the lagrangian angle.

In this section, we will make precise a notion of weak derivative for functions on an integer-rectifiable varifold and use it to define a class of lagrangian integer rectifiable varifolds we call {\em Maslov index zero}. We use mollification as described in \S \ref{sec:mollification} and Menne's second-order rectifiability result to establish a formula relating the weak derivative of the lagrangian angle to the generalized mean curvature.  Finally we show show that the class of lagrangian varifolds with Maslov index zero is compact in a reasonable sense.\\

To compute the weak derivative of a function on a varifold we will use the first variation formula. Formally the computations are possible on, for example, varifolds with bounded first variation. However the generalized boundary or, equivalently, the singular part of the variation measure $|| \d V||$ will come into these computations. To simplify the computations we will formulate the weak derivative on varifolds satisfying the condition $( H_p )$, $p >1$ \eqref{eq:Hp}. This condition also satisfies a compactness result, Theorem \ref{thm:compact-Hp-varifolds}, that is essential for the applications of this paper.  \\

Throughout, let $V$ be a rectifiable $n$-varifold in the Riemannian $2n$-manifold $N$, with induced Radon measure $\lVert V\rVert$ that satisfies $( H_p )$, $p >1$.  If $\rho$ is a Lipschitz function in a neighborhood of the support of $\lVert V\rVert$ then the tangential component of the derivative of $\rho$ is well-defined $\lVert V\rVert$-a.e.~and will be denoted $\n \rho$.

\subsection{Weak derivatives along a rectifiable varifold} 

\begin{defn}
Suppose $V$ is a rectifiable $n$-varifold satisfying $(H_p)$ for $p>1$. Let $f$ be a $\lVert V\rVert$-integrable function. We say that $f$ has a {\em weak derivative} $F$, where $F$ has values in $TN$, if $F$ is $\lVert V\rVert $-integrable and for any $\varphi \in C^{\infty}_0(N)$:
$$
\int \n \var f d \lVert V\rVert  = -\int \var F d \lVert V\rVert .
$$
\end{defn}

\begin{prop}
	Weak derivatives, if they exist, are unique.
\end{prop}
\begin{proof}
	Given $F_1,F_2$ weak derivatives for $f$, we have for any test function $\var$,
		\begin{equation}
			\begin{aligned}
				\int \var F_1 d\lVert V\rVert = -\int f \n\var d\lVert V\rVert = \int \var F_2d \lVert V\rVert
			\end{aligned}
		\end{equation}
	and since this holds for all test functions $\var$, we have $F_1=F_2$ $\lVert V\rVert$-almost everywhere.
\end{proof}

To illustrate the notion of weak derivative we consider the special case in which $U \subset M$ is open, $U \cap V \neq \emptyset$ and $g$ is a Lipschitz function on $U$ such that on $U \cap V$, $g_{|_V} = f$ (i.e., $f$ can be extended to a Lipschitz function on $U$).

\begin{prop}
\label{prop:weak-derivative-Lipschitz}
Let $V$ be a rectifiable $n$-varifold with induced Radon measure $\lVert V\rVert$ that satisfies $( H_p )$, $p >1$.
On $U \cap V$, $\n f$ is well-defined. Moreover, $f$ has a weak derivative $F$ on $U \cap V$, given by $F = \n f + f H$, where $H$ is the mean curvature of $V$.
\end{prop}  

\begin{proof}
Let $\rho$ be a compactly supported smooth test function with $\supp \rho$ contained in an open coordinate neighborhood $W$.  Let $\{e_1,\ldots,e_{2n}\}$ be a covariant constant frame in $W$. Observe that $f\rho e_i$ is an admissible variation, so we  can apply the first variation formula. We compute
	\begin{equation}
		\begin{aligned}
			-\int f\nabla\rho\cdot e_i\ d\lVert V\rVert&=-\int\left(\div(f\rho e_i)-\rho \nabla f\cdot e_i-f\rho \div(e_i)\right)\ d\lVert V\rVert\\
			&=\int f\rho H\cdot e_i\ d\lVert V\rVert+\int \rho\nabla f\cdot e_i\ d\lVert V\rVert
		\end{aligned}
	\end{equation}
	where we have used the Leibniz rule for the divergence operator and the fact that $e_i$ is covariant constant.
\end{proof}

\bigskip

\begin{ex} Let $\Sigma$ be a smooth embedded submanifold without boundary, and $U$ an open set so that $\Sigma\cap U$ and $\Sigma\cap U^c$ have nonempty interior in $\Sigma$. Let $V$ be the varifold arising from $\Sigma$. Then the characteristic function
\begin{equation} f(x) =
\begin{cases}
1 & x\in U\\
0 & x\notin U
\end{cases}
\end{equation}
does not have a weak derivative on $V$. 
\end{ex}

The fact that there is no weak derivative for such $f$ is central to our purposes. We use the lack of a weak derivative to detect ``jump discontinuities".

\begin{rem}
	We have only defined the notion of weak derivative for varifolds satisfying $(H_p)$ for $p>1$. In particular the generalized boundary vanishes in this case. For varifolds with nonvanishing generalized boundary, the formula for a weak derivative would necessarily involve a generalized boundary term. Menne \cite{m3} \cite{m4} has studied function theory on varifolds in a general setting; these results are potentially important for generalizations of the present paper.
\end{rem}

The quantity $\nabla f+fH$ also appears in the  Sobolev inequality proved by Michael-Simon (\cite{ms}, Theorem 2.1) and Allard (\cite{a}, Theorem 7.3):
\begin{thm}\label{thm:MS}
	Let $V$ be an integer-rectifiable $m$-varifold in ${\mathbb R}^N$ that satisfies $( H_p )$, $p >1$. Then there is a universal constant $C(N)$ so that for any $u\in C^1_c({\mathbb R}^N)$,
		$$\lVert u\rVert_{L^{\frac{m}{m-1}}(V)}\leq C(N) \lVert \nabla u+uH\rVert_{L^1(V)},$$
\end{thm}

Using Proposition \ref{prop:weak-derivative-Lipschitz}, we may rephrase this result so that it looks more like the classical Sobolev inequality:
\begin{thm}\label{thm:MSrestate}
	Let $V$ be an integer-rectifiable $m$-varifold in ${\mathbb R}^N$ that satisfies $( H_p )$, $p >1$. Then there is a universal constant $C(N)$ so that for any $u\in C^1_c({\mathbb R}^N)$,
		$$\lVert u\rVert_{L^{\frac{m}{m-1}}(V)}\leq C(N) \lVert U\rVert_{L^1(V)},$$
	where $U$ is the weak derivative of $u$ with respect to $V$.
\end{thm}

\bigskip

\begin{prop}\label{divergenceformula}
	$F$ is a weak derivative for $f$ if and only if for each compactly-supported smooth test vector field $X$ we have
		\begin{equation*}
			\int f\div X d\lVert V\rVert=-\int F\cdot X d\lVert V\rVert
		\end{equation*}
\end{prop}
\begin{proof}
	Given a test vector field $X$, we can write $X=X^i E_i$ where $\{E_1,\ldots, E_{2n}\}$ are a covariant constant ambient orthonormal frame. We have:
		\begin{equation}
			\begin{aligned}
			\int f\div X d\lVert V\rVert&=\int f \div X^i E_i d\lVert V\rVert \\
			&= \int f\left(\n X^i\cdot E_i+X^i\div E_i\right)d\lVert V\rVert \\
			&= \int f\n X^i\cdot E_id\lVert V\rVert
			\end{aligned}
		\end{equation}
		because the $E_i$ are covariant constant.
		
	Now if $F$ is a weak derivative for $f$, we can rewrite the latter integral as
	\begin{equation}
	\begin{aligned}
		 \int f\n X^i\cdot E_id\lVert V\rVert&= \left(\int f\n X^i d\lVert V\rVert\right)\cdot E_i\\
		 &=-\left(\int F X^i d\lVert V\rVert \right)\cdot E_i=-\int F\cdot X d\lVert V\rVert 
	\end{aligned}
	\end{equation}
	where again we have used the fact that each $E_i$ is covariant constant.
	
	Conversely, given a test function $\var$, apply the formula $\int f\div X d\lVert V\rVert=-\int F\cdot X d\lVert V\rVert$ to the vector field $X=\var E$ for some covariant constant vector field $E$ to obtain $\int f \n\var d\lVert V\rVert = - \int \var Fd \lVert V\rVert$.
\end{proof}

\bigskip

Recall Menne's decomposition theorem for varifolds stated in Theorem \ref{thm:weak-regularity-M2}   : If $\lVert \d V \rVert$ is a Radon measure then there is a countable collection $\{ M_i \}$ such that:
$$
V = M_0 \cup \cup_{i=1}^\infty M_i,
$$
where $M_0$ has $\lVert V\rVert $-measure zero and for each $i =1,2,3 \dots$ there is a $C^2$ submanifold $N_i$ with $M_i \subset N_i$. Moreover, the generalized mean curvature $H$ of $V$ satisfies for a.e. $x \in M_i$
$$
H(x) = H^{N_i}(x),
$$
where $H^{N_i}$ is the mean curvature along $N_i$.

Given an ambient $C^2$ vector field $X$ we can decompose $X$ into tangential $X^{\top}$ and normal $X^{\perp}$ components along $N_i$. Both $X^{\top}$ and $X^{\perp}$ can be extended to $C^1$ ambient vector fields and therefore both $\div_{N_i} X^{\top}$ and $\div_{N_i} X^{\perp}$ are well-defined integrable functions. We define the decomposition of $X$ into tangential $X^{\top}$ and normal $X^{\perp}$ components along  $M_i$ using the decomposition along $N_i$ and define
$$
\div_{M_i} X^{\top} = \div_{N_i} X^{\top}
$$
$$
\div_{M_i} X^{\perp} = \div_{N_i} X^{\perp}
$$
We have, thus, defined the decomposition of $X$ into tangential $X^{\top}$ and normal $X^{\perp}$ components along $V$ $\lVert V\rVert $ a.e. and defined the integrable functions $\div_{V} X^{\top}$ and $\div_{V} X^{\perp}$ $\lVert V\rVert $ a.e.

The next result gives a partial characterization of the weak derivative along a varifold that satisfies $(H_p)$ for $p >1$.

\begin{thm}
\label{thm:weak-derivative-Menne}
Suppose $V$ is a $n$-varifold that satisfies $(H_p)$ for $p >1$. Suppose $f \in L^\infty( V )$ with weak derivative $F \in L^p( V)$. Then $F$ admits a decomposition 
$$
F = F^{\top} + fH,
$$ 
where $F^{\top} \in L^p( V)$ satisfies $F^{\top}(x) \in T_x V$ for $\lVert V\rVert $ a.e. $x$. Moreover,
$$
\int F^{\top} \cdot X d\lVert V\rVert  = - \int f \div_V X^{\top} d\lVert V\rVert .
$$
\end{thm}

\begin{proof}
First we claim that if $f, X\in C^1_c({\mathbb R}^{2n})$ then $$\int f\div_V(X^\top)d\lVert V\rVert = -\int \nabla f\cdot Xd\lVert V\rVert$$ To see this we compute the first variation of $\lVert V\rVert$ with respect to the vector field $fX$, using Menne to decompose $X=X^\top+X^\perp$ into $C^1$ vector fields on the $C^2$ pieces of some Menne decomposition and using the fact that pointwise (as in the $C^2 $ case) $\div_VX^\perp=-H\cdot X$,
		\begin{equation}
			\begin{aligned}
				\delta V(fX)&=\int \div_V(fX)d\lVert V\rVert=\int \nabla f\cdot X+f\div_V Xd\lVert V\rVert \\
				&= \int \nabla f\cdot X+ f\div_V X^\top+f\div_VX^\perp d\lVert V\rVert \\
				&=\int \nabla f\cdot X+ f\div_V X^\top-fH\cdot Xd\lVert V\rVert \\
			\end{aligned}
		\end{equation}
	On the other hand, by the definition of $H$, we have $\delta V(fX)=-\int H\cdot(fX)d\lVert V\rVert$.
	
Next, suppose that $\{ f_j \}$ is a sequence of $C^\infty_c({\mathbb R}^{2n})$ functions with $f_j\rightarrow f$ in $L^\infty$, constructed say using Proposition \ref{prop:mollific}. For any $X \in C^1_c({\mathbb R}^{2n})$,
		\begin{equation}
			\begin{aligned}
				\int \nabla f_j Xd\lVert V\rVert &= -\int f_j H\cdot X d\lVert V\rVert - \int f_j \div_V X d\lVert V\rVert
			\end{aligned}
		\end{equation}
	which converges (using Proposition \ref{prop:weak-derivative-Lipschitz}) to 
		\begin{equation}
			-\int fH\cdot Xd\lVert V\rVert - \int f\div_V Xd\lVert V\rVert=-\int f H\cdot Xd\lVert V\rVert +\int F\cdot Xd\lVert V\rVert
		\end{equation} 
Since $X$ is arbitrary, this means $\nabla f_j\rightharpoonup F-fH$ in $L^p(V)$. Moreover, $H$ is $\lVert V\rVert$-almost everywhere normal and $\nabla f_j$ are everywhere tangential, so $F-fH=F^\top$.

The claimed integration-by-parts formula for $F^\top$ follows from the fact that for each $j$, $$\int f_j\div_V(X^\top)d\lVert V\rVert = -\int \nabla f_j\cdot Xd\lVert V\rVert.$$

\end{proof}

\begin{cor}\label{powerrule}
Let  $V$be an $n$-varifold that satisfies $(H_p)$ for $p >1$.
	Suppose $f \in L^\infty(V)$ has a weak derivative $F\in L^p(V)$, and $g=f^k$ for some $k\neq0$. Then $g$ has weak derivative
	\begin{equation*}
		G=kf^{k-1}F-(k-1)f^kH
	\end{equation*}
\end{cor}
\begin{proof}
	Using Proposition \ref{prop:mollific}, let $f_j$ be a sequence of $C^\infty_0({\mathbb R}^{2n})$ functions with $f_j\rightarrow f$ in $L^\infty(V)$. Each $f_j$ satisfies the claimed formula by direct computation, that is, $g_j=f_j^k$ has weak derivative
		$$G_j=kf_j^{k-1}F_j-(k-1)f_j^kH$$
	We have $G_j^\top=f_j^{k-1}\nabla f_j$; as in the proof of the Theorem this goes to $kf^{k-1}F^\top$. Clearly  $G_j^\perp=f_j^kH\rightarrow f^kH$. So $$G_j\rightarrow kf^{k-1}(F-fH)+f^kH=kf^{k-1} F -(k-1)f^kH\text{ in }L^p(V).$$
	
	On the other hand, we have for any smooth test function $\varphi$
	$$\int g_j\nabla \varphi d\lVert V\rVert=-\int G_j\varphi d\lVert V\rVert,$$and taking the limit in $j$ we obtain that $G$ is a weak dervative for $g$.
\end{proof}

More generally, we have the following chain rule:
\begin{cor}
Let  $V$be an $n$-varifold that satisfies $(H_p)$ for $p > 1$.
	Let $h\in C^1(\mathbb{R})$. If $f \in L^\infty(V)$ has a weak derivative $F\in L^p(V)$, and $g=h\circ f$, then $g$ has weak derivative
	\begin{equation*}
		G=h'(f)F+\left(h(f)-fh'(f)\right)H
	\end{equation*}
\end{cor}
\begin{proof}
	As above, we observe that the claimed formula holds for smooth $f$. For general $f\in L^\infty(V)$, we let $f_j$ be a smooth approximating sequence with $f_j\rightarrow f$ in $L^\infty(V)$. Consider $g_j=h\circ f_j$ and its weak derivative
	$$G_j=h'(f_j)F_j+\left(h(f_j)-f_j h'(f_j)\right)H.$$
	Since $f\in L^\infty(V)$, we may assume that each $\lVert f_j\rVert_{L^\infty(V)}\leq 2\lVert f\rVert_{L^\infty(V)}$; thus because $h\in C^1$ we see that $h'(f_j)$ and $ h(f_j)$ are uniformly bounded in $L^\infty(V)$. Moreover $h'(f_j)\rightarrow h'(f)$ and $h(f_j)\rightarrow h(f)$ in $L^\infty(V)$.
	
	Thus because $H\in L^p(V)$, $$G_j^\perp=h(f_j) H\rightarrow h(f)H\text{ in }L^p(V).$$ Also $\nabla f_j\rightarrow F^\top$, so we have
	$$G_j^\top=h'(f_j)\nabla f_j\rightarrow h'(f)F^\top=h'(f)(F-fH) \text{ in }L^p(V)$$
	so that $$G_j\rightarrow h'(f)F+\left(h(f)-fh'(f)\right)H\text{ in }L^p(V).$$ As in the previous corollary this shows that $h'(f)F+\left(h(f)-fh'(f)\right)H$ is a weak derivative for $h$.
\end{proof}

Another interesting corollary is that, on varifolds satisfy $(H_m)$, the Michael-Simon inequality extends from ambient $C^1$ functions to functions which are weakly-differentiable with respect to $\lVert V\rVert$.
\begin{thm}
	Let $m\geq 2$. Let $V$ be a rectifiable $m$-varifold in ${\mathbb R}^N$ that satisfies $(H_m)$. There is a universal constant $C(N)$ so that for any $f\in L^{\frac{m}{m-1}}(V)$ which has a weak derivative $F\in L^1(V)$,
		$$\lVert f\rVert_{L^{\frac{m}{m-1}}(V)}\leq C(N) \lVert F\rVert_{L^1(V)}.$$
\end{thm}
\begin{proof}
	Let $f_j\in C^\infty_c({\mathbb R}^N)$ have $f_j\rightarrow f$ in $L^{\frac{m}{m-1}}(V)$. Each $f_j$ has weak derivative $F_j$ with
		$$F_j^\perp=f_jH\rightarrow fH$$
	in $L^1(V)$ and as in the proof of Theorem \ref{thm:weak-derivative-Menne} 
		$$F_j^\top\rightharpoonup F-fH$$
	in $L^2(V)$. Now consider the smooth test vector field $F_k$, and obtain
		$$\int (F_j-f_jH)\cdot F_k\ d\lVert V\rVert \rightarrow \int F^\top \cdot F_k\ d\lVert V\rVert $$
	as $j\rightarrow\infty$. Passing to a subsequence, we have 
	$$\int (F_j-f_jH)\cdot F_j d\lVert V\rVert\rightarrow \lVert F^\top\rVert_{L^2(V)}^2$$
	On the other hand by orthogonality, $\int (F_j-f_jH)\cdot F_j d\lVert V\rVert=\lVert F_j^\top\rVert_{L^2(V)}^2$. So $F_j-f_jH$ converges weakly and in norm to $F^\top$, hence $F_j-f_jH\rightarrow F^\top$ in $L^2(V)$, hence $F_j-f_jH\rightarrow F^\top=F-fH$ in $L^1(V)$.
	
	Thus $F_j\rightarrow F$ in $L^1(V)$. Using the restated Michael-Simon inequality (Theorem \ref{thm:MSrestate}), each pair $(f_j,F_j)$ has
		$$\lVert f_j\rVert_{L^{\frac{m}{m-1}}(V)}\leq C(N) \lVert F_j\rVert_{L^1(V)}.$$
	The claimed inequality follows.
\end{proof}

\begin{rem}
	It is standard to go from the inequality above to inequalities involving other exponents. Observe that to do so involves changing the assumed bound on $H$ as well.
\end{rem}

\bigskip

\subsection{Varifolds with Maslov index zero}\label{sec:maslovindexzero}

Here we assume $N$ is a Calabi-Yau manifold and $V$ is a integer rectifiable lagrangian $n$-varifold in $N$ that satisfies $( H_p )$, $p >1$. At each point $x$ in the support of $V$ with approximate tangent plane $T_xV$, since $T_xV$ is a lagrangian $n$-plane in $N$, there is an associated lagrangian angle $\ell(x) \in S^1$.

\begin{defn}
Let $V$ be a lagrangian integer-rectifiable varifold. We say $V$ has $(r,s)$-Maslov index zero if there is a real valued lift of the $S^1$ lagrangian angle, denoted $\b \in L^r(V)$, such that $\b$ has a weak derivative $B \in L^s( V)$, where $1 \leq s \leq r \leq \infty$.
\end{defn}

\begin{prop}\label{prop:embeddedZM}
If the varifold $V$ is given by a smoothly embedded lagrangian submanifold $L \subset N$ then the condition that $V$ has $(r,s)$-Maslov index zero is equivalent to the classical notion of the Maslov class vanishing on the smooth lagrangian submanifold $L$ in the case that $r = s  > n$.
\end{prop}

\begin{proof}
If the Maslov class vanishes, there is a smooth lift of the $S^1$ lagrangian angle. Conversely, if there is a scalar-valued lift $\b \in L^r(V)$ of the $S^1$ lagrangian angle with weak derivative $B \in L^s(V)$ then $\b \in W^{1,s}(V)$. Hence by the Sobolev embedding theorem $\b \in C^0(V)$. This implies that the Maslov class vanishes.
\end{proof}

For the purposes of this paper we confine ourselves to the special case with $r = \infty$ and $s = 2$. Therefore we define:

\begin{defn}
We say  $V$ has Maslov index zero if there is a real valued lift of the $S^1$ lagrangian angle, denoted $\b \in L^\infty(V)$, such that $\b$ has a weak derivative $B \in L^2(V)$.
\end{defn}

\begin{prop}\label{prop:immersedZM}
If $V$ is given by a smoothly immersed lagrangian submanifold $L:\Sigma\rightarrow N$ with vanishing Maslov class, then $V$ has Maslov index zero.
\end{prop}

\begin{proof} By hypothesis, the Maslov class $[h]=[d\beta]=0$ in $H^1(\Sigma;{\mathbb R})$, so there is a smooth function $\beta:\Sigma\rightarrow {\mathbb R}$ with, for all $p\in \Sigma$, $e^{i\beta(p)}=\ell(L_*T_p\Sigma)$. If $L$ is not an embedding, the assignment $L(x)\mapsto \beta(L(x))$ may not be well-defined, but because $L$ is in immersion, it is only so on a set of $\lVert V\rVert$-measure zero, so we may think of $\beta\in L^\infty(V)$. Then $\beta$ is a lift of the lagrangian angle. Similarly, $L_*(\nabla\beta)$ is defined $\lVert V\rVert$-a.~e.

We claim that $B=L_*(\nabla\beta)+\beta H$ is a weak derivative for $\beta$. Let $X$ be a test vector field. Considering the bundle $L^*TN\rightarrow \Sigma$, observe that $\Sigma$ is an embedded submanifold of $L^*TN$, and we can equip $L^*TN$ with a Riemannian metric which agrees with the pullback metric along $\Sigma$, so that the local submanifold geometry of $\Sigma$ in $L^*TN$ is the same as that of $L(\Sigma)$ in $N$. In particular, if we let $\mu_\Sigma$, $\div_\Sigma$, and $H_\Sigma$ denote the volume form, divergence operator, and mean curvature respectively of $\Sigma$ in $L^*TN$, we have
\begin{equation}
	\begin{aligned}
		\int \beta \div_V Xd\lVert V\rVert &= \int \beta \div_\Sigma L^*X\ d\mu_\Sigma\\
		 &= \int_\Sigma \beta \div_\Sigma (L^*X)^\top\ d\mu_\Sigma+ \int_\Sigma \beta \div	_\Sigma (L^*X)^\perp\ d\mu_\Sigma\\
		 &= -\int_\Sigma \nabla\beta \cdot (L^*X)^\top d\mu_\Sigma - \int_\Sigma \beta H_\Sigma\cdot (L^*X)^\perp\ d\mu_\Sigma\\
		 &= -\int L_*(\nabla\beta) \cdot X^\top d\lVert V\rVert - \int \beta H \cdot X^\perp d\lVert V\rVert\\
		 &= -\int L_*(\nabla\beta) \cdot X d\lVert V\rVert - \int \beta H \cdot X d\lVert V\rVert = -\int B\cdot X d\lVert V\rVert
	\end{aligned}
\end{equation}

\end{proof}

The following example illustrates the essential global nature of the zero Maslov index property and its relationship to the notion of vanishing Maslov class for an immersed lagrangian submanifold. In particular, it is possible to decompose a varifold with Maslov index zero into a sum of varifolds, each summand of which does not have Maslov index zero.

\begin{ex} 
Consider the figure-eight curve in Figure \ref{fig:fig8trans} as a 1-varifold $V$ in $\R^2$. There is an immersion $C$ of $S^1$ into $\R^2$ whose image is $V$; thus $V$ is a lagrangian varifold satisfying $(H_2)$. The lagrangian angle $\beta$ in this case is nothing more than the angle $\theta$ between the tangent line to the immersion and some fixed line, say the $x$-axis. Because the winding number of the immersion is zero, its Maslov class vanishes and hence the varifold $V$ has  Maslov index zero.

\begin{figure}
\centering
\includegraphics{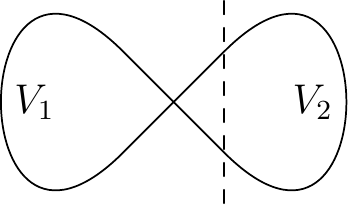}
\caption{The transversely-self-intersecting figure-eight curve $V$, decomposed.}\label{fig:fig8trans}
\end{figure}

Decompose $V$ into $V_1$ and $V_2$ as indicated in Figure \ref{fig:fig8trans}. Each $V_i$ comes from a smooth immersion $C_i$.  Computing classically we see that the change in $\theta$ along $V_1$ and $V_2$ is $\pm \frac{3\pi}{2}$.

Both $V_1$ and $V_2$ are lagrangian varifolds whose supports overlap in a set of $\lVert V\rVert$-measure zero, and $V=V_1+V_2$ as measures. However, observe that each $V_i$ has nonvanishing generalized boundary, hence does not satisfy condition $(H_2)$. Thus neither $V_1$ nor $V_2$ has Maslov index zero in our sense.\\

Now consider the figure-eight curve in Figure \ref{fig:fig8tan} as a 1-varifold $W$ in $\R^2$. As above, $W$ has Maslov index zero in our sense. If we let $V_3$ be the left lobe of $W$ and $V_4$ the right lobe, then again we have a decomposition $W=V_3+V_4$ as varifolds; moreover $V_3$ and $V_4$ themselves  arise from smooth immersions $C_3$, $C_4$ of $S^1$, hence satisfy $(H_2)$.

\begin{figure}[h]
\centering
\includegraphics{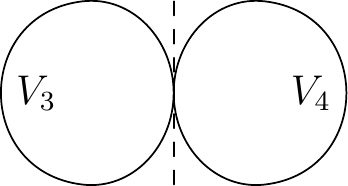}
\caption{The tangentially-self-intersecting figure-eight curve $W$, decomposed.}\label{fig:fig8tan}
\end{figure}

However, $V_3$ and $V_4$ do not have Maslov index zero in our sense. Classically, the immersions $C_3$, $C_4$ have nonzero winding numbers which sum to zero.\end{ex}

\begin{defn}
If $V$ is a lagrangian varifold coming from a special lagrangian variety, we say $V$ is a {\em special lagrangian varifold}.
\end{defn}

\begin{prop}
If $V$ is a special lagrangian varifold, then $V$ has Maslov index zero.
\end{prop}

\begin{proof}
If $V$ is special lagrangian then there is a constant lift of the lagrangian angle. We claim that this constant lift has weak derivative zero. That is, we must show that for any $\varphi \in C^{\infty}_0(M)$:
$$
\int \n \var d\lVert V\rVert = 0.
$$
Let $\{e_1, \dots, e_{2n} \}$ be a fixed frame. Using the first variation formula and the fact that $V$ is stationary we have:
$$
0 = -\int H\cdot (\var e_j) d\lVert V\rVert=\int \div_V (\var e_j) d\lVert V\rVert = \int \n \var \cdot e_j d\lVert V\rVert.
$$
The result follows.
\end{proof}

\bigskip

\subsection{A formula for the weak derivative of $\beta$} We will now use mollification to give a formula for the weak derivative of $\beta$ on a varifold with Maslov index zero that satisfies $(H_2)$.

We can think of $\beta$ as a composition $\beta(x)=\tilde{\beta}\circ T$, where $\tilde{\beta}:LGr \rightarrow S^1$ is the mod-$\pi$ angle function on the lagrangian Grassmannian and $x\mapsto T_xV$ is the tangent plane map, which by Menne's Theorem \ref{thm:Menne-tilt-excess} has a $\lVert V\rVert$-approximate differential $\nabla T$ for $\lVert V\rVert$-almost every $x$. Observe that $T_SLGr \subset \operatorname{Hom}({\mathbb R}^{2n},{\mathbb R}^{2n})$ and following Menne, we think of $(\nabla_Z)T(x)\in\operatorname{Hom}({\mathbb R}^{2n},{\mathbb R}^{2n})$ for each $x\in{\mathbb R}^{2n}$, $Z\in {\mathbb R}^{2n}$.

In fact we will first prove a more general result for any function $f\in L^p(V)$ obtained as $f={\mathbb F}\circ T$ for some smooth ${\mathbb F}:LGr \rightarrow {\mathbb R}$. At points $x$ where $T$ has a $\lVert V\rVert$-approximate differential, so does $f$, and we have following formula for it:

\begin{defn}
	We define the tangential gradient of $f$ at a  point $x$ of $C^2$ rectifiability as follows: for any $Z\in T_xV$,
	$$\nabla f(x)\cdot Z=\langle {\mathbb D}{\mathbb F}(T_xV),(\nabla_ZT)(x)\rangle$$
	where ${\mathbb D}$ is the gradient on $LGr $ with respect to the invariant metric and $\langle \cdot,\cdot\rangle$ is the inner product on $\operatorname{Hom}({\mathbb R}^{2n},{\mathbb R}^{2n})$. We will write
	$$\nabla f(x)=\langle {\mathbb D}{\mathbb F}(T_xV),(\nabla T)(x)\rangle$$
	when we need to omit the test vector $Z$.
\end{defn}

\begin{thm}\label{thm:mollificationnabla}
	$Df_\e\rightarrow \nabla f$ pointwise $\lVert V\rVert$-almost everywhere. In particular, at any point $x$ of $C^2$ rectifiability, we have $Df_\e(x)\rightarrow \nabla f(x)$ as $\e\rightarrow 0$.
\end{thm}

\begin{proof}
	We first use Taylor's theorem applied to ${\mathbb F}$, which is a smooth function. To apply Menne's Theorem \ref{thm:Menne-tilt-excess}, we identify $T_xV$ with orthogonal projection onto $T_xV$, $R(x)\in \operatorname{Hom}({\mathbb R}^{2n},{\mathbb R}^{2n})$. We have

\begin{equation}
		\begin{aligned}
		f(y)&={\mathbb F}(T_{y}V)={\mathbb F}(T_{x}V)+\langle{\mathbb D}{\mathbb F}(T_xV),(R(y)-R(x))\rangle+O(\lvert R(y) - R(x) \rvert^2)
		\end{aligned}
	\end{equation}

	 Recall the computation of $Df_\e(x)$.
\begin{equation}\label{eqn:Dfepsilon}
	\begin{aligned}
				Df_\e(x)&=\frac{\int D_x\phi_\e(x-y) f(y)d\lVert V\rVert_y}{\int \phi_\e(x-z) d\lVert V\rVert_z +\e \lVert V\rVert}-\frac{\int D_x\phi_\e(x-y) d\lVert V\rVert_y}{\int \phi_\e(x-z) d\lVert V\rVert_z+\e \lVert V\rVert}f_\e(x)
				\end{aligned}
		\end{equation}
		
	Consider the numerator of the first term of \eqref{eqn:Dfepsilon}:
		\begin{equation}\label{eqn:numeratorDfepsilonexpanded}
			\begin{aligned}
				\int f(y)&D\phi_\e(x-y)d\lVert V\rVert_y\\
				=& \int \left[{\mathbb F}(T_{x}V)+\langle {\mathbb D}{\mathbb F}(T_xV),(R(y)-R(x))\rangle+O(\lvert R(y) - R(x) \rvert^2)\right]D\phi_\e(x-y)d\lVert V\rVert_y \\
				=&f(x)\int D\phi_\e(x-y)d\lVert V\rVert_y + \int \langle{\mathbb D}{\mathbb F}(T_xV),\nabla_{y-x} T(x)\rangle D\phi_\e(x-y)d\lVert V\rVert_y \\
				&+ \int \langle{\mathbb D}{\mathbb F}(T_xV),(R(y)-R(x)) -\nabla_{y-x} T(x)\rangle D\phi_\e(x-y)d\lVert V\rVert_y \\
				&+\int  O(\lvert R(y) - R(x) \rvert^2)D\phi_\e(x-y)d\lVert V\rVert_y\\
				=&f(x)\int D\phi_\e(x-y)d\lVert V\rVert_y + \int  \nabla f(x)\cdot (y-x) D\phi_\e(x-y)d\lVert V\rVert_y\\
				&+\int \langle {\mathbb D}{\mathbb F}(T_xV),(R(y)-R(x) -\nabla_{y-x}T\rangle D\phi_\e(x-y)d\lVert V\rVert_y \\
				&+\int  O(\lvert R(y) - R(x) \rvert^2)D\phi_\e(x-y)d\lVert V\rVert_y
			\end{aligned}
		\end{equation}
	so that 
	
	\begin{equation}\label{eqn:Dfepsilonexpanded}
		\begin{aligned}
		Df_\e(x)&=
		\left[f(x)-f_\e(x)\right]\frac{\int D\phi_\e(x-y)d\lVert V\rVert_y}{\int \phi_\e(x-z)d\lVert V\rVert_z +\e \lVert V\rVert}\\
		&+\frac{\int  \nabla f(x)\cdot (y-x) D\phi_\e(x-y)d\lVert V\rVert_y}{\int \phi_\e(x-z)d\lVert V\rVert_z +\e \lVert V\rVert}\\
		&+\frac{\int \langle {\mathbb D}{\mathbb F}(T_xV),(R(y)-R(x) -\nabla_{y-x}T\rangle D\phi_\e(x-y)d\lVert V\rVert_y }{\int \phi_\e(x-z)d\lVert V\rVert_z +\e \lVert V\rVert}\\
		&+\frac{\int  O(\lvert R(y) - R(x) \rvert^2)D\phi_\e(x-y)d\lVert V\rVert_y}{\int \phi_\e(x-z)d\lVert V\rVert_z +\e \lVert V\rVert}
		\end{aligned}
	\end{equation}
	
	By virtue of Lemma \ref{lem:Dphiepsilonbounded} and Theorem \ref{thm:strong-conv-Lp}, the first term tends to zero. We will now handle the remaining terms of \eqref{eqn:Dfepsilonexpanded}. Observe that by Lemma \ref{lem:mollifiertomultiplicity}, each denominator tends to $\theta(x)$.
	
		Consider the numerator of the second term of \eqref{eqn:Dfepsilonexpanded}.  
	\begin{equation}
		\begin{aligned}
		\int \nabla f(x)&\cdot (y-x) D_x\phi_\e(x-y)d\lVert V\rVert_y\\
		&= -\e^{-1}\int\nabla f(x)\cdot (\e t) g'(\lvert t\rvert) \frac{t}{\lvert t\rvert} d\lVert \e^{-1} V-x\rVert t\\
		&=-\int\nabla f(x)\cdot t g'(\lvert t\rvert) \frac{t}{\lvert t\rvert} d\lVert \e^{-1} V-x\rVert t\\
		\end{aligned}
	\end{equation}
	As $\e\rightarrow 0$, this  converges to 
		\begin{equation}
			-\theta(x)\int_{T_xV} \nabla f(x)\cdot tg'(\lvert t\rvert) \frac{t}{\lvert t\rvert} dt
		\end{equation}
	and observing that on a plane, $\nabla \phi_1=g'(\lvert t\rvert)\frac{t}{\lvert t\rvert}$, we may integrate by parts to obtain
	\begin{equation}
		\theta(x) \nabla f(x) \int_{T_xV} \phi_1( t)dt
	\end{equation}
	Our choice of normalization is that the integral of $\phi_1$ on ${\mathbb R}^n$ is 1, so
		\begin{equation}
			\int \nabla f(x)\cdot (y-x) D_x\phi_\e(x-y)d\lVert V\rVert_y\rightarrow \theta(x)\nabla f(x)
		\end{equation}
	The corresponding denominator $\int \phi_\e(x-z) d\lVert V\rVert_z+\e \lVert V\rVert \rightarrow \theta(x)$, so the corresponding term in the limit of \eqref{eqn:Dfepsilon} is $\nabla f(x)$.
	
	Now consider the third term of \eqref{eqn:Dfepsilonexpanded}. Recall that by our choice of $\phi_\e$, all integrals are over $B(x,\e)$.
		\begin{equation}
			\begin{aligned}
				&\left\lvert \int  {\mathbb D}{\mathbb F}(T_xV)(R(y)-R(x) -\nabla_xT\cdot (y-x)) D\phi_\e(x-y)d\lVert V\rVert_y \right\rvert\\
				&\leq C\e^{-n}\e^{-1} \int \lvert R(y)-R(x)-\nabla_xT\cdot (y-x)\rvert g'\left(\frac{\lvert x-y\rvert}{\e}\right)d\lVert V\rVert_y\\
				&\leq C \e^{-n}\e^{-1}\int \lvert R(y)-R(x)-\nabla_xT\cdot (y-x)\rvert d\lVert V\rVert_y\\
				&\leq C \e^{-n} \int \frac{\lvert R(y)-R(x)-\nabla_xT\cdot (y-x)\rvert}{\lvert x-y\rvert} d\lVert V\rVert_y\\
				&\leq C \lVert V\rVert (B(x,\e))^{\frac{1}{2}} \e^{-n} \left(\int \left(\frac{\lvert R(y)-R(x)-\nabla_xT\cdot (y-x)\rvert}{\lvert x-y\rvert}\right)^2 d\lVert V\rVert_y\right)^{\frac{1}{2}}\\
				&= C  \frac{\lVert V\rVert (B(x,\e)) \e^{-n}}{\lVert V\rVert (B(x,\e))^{\frac{1}{2}}}\left(\int \left(\frac{\lvert R(y)-R(x)-\nabla_xT\cdot (y-x)\rvert}{\lvert x-y\rvert}\right)^2 d\lVert V\rVert_y\right)^{\frac{1}{2}}\\
			\end{aligned}
		\end{equation}
	where the constant $C$ absorbs ${\mathbb DF}$ and $g'$ and the penultimate step is an application of H\"older's inequality.  Menne's Theorem \ref{thm:Menne-tilt-excess} says precisely the last line of this inequality tends to $0$ as $\e\rightarrow 0$.
	
	For the fourth and final term of \eqref{eqn:Dfepsilonexpanded}, we estimate:
		\begin{equation}
			\begin{aligned}
				&\left\lvert\int  O(\lvert R(y) - R(x) \rvert^2)D\phi_\e(x-y)d\lVert V\rVert_y\right\rvert\\
				&\leq C\e^{-n}\e^{-1}\int \lvert R(y)-R(x)\rvert^2 g'\left(\frac{\lvert x-y\rvert}{\e}\right)d\lVert V\rVert_y \\
				&\leq C\e^{-n}\e^{-1}\int_{B(x,\e)} \lvert R(y)-R(x)\rvert^2 d\lVert V\rVert_y \\
				&=C\left(\e^{-\frac{n}{2}-\frac{1}{2}}\left(\int_{B(x,\e)} \lvert R(y)-R(x)\rvert^2 d\lVert V\rVert_y\right)^{\frac{1}{2}}\right)^2
			\end{aligned}
		\end{equation}
	which tends to zero by Brakke-Menne's quadratic tilt-excess decay estimate Theorem \ref{thm:tilt-excess-B-M}.

	Thus we have $Df_\e(x)\rightarrow \nabla f(x)$ as $\e\rightarrow 0$, as claimed.
\end{proof}

\begin{prop}\label{prop:nablabeta}
	At a point $x$ of $C^2$ rectifiability, $\nabla \beta(x)=-JH(x)$.
\end{prop}
\begin{proof}
	Because $x$ is a point of $C^2$ rectifiability, we have $x$ in some $C^2$ submanifold $M$ so that $T_xM=T_xV$ and $H(M;x)=H(V;x)$. Moreover, it follows from the proof of Menne's Theorem \ref{thm:Menne-tilt-excess} (see \S 3.8 of \cite{m2}) that the $\lVert V\rVert$-approximate differential of $T$ computed at $x$ is the same as that coming from $M$. So we compute $\nabla\beta$ with respect to $M$, observing that the approximate differential of the tangent plane to a $C^2$ submanifold is its second fundamental form $\II^M$.
			\begin{equation}
				\begin{aligned}
					\nabla \beta(x)=\nabla^M\beta(x)\cdot Z&=\langle {\mathbb D}\tilde{\beta}(T_xM),(\nabla_Z T)^M(x)\rangle\\
					&=\tr(\II^M(\cdot,\cdot,JZ))\\
					&=H(M;x)\cdot JZ\\
					&=H(V;x)\cdot JZ\\
					&=-JH(V;x)\cdot Z
				\end{aligned}
			\end{equation}
	where we have used the computation of ${\mathbb D}\tilde{\beta}$ as in \S \ref{sec:variationofbeta}.
			
	In particular, $\nabla\beta(x)=\nabla^M\beta(x)=-JH$ is independent of which such $C^2$ submanifold $M$ we chose.
\end{proof}

\begin{rem}
	Observe that this formula for $\nabla\beta(x)$ is the same as in the $C^2$ case, see~\cite{hl}.

	With respect to $\beta$, for varifolds satisfying $(H_2)$, we can thus interpret Theorem \ref{thm:mollificationnabla} as saying that \emph{almost everywhere, in the limit $\e\rightarrow 0$, mollification and (tangential) differentiation commute}.
\end{rem}

\begin{rem}
	For curvature varifolds in the sense of Hutchinson \cite{hutch} and Mantegazza \cite{mant}, similar results to Proposition \ref{prop:nablabeta} hold for other functions ${\mathbb F}$, whose differentials ${\mathbb DF}$ we think of as algebraic operations on the second fundamental form other than trace.
\end{rem}
\bigskip

\begin{thm}\label{thm:harvey-lawson-varifold}
Suppose $V$ is a lagrangian $n$-varifold with Maslov index zero satisfying $(H_2)$. Suppose that the lagrangian angle $\b \in L^\infty(V)$ has weak derivative $B$. Suppose $U$ is an open subset of $\R^{2n}$. Then for $\lVert V\rVert$-almost all $z \in U$:
\begin{equation}
\label{equ:compute-weak-derivative}
B(z) =\nabla\beta(z)+\beta H(z) =-JH(z) + \b H( z)
\end{equation}
\end{thm}
\begin{proof}
Let $\eta_j$ be a sequence with $\eta_j\rightarrow 0$. We have for any ambient test vector field $X$
\begin{equation}
\begin{aligned}
-\int B \cdot Xd\lVert V\rVert &= \int \beta\div Xd\lVert V\rVert\\
 &= \lim \int \beta_{\eta_j} \div Xd\lVert V\rVert \\
 &= \lim -\int X\cdot \nabla\beta_{\eta_j}d\lVert V\rVert -\int\beta_{\eta_j}H\cdot Xd\lVert V\rVert\\
 &=\lim -\int X\cdot \nabla\beta_{\eta_j}d\lVert V\rVert - \int\beta H\cdot Xd\lVert V\rVert 
\end{aligned}
\end{equation}
Now as in the proof of Theorem \ref{thm:weak-derivative-Menne}, splitting $X=X^\top+X^\perp$, we obtain the equation

$$\int (B-\beta H)\cdot Xd\lVert V\rVert = \lim \int X\cdot \nabla\beta_{\eta_j}d\lVert V\rVert $$
i.~e., $\nabla\beta_{\eta_j} \rightharpoonup  B-\beta H=B^\top$ in $L^2(V)$.

Now let $\e_k\rightarrow 0$ and consider the test vector field $X=D\beta_{\e_k}$. We have

$$\int B^\top\cdot \nabla\beta_{\e_k} d\lVert V\rVert =\lim_{\eta\rightarrow 0} \int \nabla\beta_{\eta_j} \cdot \nabla\beta_{\e_k} d\lVert V\rVert$$
As $k\rightarrow \infty$, the left-hand side converges to $\lVert B^\top\rVert^2_{L^2(V)}$. We may take a diagonal sequence on the right-hand side to conclude that  as $k\rightarrow\infty$,  $$\lVert \nabla\beta_{\e_k}\rVert^2_{L^2(V)}\rightarrow \lVert B^\top\rVert^2_{L^2(V)}$$

Thus $\nabla\beta_{\e_k}$ converges weakly and in norm to $B^\top$, hence $\nabla\beta_{\e_k}\rightarrow B^\top$ in $L^2(V)$.

On the other hand, the almost-everywhere pointwise limit of $\nabla\beta_{\e_k}$, by Theorem \ref{thm:mollificationnabla}, is $\nabla \beta$. Therefore we have
$$B^\top=\nabla\beta$$
almost everywhere. By Proposition \ref{prop:nablabeta}, $B^\top=-JH$.

By the statement of Theorem \ref{thm:weak-derivative-Menne}, $B^\perp=\beta H$ almost everywhere.
\end{proof}

\begin{rem} 
Observe that the formula \eqref{equ:compute-weak-derivative}  is formally the same as if $\beta$ were an ambient Lipschitz function.
\end{rem}

\begin{rem} 
We have chosen $p=2$ in Theorem \ref{thm:harvey-lawson-varifold} for convenience; for other $2 \leq p<\infty$ we would obtain the same formula. Observe that from (\ref{equ:compute-weak-derivative}) and the assumption $\beta\in L^\infty(V)$, it follows that $B \in L^p(V)$ if and only if $H \in L^p(V)$. 
\end{rem}

\bigskip

\subsection{Compactness of the zero-Maslov condition}
As in the classical case, our notion of weak derivative has good compactness properties.  We note these in the following lemmas, which in turn lead to a weak compactness result for Maslov zero index varifolds.

\begin{lem}
Let $V$ be an integer-rectifiable varifold satisfying $(H_p)$ for $p >1$ and let $f_j \in L^\infty( V )$ be a sequence of functions with weak derivatives $F_j \in L^p(\lVert V\rVert )$. Suppose there is a constant $C$ such that 
$$
\lVert f_j\rVert _{L^{\infty}( V)} < C, \;\;\;  \lVert F_j\rVert _{L^{p}( V)} < C 
$$
Then there is a subsequence (which we still denote $\{ f_j \}$) and a function $f \in L^\infty(V)$ with weak derivative $F \in L^p(V)$ such that $\lVert f\rVert_{L^\infty(V)} < C$, $ \lVert F\rVert _{L^{p}( V)} < C$  and for any $\varphi \in C^{\infty}_0(\R^{2n})$:
\begin{equation}
\label{equ:weak-limits-f-F}
\int \var f_j d\lVert V\rVert \to \int \var f d\lVert V\rVert, \;\;\;\; \int \var F_j d\lVert V\rVert \to \int \var F d\lVert V\rVert,
\end{equation}
\end{lem}

\begin{proof}
The existence of $f$ and $F$ and (\ref{equ:weak-limits-f-F}) follow from the precompactness of  $L^\infty(V)$ and $L^p(V)$ for $p > 1$  in the weak topology. To show that $F$ is the weak derivative of $f$ we must verify that for any $\varphi \in C^{\infty}_0(\R^{2n})$:
\begin{equation}
\label{equ:weak-deriv.-in-the-limit}
\int \nabla \var f d\lVert V\rVert = -\int \var F d\lVert V\rVert 
\end{equation}
Since $f_j \to f$ in $L^\infty(V)$ we have $f_j(x) \to f(x)$, for $V$ a.e. $x$. Using $| \nabla \var | \leq | D \var |$ by dominated convergence we have that as $j \to \infty$:
$$
\int \nabla \var f_j d\lVert V\rVert  \to \int \nabla \var f d\lVert V\rVert. 
$$
Then (\ref{equ:weak-deriv.-in-the-limit}) follows from (\ref{equ:weak-limits-f-F}) and for any $\varphi \in C^{\infty}_0(\R^{2n})$ and every $j$:
$$
\int \nabla \var f_j d\lVert V\rVert = -\int \var F_j d\lVert V\rVert
$$ 
The result follows.
\end{proof}

Our compactness result will concern a sequence $(V_j)$ of Maslov-index-zero varifolds, each with an associated lagrangian angle $\beta_j$ and weak derivative $B_j$. We will therefore need to examine the convergence of the triple $(V_j, \beta_j,B_j)$. We will use the following result about Radon measures:

\begin{lem}\label{lem:radon-measures-p}
	Let $1\leq q<\infty$, and let $(\mu_j)$, $(\nu_j)$ be sequences of Radon measures on ${\mathbb R}^{2n}$ with $\mu_j\rightarrow \mu$ and $\nu_j\rightarrow \nu$ in the weak topology. Suppose that that there exists $C$, independent of $j$, so that for each $\nu_j$-measurable $A\subset{\mathbb R}^{2n}$, $A$ is $\mu_j$-measurable and $$(\mu_j(A))^q\leq C\nu_j(A).$$	
	Then $\mu$ is absolutely continuous with respect to $\nu$ and for any $\nu$-measurable set $A$, $A$ is $\mu$ measurable with $$\left(\mu(A)\right)^q\leq C\nu(A).$$
\end{lem}
\begin{proof}
	First we recall (see e.~g.~\cite{eg}) that the convergence $\mu_j\rightarrow\mu$, $\nu_j\rightarrow \nu$ implies that for any open set $U$ and any compact set $K$, we have
		\begin{equation}
			\begin{aligned}
				\mu(U)\leq \liminf_j \mu_j(U), \;\; \limsup_j\mu_j(K)\leq \mu(K)\\
				\nu(U)\leq \liminf_j \nu_j(U), \;\; \limsup_j\nu_j(K)\leq \nu(K)
			\end{aligned}
		\end{equation}
	Therefore, for any bounded set $A$,
	$$\left(\mu(\mathring{A})\right)^q\leq C\nu(\overline{A})$$
	where $\mathring{A}$ and $\overline{A}$ are the interior and closure of $A$, respectively.
	
	Let $K$ be any compact set, and let $N_r(K)$ be the open $r$-neighborhood of $K$. Then we have
	$$K=\bigcap_{r>0} N_r(K)=\bigcap_{r>0} \overline{N_r(K)},$$
	hence
	\begin{equation}
		\left(\mu(K)\right)^q=\lim_{r\searrow 0}\mu\left(N_r(K)\right)^q\leq C \lim_{r\searrow 0}\nu\left(\overline{N_r(K)}\right)=C\nu(K).
	\end{equation}
	 Let $B$ be any Borel set; we have that $B$ is both $\nu$- and $\mu$-measurable and
	\begin{equation}\label{eqn:borel-ineq}
		\begin{aligned}
		\left(\mu(B)\right)^q&=\sup\left\{\left(\mu(K)\right)^q\middle\vert K\subset B, K\text{ compact}\right\}\\
		&\leq C\sup\left\{\nu(K)\middle\vert K\subset B, K\text{ compact}\right\}=C\nu(B).
		\end{aligned}
	\end{equation}
	
	Now if $A$ is a set with $\nu(A)=0$, then for any $k\in\mathbb{N}$ there is an open set $U_k$ with $A\subset U_k$ and $\nu(U_k)<\frac{1}{k}$. Then $A \subset U=\bigcap_{k}U_k$, $U$ is Borel, and for all $k\in\mathbb{N}$, $\left(\mu(U)\right)^q\leq C\frac{1}{k}$. So $U$ is a nullset for $\mu$, hence $\mu(A)=0$. Thus we have shown that $\mu$ is absolutely continuous with respect to $\nu$.\\
	
	By the Radon-Nikodym theorem, we have that any $\nu$-measurable set is $\mu$-measurable, and the inequality \eqref{eqn:borel-ineq} extends for all $\nu$-measurable $A$.
	
\end{proof}

\begin{prop}
\label{lem:Lpweak-conv}
Let $V_j$ be a sequence of integer-rectifiable varifolds which converge weakly to the integer-rectifiable varifold $V$. Suppose $V_j$ and $V$ satisfy $(H_p)$ for $p >1$. Let $f_j \in L^\infty( V_j)$ be a sequence of functions with weak derivatives $F_j \in L^p( V)$. Suppose there is a constant $C$ such that 
$$
\lVert f_j\rVert _{L^{\infty}( V_j)} < C, \;\;\;  \lVert F_j\rVert _{L^{p}( V_j)} < C 
$$
Then there is a subsequence (which we still denote $\{ f_j \}$) and a function $f \in L^\infty(V)$ with weak derivative $F \in L^p( V)$ such that $\lVert f\rVert_{L^\infty(V)} < C$, $\lVert F\rVert _{L^{p}( V)} < C$ and for any $\varphi \in C^{\infty}_0(\R^{2n})$:
$$
\int \var f_j d\lVert V_j\rVert \to \int \var f d\lVert V\rVert, \;\;\;\; \int \var F_j d\lVert V_j\rVert \to \int \var F d\lVert V\rVert,
$$
\end{prop}

\begin{proof}
For each $j$ write:
$$ f^+_j(x) =
\begin{cases}
f_j(x) & \mbox{if} \;\;\;  f_j(x) \geq 0\\
0  &\mbox{if} \;\;\; f_j(x) < 0\\
\end{cases}
$$
$$ f^-_j(x) =
\begin{cases}
-f_j(x) & \mbox{if} \;\;\;  f_j(x) \leq 0\\
0  &\mbox{if} \;\;\; f_j(x) > 0\\
\end{cases}
$$
Then 
$$
f_j(x) = f_j^+(x) - f_j^-(x).
$$

For a $V_j$-measurable set $A \subset \R^{2n}$ define:
\begin{equation}
\begin{aligned} 
\mu^+_j(A) &= \int_A f_j^+ d\lVert V_j\rVert \\
\mu^-_j(A) &= \int_A f_j^- d \lVert V_j\rVert
\end{aligned}
\end{equation}
Then $\mu^+_j$ and $\mu^-_j$ are Radon measures. $\mu^+_j$ and  $\mu^-_j$ are both absolutely continuous with respect to $\lVert V_j\rVert$. Moreover,
\begin{equation} 
\label{equ:ineuality-on-mu+}
\begin{aligned}
\mu^+_j(A) &\leq  \lVert f_j\rVert_{L^\infty(V_j)} \lVert V_j\rVert (A) \leq C \lVert V_j\rVert (A) \\
\mu^-_j(A) &\leq  \rVert f_j\rVert_{L^\infty(V_j)} \lVert V_j\rVert (A) \leq C \lVert V_j\rVert(A) 
\end{aligned}
\end{equation}
Choosing subsequences we can suppose that there are Radon measures $W^+$ and $W^-$ such that as $j \to \infty$ we have:
$$
\mu^+_j \to W^+, \;\;\;\; \mu^-_j \to W^-
$$
where the convergence is weak convergence of Radon measures. By the $q=1$ case of Lemma \ref{lem:radon-measures-p} we have for any $\lVert V\rVert$-measurable set that 
\begin{equation}
	W^\pm(A)\leq C \lVert V\rVert(A).
\end{equation}\bigskip

Let $f^+=D_{\lVert V\rVert}W^+$, $f^-=D_{\lVert V\rVert}W^-$. We have for any $\var \in C^\infty_0(\R^{2n})$, as $j \to \infty$:
\begin{equation} 
\begin{aligned}
\int \var f_j^+ d\lVert V_j\rVert = \int \var d\mu_j^+&\to \int \var d W^+=\int \var f^+ d\lVert V\rVert\\
\int \var f_j^- d\lVert V_j\rVert =\int \var d\mu_j^-&\to \int\var dW^-=\int \var f^- d\lVert V\rVert
\end{aligned}
\end{equation}

Setting $f=f^+-f^-$, it follows from the $q=1$ case of Lemma \ref{lem:radon-measures-p} that $\lVert f^\pm\rVert_{L^\infty(V)} \leq C$, hence $\lVert f\rVert_{L^\infty(V)} \leq C$ and we have as $j \to \infty$,
$$
\int \var f_j d\lVert V_j\rVert \to \int \var f d\lVert V\rVert.
$$\bigskip

Next we consider the sequence $F_j \in L^p(V_j)$. Let $\{ e_1, \dots, e_{2n} \}$ be an orthonormal frame and set $F_j^i = F_j \cdot e_i$ for $i=1, \dots, 2n$. Define for each $j$ and $i$:
$$ F^{i, +}_j(x) =
\begin{cases}
F^i_j(x) & \mbox{if} \;\;\;  F^i_j(x) \geq 0\\
0  &\mbox{if} \;\;\; F^i_j(x) < 0\\
\end{cases}
$$
and similarly define $F^{i, -}_j$. For a $V_j$-measurable set $A \subset \R^{2n}$ define:
\begin{equation}
\begin{aligned}
\nu^{i,+}_j(A) &= \int_A F^{i, +}_j d\lVert V_j\rVert\\
\nu^{i,-}_j(A) &= \int_A F^{i, -}_j d\lVert V_j\rVert 
\end{aligned}
\end{equation}
Then $\nu^{i,+}_j$ and $\nu^{i,-}_j$ are Radon measures. Both are both absolutely continuous with respect to $\lVert V_j\rVert$. Moreover, using the H\"older inequality,
\begin{equation}
\begin{aligned} 
\nu^{i,+}_j(A) &\leq  \lVert F^i_j\rVert_{L^p(V_j)} \lVert V_j\rVert(A)^{\frac{p-1}{p}} \leq C \lVert V_j\rVert(A)^{\frac{p-1}{p}}  \\
\nu^{i,-}_j(A) &\leq  \lVert F^i_j\rVert_{L^p(V_j)} \lVert V_j\rVert(A)^{\frac{p-1}{p}}  \leq C \lVert V_j\rVert(A)^{\frac{p-1}{p}} 
\end{aligned}
\end{equation}
Applying Lemma \ref{lem:radon-measures-p} with $q$ chosen so $\frac{1}{p}+\frac{1}{q}=1$, we obtain limit measures $\nu^{i,\pm}$ which are absolutely continuous with respect to $\lVert V\rVert$. Let $F^{i,\pm}$ be the Radon-Nikodym derivative of $\nu^{i,\pm}$ with respect to $\lVert V\rVert$. Observe that for any $\var\in C^\infty_c$,
\begin{equation}
\begin{aligned}
\int \var F_j^{i,+} d\lVert V_j\rVert = \int \var d\nu_j^{i,+}&\to \int \var d \nu^{i,+}=\int \var F^{i,+} d\lVert V\rVert\\
\int \var F_j^{i,-} d\lVert V_j\rVert =\int \var d\nu_j^{i,-}&\to \int\var d\nu^{i,-}=\int \var F^{i,-} d\lVert V\rVert
\end{aligned}
\end{equation}

Setting $F=\sum \left(F^{i,+}-F^{i,-}\right)e_i$, we obtain $F$ so that for any smooth compactly supported test vector field $X$,
\begin{equation}
	\int X\cdot F_j d\lVert V_j\rVert\rightarrow \int X\cdot F d\lVert V\rVert
\end{equation}

In particular, consider $X$ with $\lVert X\rVert_{L^q(V)}\neq 0$. Then
\begin{equation}
	\frac{\int X\cdot F\ d\lVert V\rVert}{\lVert X\rVert_{L^q(V)}}=\lim_j\frac{\int X\cdot F_j\ d\lVert V_j\rVert}{\lVert X\rVert_{L^q(V_j)}}\leq C
\end{equation}
hence $\lVert F\rVert_{L^p(V)}\leq C$.\\

We wish to show that $F$ is the weak derivative of $f$. We have verified that:
$$
\int \var F_j d\lVert V_j \rVert \to \int \var F d\lVert V \rVert
$$
Next we wish to show  that:
\begin{equation}
\label{equ:nabla-convergence}
 \int \n \var f_j d\lVert V_j\rVert \to \int \n \var f d\lVert V\rVert.
\end{equation}
Observe that (perhaps passing to a subsequence) may assume that the Radon measures $\tilde{\mu}_j^\pm=f_j^\pm V_j$ on $Gr(n)$ converge to $\tilde{\mu}^\pm$ with $\mu^\pm(A)=\int_{A}\int_{Gr(n)_x} d\tilde{\mu}^\pm(x,S)$. Then we have:
\begin{equation}
\begin{aligned}
 \int \n \var f_j^+ d\lVert V_j\rVert &= \int_{Gr(n)} \proj_S D\var(x) d\tilde{\mu}_j^+(x, S) \\
 &\to  \int_{Gr(n)} \proj_S D\var(x) d\tilde{\mu}^+(x, S) = \int \n \var f^+ d\lVert V\rVert
 \end{aligned}
 \end{equation}
Similarly:
$$
 \int \n \var f_j^- d\lVert V_j\rVert \to \int \n \var f^- d\lVert V\rVert.
$$
Clearly, (\ref{equ:nabla-convergence}) follows. But since $F_j$ is the weak derivative of $f_j$ we have:
$$
 \int \n \var f_j d\lVert V_j\rVert = -\int \var F_j d\lVert V_j\rVert
$$
Taking the limit as $j \to \infty$ shows that $F$ is the weak derivative of $f$.
\end{proof}

\begin{lem}
\label{lem:L-infy-conv}
Let $V_j$ be a sequence of integer-rectifiable varifolds of uniformly bounded mass which converge weakly to the integer-rectifiable varifold $V$. Let $f_j \in L^{\infty}(V)$ be a sequence of functions. Suppose there is  a constant $C$ independent of $j$ such that 
$$
\lVert f_j\rVert_{L^{\infty}(V_j)} < C
$$
Then there is a subsequence (which we still denote $\{ f_j \}$) and a function $f \in L^{\infty}(V)$ such that $\lVert f\rVert_{L^{\infty}(V)} < C$ and for any $\varphi \in C^{\infty}_0(\R^{2n})$:
$$
\int \var \sin( f_j) d\lVert V_j\rVert \to \int \var \sin(f) d\lVert V\rVert, \;\;\;\; \int \var \cos(f_j) d\lVert V_j\rVert \to \int \var \cos(f) d\lVert V\rVert,
$$
\end{lem}

\begin{proof}
To begin note that by Proposition \ref{lem:Lpweak-conv} there is a subsequence of $\{ f_j \}$ that we continue to denote $\{ f_j \}$ and a function $f \in L^{\infty}(V)$ such that for any $\varphi \in C^{\infty}_0(\R^{2n})$:
$$
\int \var f_j d\lVert V_j\rVert  \to \int \var f d\lVert V\rVert 
$$
We want to show that for each integer $\ell > 0$ and any $\varphi \in C^{\infty}_0(\R^{2n})$ we have as $j \to \infty$:
\begin{equation}
\label{equ:induction-L}
\int \var f_j^\ell d\lVert V_j\rVert  \to \int \var f^\ell d\lVert V\rVert .
\end{equation}
To prove (\ref{equ:induction-L}) we use induction and suppose we have for any $\varphi \in C^{\infty}_0(\R^{2n})$ as $j \to \infty$:
\begin{equation}
\label{equ:induction-L-1}
\int \var f_j^{\ell-1} d\lVert V_j\rVert  \to \int \var f^{\ell-1} d\lVert V\rVert .
\end{equation}

We use mollification as follows: Mollify $f_j$ with respect to the measure $V_j$. That is,
$$
(f_j)_\e(x) = \frac{\int f_j(y) \phi_\e(x-y) d\lVert V_j \rVert(y)}{\int  \phi_\e(x-y) d\lVert V_j \rVert(y) + \e \lVert V_j \rVert}
$$
Note that for all $j$, $(f_j)_\e \in C^\infty_0(\R^{2n})$ and
\begin{equation}
|| (f_j)_\e ||_{L^\infty(\R^{2n})} < C.
\end{equation}
In addition for every $\e >0$;
\begin{equation}
\lim_{j \to \infty} (f_j)_\e = (f)_\e   \;\;\;\mbox{pointwise}.
\end{equation}
We have by the inductive assumption and dominated convergence that for any $\e > 0$:
\begin{equation}
\lim_{j \to \infty} \int  (f_j)_\e \var f_j^{\ell-1} d\lVert V_j\rVert = \int (f)_\e \var  f^{\ell-1} d\lVert V\rVert
\end{equation}
Hence, letting $\e \to 0$
\begin{equation}
\lim_{j \to \infty} \int   \var f_j^{\ell} d\lVert V_j\rVert = \int  \var  f^{\ell} d\lVert V\rVert
\end{equation}
This proves the inductive step, verifying (\ref{equ:induction-L}) for each integer $\ell$.
The result follows using the Taylor expansion for sine and cosine.
\end{proof}

\begin{thm}
\label{thm:convergence-beta}
Let $V_j$ be a sequence of varifolds with Maslov index zero satisfying $(H_2)$, each with finite mass, that converge,  in the sense of varifold convergence, to the  lagrangian integer-rectifiable varifold $V$ also satisfying $(H_2)$. Suppose that for each $j$ the varifold $V_j$ admits a lift $\b_j$ of the $S^1$ lagrangian angle with weak derivative $B_j$ so that there are  constants $c, C$ independent of $j$ with:
$$
\lVert \b_j \rVert_{L^{\infty}(V_j)} < c,
$$
and
$$
 \lVert B_j \rVert_{L^{2}(V_j)} < C.
$$
Then $V$ admits a lift $\b$ of the $S^1$ lagrangian angle with weak derivative $B$ satisfying: 
$$
\lVert \b \rVert_{L^{\infty}(V)} < c, \;\;\;  \lVert B \rVert_{L^{2}(V)} < C.
$$

\end{thm}

\begin{proof}
Apply Proposition \ref{lem:Lpweak-conv} to the sequences $\{ b_j \}$ and $\{ B_j \}$ to conclude that there exist $b \in L^\infty(V)$ with weak derivative $B \in L^2(V)$ with $ \lVert \b \rVert_{L^{\infty}(V)} < C$ and $\lVert B \rVert_{L^{2}(V)} < C$ such that for subsequences of $\{ b_j \}$ and $\{ B_j \}$ (which we continue to denote $\{ b_j \}$ and $\{ B_j \}$) we have for any $\varphi \in C^{\infty}_0(\R^{2n})$ as $j \to \infty$

\begin{equation}
\label{equ:version1-beta}
\int \var b_j d\lVert V_j\rVert \to \int \var b d\lVert V\rVert, \;\;\;\; \int \var B_j d\lVert V_j\rVert \to \int \var B d\lVert V\rVert
\end{equation}
Apply Lemma \ref{lem:L-infy-conv} to the sequence $\{ b_j \}$  to conclude that there is a subsequence (which we continue to denote $\{ b_j \}$) such that for any $\varphi \in C^{\infty}_0(\R^{2n})$ as $j \to \infty$:
\begin{equation}
\label{equ:version2-beta}
\int \var \sin(b_j) d\lVert V_j\rVert \to \int \var \sin(b) d\lVert V\rVert, \;\;\;\; \int \var \cos(b_j) d\lVert V_j\rVert \to \int \var \cos(b) d\lVert V\rVert,
\end{equation}

We need to verify that $b$ is a lift of the lagrangian angle. To show this we lift the lagrangian integer-rectifiable varifolds $V_j$ and $V$ to lagrangian varifolds $\tilde{V_j}$ and $\tilde{V}$, i.e.~Radon measures on $LGr$. On $LGr$ the lagrangian angle $\ell$ is a global smooth function. By  assumption for each $j$ we have,
$$
\ell_{|_{V_j}} = e^{i b_j}
$$
Let $\rho \in C^{\infty}_0(LGr)$.  Then, as $j \to \infty$,
$$
\int \rho \ell d \tilde{V_j} \to \int \rho \ell  d \tilde{V}. 
$$
On the other hand, 
$$\int\rho(x,S)\ell(x,S)d \tilde{V_j}(x,S)= \int \rho(x,T_xV_j)\ell(T_xV_j)d\lVert V_j\rVert_x= \int \rho(x,T_xV_j) e^{ib_j(x)}d\lVert V_j\rVert_x$$
If we take $\rho$ to be independent of $S$, i.e.~$\rho\in C^{\infty}_0(\R^{2n})$, and apply (\ref{equ:version2-beta}), we have for any $\rho \in C^{\infty}_0(\R^{2n})$ as $j \to \infty$:
$$
\int \rho e^{i b_j} d\lVert V_j\rVert \to \int \rho e^{i b} d\lVert V\rVert,
$$
Using the uniqueness of the weak limit we conclude that $e^{i b} = \ell_{|_{V}}$ a.e.
The result follows.
\end{proof}

Combining Allard's compactness theorem \ref{thm:compact-Hp-varifolds} and Theorem \ref{thm:convergence-beta}, we obtain the following compactness result for varifolds with Maslov index zero:
\begin{thm}
	Let $V_j$ be a sequence of lagrangian varifolds satisfying $(H_2)$ with Maslov index zero.  Suppose there are $C_1,C_2,C_3,C_4<\infty$ so that each lift $\beta_j$ of the lagrangian angle of $V_j$ and its weak derivative $B_j$ satisfy:
		$$\lVert V_j\rVert\leq C_1,\;\; \lVert H_j\rVert_{L^2(V_j)}\leq C_2, \;\; \lVert \beta_j\rVert_{L^\infty(V_j)}\leq C_3, \;\; \lVert B_j\rVert_{L^2(V_j)}\leq C_4.$$
	Then there is a subsequence of the $V_j$ which converges to a lagrangian varifold $V$ satisfying $(H_2)$ with Maslov index zero, for which the lift $\beta$ of the lagrangian angle of $V$ and its weak derivative $B$ satisfy 
		$$	\lVert V\rVert\leq C_1,\;\; \lVert H\rVert_{L^2(V)}\leq C_2,\;\;
			\lVert \beta\rVert_{L^\infty(V)}\leq C_3,\;\; \lVert B\rVert_{L^2(V)}\leq C_4.$$
\end{thm}

\begin{rem}
	In virtue of Theorem \ref{thm:harvey-lawson-varifold}, we observe that in the presence of a uniform bound on $\lVert \beta_j\rVert_{L^\infty(V_j)}$, a uniform bound on $\lVert H_j\rVert_{L^2(V_j)}$ implies a bound on $\lVert B_j\rVert_{L^2(V_j)}$ and vice versa.
\end{rem}

\begin{prop}
Let $\Phi$ be a symplectic diffeomorphism. If $V$ is a varifold satisfying $(H_2)$ with Maslov index zero, then so is $(\Phi)_\sharp(V)$.
\end{prop}

\bigskip

\section{ {The approximate flows}}\label{sec:epsilonflows}
\label{sec:epsilon-flows}

Now let $V=V(0)$ be a varifold with Maslov index zero satisfying $(H_2)$. In this section, we will show that for each $\e>0$, the flow equation
	\begin{equation}\label{eqn:epsilonflowequation}
		\frac{d}{dt}V^\e(t)=JD(\beta(V^\e(t))_\e)
	\end{equation}
has a solution starting from initial data $V=V(0)$, where the mollification is carried out with respect to $V^\e(t)$. Each $V^\e(t)$ will be a  varifold with Maslov index zero satisfying $(H_2)$.

Throughout, $\lVert\cdot\rVert$ denotes the $C^0$ norm of a continuous function, and the supremum of the operator norm of a tensor.\\

We will produce the one-parameter family $V(t)$ of varifolds with Maslov index zero solving \ref{eqn:epsilonflowequation} by a sort of Euler's method.  We wish to define, for any mollification parameter $\e$, and any $k\in\mathbb{N}$, a sequence of varifolds with Maslov index zero. Each varifold in the sequence will be given by flowing by the hamiltonian diffeomorphism generated by $\beta_{\e}$ for time $2^{-k}$; then we compute $\beta_{\e}$ for the new varifold and repeat. Again for exposition we state all results and proofs for varifolds in ${\mathbb R}^{2n}$.\\

For now we fix $\e>0$.

\begin{defn}
	Given $p,k\in\mathbb{N}$, define the diffeomorphisms $\Psi_t^{k,p}$ and lagrangian integer-rectifiable varifolds $V^{k,p}$ by:
	\begin{equation*}
		\begin{aligned}
			V^{k,0}&=V\\
			\beta^{k,p}&=\beta(V^{k,p})\\
			\beta_{\e}^{k,p}&=(\beta^{k,p})_\e\\
			\Psi_t^{k,p}& \text{ generated by } JD\beta_{\e}^{k,p}\\
			V^{k,p+1}&=\left(\Psi_{2^{-k}}^{k,p}\right)_\sharp V^{k,p}
		\end{aligned}
	\end{equation*}
	We point out for clarity that the mollification in $\beta^{k,p}_\e$ is with respect to the varifold $V^{k,p}$. 
\end{defn}

Each $V^{k,p+1}$ is the hamiltonian-diffeomorphic image of $V^{k,p}$, hence of $V^{k,0}=V$. Therefore the $V^{k,p}$ all have Maslov index zero.

Our goal is to use the $V^{k,p}$ to approximate a one-parameter family of varifolds with Maslov index zero moving according to (\ref{eqn:epsilonflowequation}); we will use the fact that the dyadic rationals are dense in $\R$. To that end, we adopt the notation:
\begin{defn}
	Given a dyadic rational $t$, define for each $k\in\mathbb{N}$
		\begin{equation*}
			\begin{aligned}
			\Psi^k(t)&=\Psi_{2^{-k}}^{k,p-1}\circ\cdots\circ\Psi_{2^{-k}}^{k,1}\circ \Psi_{2^{-k}}^{k,0}\\
			V^k(t)&=\left(\Psi^k(t)\right)_\sharp  V=V^{k,p}
			\end{aligned}
		\end{equation*}
	where $t=p2^{-k}$.
\end{defn}

We wish to roughly estimate how the lagrangian angle $\beta$ changes along each sequence $V^{k,p}$. 

\begin{lem}\label{lem:betaepsilonexponential}
	$\lVert\beta^{k,p}\rVert_{L^\infty( V^k(t) )} \leq \left(1+C\e^{-q}2^{-k}\right)^{2^kt}\lVert\beta(V)\rVert_{L^\infty(V)}$, where $C$ and $q$ are universal constants coming from Lemma \ref{lem:derivativesofbetaepsilon}.
\end{lem}
\begin{proof}
	For any $(x,S)\in LGr $, we have using the first variation of $\beta$ (\ref{eqn.deltabeta3}) that
	\begin{equation}
		\begin{aligned}
			\lvert\beta^{k,p}(\Psi_{2^{-k}}^{k,p-1}(x,S))-\beta^{k,p-1}(x,S)\rvert&=\left\lvert\int_0^{2^{-k}}-\tr_{\left(\Psi^{k,p-1}_\tau\right)_\sharp S}DJJD\beta_{\e}^{k,p-1}\left(\Psi_\tau^{k,p-1}(x,S)\right)d\tau\right\rvert\\
			&\leq 2^{-k}\lVert D^2\beta_{\e}^{k,p-1}\rVert_{C^0}
		\end{aligned}
	\end{equation}
	Thus $\lVert\beta^{k,p}\rVert_{L^\infty( V^{k,p})}\leq \lVert\beta^{k,p-1}\rVert_{L^\infty( V^{k,p-1})}+2^{-k}\lVert D^2\beta_{\e}^{k,p-1}\rVert_{C^0}$. Applying Lemma \ref{lem:derivativesofbetaepsilon}, we obtain
		\begin{equation}
			\lVert\beta^{k,p}\rVert_{L^\infty( V^{k,p})}\leq \left(1+2^{-k}C\e^{-3-n}\right)\lVert\beta^{k,p-1}\rVert_{L^\infty(V)}
		\end{equation}
	Iterating this estimate gives the result.	
\end{proof}

	Fixing $t$ and noting that $(1+\frac{x}{n})^n\nearrow e^x$ for $x>0$, we obtain:
\begin{cor}\label{cor:betaexponentialtime}
	$\lVert\beta(V^k(t))\rVert_{L^\infty( V^k(t) )}\leq \exp(C\e^{-q}t)\lVert\beta(V)\rVert_{L^\infty(V)}$. 
\end{cor}
	
	In particular, $\lVert\beta(V^k(t))\rVert_{L^\infty( V^k(t) )}$ is bounded uniformly in $k$ and uniformly for $t$ in a compact interval. Moreover, 

\begin{lem}
	$\lVert D\Psi^k(t)\rVert_{C^0}$ is bounded uniformly in $k$ and uniformly for $t$ in a compact interval.
\end{lem}
\begin{proof}
	Differentiating the Taylor expansion $\Psi_t^{k,s}(x)=x+tJD\beta_{\e}^{k,s}(x)+O(t^2)$, we have
	\begin{equation}
		D\Psi_{2^{-k}}^{k,s}=\Id+2^{-k}DJD\beta_{\e}^{k,s}+O(2^{-2k})
	\end{equation}
	and we can estimate $\lVert DJD\beta^{k,s}_\e\rVert\leq \lVert D^2\beta^{k,s}_\e\rVert +c\rVert D\beta^{k,s}_\e\rVert$ and apply Lemma \ref{lem:derivativesofbetaepsilon} and Corollary \ref{cor:betaexponentialtime} to get 
		\begin{equation}
			\lVert D\Psi_{2^{-k}}^{k,s}\rVert\leq 1+2^{-k}\left(C\e^{-q}+c\e^{-(n+\e+1)}\right)\exp(C\e^{-q}2^{-k}s)\lVert\beta\rVert+O(2^{-2k})
		\end{equation}
	Applying this estimate to each factor in $D\Psi^k=D\Psi^{k,p-1}_{2^{-k}}\cdot\cdots\cdot D\Psi^{k,0}_{2^{-k}}$ and using Lemma \ref{lem:betaepsilonexponential}, we have
		\begin{equation}
			\begin{aligned}
			\lVert D\Psi^{k}(t)\rVert\leq& \prod_{s=0}^{p-1}\left(1+2^{-k}\left(C\e^{-q}+c\e^{-(n+2)}\right)\exp(C\e^{-q}s2^{-k})\lVert\beta\rVert_{L^\infty(V)}+O(2^{-2k})\right)\\
			\leq& \prod_{s=0}^{p-1}\left(1+2^{-k}\left(C\e^{-q}+c\e^{-(n+2)}\right)\exp(C\e^{-q}t)\lVert\beta\rVert_{L^\infty(V)}+O(2^{-2k})\right)\\
			=&\left(1+2^{-k}\left(C\e^{-q}+c\e^{-(n+2)}\right)\exp(C\e^{-q}t)\lVert\beta\rVert_{L^\infty(V)}\right)^{2^kt}\\
			&+\sum_{s=1}^{2^kt}\left(\begin{smallmatrix}2^kt\\s\end{smallmatrix}\right)\left(2^{-2k}\right)^s\left(1+2^{-k}\left(C\e^{-q}+c\e^{-(n+2)}\right)\exp(C\e^{-q}t)\lVert\beta\rVert_{L^\infty(V)}\right)^{2^kt-s}
			\end{aligned}
		\end{equation}
	where the first term converges to $\exp((C\e^{-q}+c\e^{-(n+2)})\exp(C\e^{-q}t)\lVert\beta\rVert_{L^\infty(V)} t)$  and the second term goes to zero as $k\rightarrow\infty$.
\end{proof}

Similar estimates hold for $D^r\Psi^k(t)$, uniform in $k$ and $t$ in a compact interval, for any $r$. Observe that $\frac{\partial}{\partial t}\Psi_t^{k,p}=JD\beta_{\e}^{k,p}$ is uniformly bounded, independent of $k$. Thus $\Psi^k(t)$ is Lipschitz with Lipschitz constant $\e^{-(n+2)}\exp(C\e^{-q}t)$, which is uniform in $k$ and uniform in $t$ in compact intervals.

Therefore we obtain, for each dyadic rational $t=p2^{-k}$, a sequence of $V^{k,p}$ of  varifolds with Maslov index zero arising as the pushforward of $V$ under uniformly-bounded diffeomorphisms $\Psi^{k}(t)$; we can therefore extract a subsequential limit $V^\e(t)$. Since the dyadic rationals are countable, we can apply a diagonal argument so that the same indices $k$ give convergence to $V^\e(t)$ for each $t$.

Moreover, if we define at each dyadic rational $t=p2^{-k}$, $\beta_{\e}^k(t)=\beta_{\e}^{k,p}$, we have uniform similar uniform bounds allowing us to extract $\beta^{\e}(t)$, which is the Lipschitz limit in $t$ and the smooth limit spatially.\\

We would like to say that these $V^\e(t)$ can be extended into a one-parameter varifolds which satisfies the flow equation \eqref{eqn:epsilonflowequation}; in fact we will see that the $V^\e(t)$ are generated by a one-parameter family of hamiltonian diffeomorphisms.  Given an interval $[0,T]$, let $\Psi^{\e}(t)$ denote a choice of $C^\infty$ subsequential limit of the $\Psi^k(t)$ for each dyadic rational $t\in[0,T]$.

\begin{thm}
	We may extend $\Psi^\e(t)$ to all $t\in[0,T]$ so that $\Psi^\e$ is smooth in $t$. Moreover, if we define $V^{\e}(t)=\left(\Psi^{\e}(t)\right)_\sharp V$, then the $\Psi^\e(t)$ are hamiltonian, generated by  $\beta^{\e}(t)=\lim_k\beta^k_\e(t)=\beta(V^\e(t))_\e$:
		\begin{equation*}
			\frac{\partial}{\partial t}\Psi^{\e}(t)=JD(\beta(V^\e(t))_\e)
		\end{equation*}
\end{thm}

\begin{proof}
	First observe that each $\frac{\partial^r}{\partial t^r}\Psi_t^{k,p}$ can be expressed in terms of the derivatives of $\beta_\e^{k,p}$. By Lemma \ref{lem:derivativesofbetaepsilon}, these are bounded depending on $r$, $\e$, and $\lVert \beta^{k,p}\rVert$, which by Lemma \ref{lem:betaepsilonexponential} can be bounded independently of $k$.

	 Consider $t_0<t_1$, dyadic rationals. We compute
	 \begin{equation}
	 	\begin{aligned}
	 		\Psi_\e(\Psi_\e(x,t_0),t_1-t_0)-\Psi_\e(x,t_0)=&\lim_k \Psi_\e^k(\Psi_\e^k(x,t_0),t_1-t_0)-\Psi_\e^k(x,t_0)\\
	 		=&\lim_k\sum_{\ell} \int_0^{2^{-k}}JD\beta_{\e}^{k,p+\ell}(\Psi^{k,p+\ell}(x,\tau))d\tau\\
	 		=&\lim_k\int_{t_0}^{t_1}JD\beta_{\e}^k(\Psi^k(x,\tau))d\tau
	 	\end{aligned}
	 \end{equation}
	 where $t_0=p2^{-k}$ and the sum is over all $\ell$ with $t_0\leq \ell 2^{-k}<t_1$.
	 
	 Now everything in sight that depends on $k$ converges uniformly spatially and in $t$, so we have
	  \begin{equation}\label{eqn:floweqnintegral}
	 	\begin{aligned}
	 		\Psi_\e(\Psi_\e(x,t_0),t_1-t_0)-\Psi_\e(x,t_0)=\int_{0}^{t_1-t_0}JD\beta_{\e}(\Psi_\e(x,t_0+\tau))d\tau
	 	\end{aligned}
	 \end{equation}
	This extends to all $t_0$,  $t_1$ by continuity, which in turn implies the claimed evolution of $\Psi_\e$.
	
	That $\lim_k\beta^k_\e(t)=\beta(V^\e(t))_\e$ follows from the convergence $V^k(t)\rightarrow V^\e(t)$.
	 \end{proof}

\begin{defn}
	 We define the $\e$-approximate flow by $V^{\e}(t)=\left(\Psi^{\e}(t)\right)_\sharp V$.
\end{defn}

\bigskip

\section{ Uniform sup bounds on $\beta$ along the  $\e$-flows}\label{sec:sup-bounds}

Consider the $\e$-flow of  varifolds with Maslov index zero $V^\e(t)$ constructed in Section \ref{sec:epsilon-flows} and let $\b^\e(t) \in L^\infty(V^\e(t))$ denote the single-valued lift of the $S^1$ lagrangian angle. The main theorem of this section is:

\begin{thm}
\label{thm:boundonbeta}
There is a constant $A$, depending on $\lVert \beta\rVert_{L^\infty(V(0))}$ but independent of $\e$ and $t \in [0, T]$ such that for $\e < 1$, 
$$\lVert\b^\e(t) \rVert_{L^\infty(\lVert V^\e(t)\rVert)} < A.$$
\end{thm}

\bigskip

A key observation in the proof of this result is
\begin{prop}
	Along each $\e$-flow, the lagrangian angle $\beta^\e(t)$ satisfies
\begin{equation} \label{eqn:evolbeta}
			\begin{cases}\frac{\partial}{\partial t}\b^\e(t)=\operatorname{div}_{V^\e(t)}D(\b^\e)_\e\\
			\b^\e(0)=\b\end{cases}
		\end{equation}
\end{prop}
\begin{proof}
	Apply the variation formula for $\beta$ \eqref{eqn.deltabeta3} and the fact that $V^\e(t)$ evolves by $JD(\beta^\e)_\e$.
\end{proof}

While \eqref{eqn:evolbeta} is similar to the parabolic equation $\frac{\partial}{\partial t}u=\div_{V^{\e}}Du$ --- which has a maximum principle --- because we have mollified, \eqref{eqn:evolbeta} is not a partial differentiation equation. Rather it is a pseudo-differential evolution equation. To obtain estimates we regularize this equation in two distinct ways. First we mollify the varifolds $V^\e(t)$ to produce an initial value problem on an open set in $\R^{2n}$. Next we regularize the operator $\div_V D$ to make it into a uniformly elliptic operator. The standard elliptic estimates can then be exploited to prove the theorem. This argument is similar, in spirit, to the method of vanishing viscosity.
Our treatment of the pseudo-differential evolution equation is motivated by treatment of a similar equation in \cite{ta}.

\bigskip

\begin{lem}\label{lem:unique-Linfty-solution}
	Fix $\e >0$. For any $\eta > 0$ the initial value problem
		\begin{equation}\label{IVP}
			\begin{cases}\frac{\partial}{\partial t}u=\operatorname{div}_{V^\e(t)}Du_\eta\\
			u(0)=u_0\end{cases}
		\end{equation}
	has a unique solution $u(\eta)$ in $C({\mathbb R}_+;L^\infty(V^\e(t)))$ for $u_0\in L^\infty(V^\e(t))$.
\end{lem}

\begin{proof}
	We regard the equation as an ordinary differential equation on $L^\infty(V^\e(t))$ and use the method of Picard.
	
	For any $T>0$, use ${\mathcal X}(T)$ to denote $C([0,T];L^\infty(V^\e(t)))$ equipped with the norm 
		\begin{equation}
		\label{defn-of norm}
			\lVert \psi\rVert=\sup_{t\in[0,T]}\exp(-2C(\eta)t)\lVert \psi(t)\rVert_{L^\infty(V^\e(t))}
		\end{equation}
	where $C(\eta)=\sup \lVert \phi_\eta\rVert_{C^2}$. Note that as $\eta \to 0$, $C(\eta) \to \infty$ and therefore the norm goes to zero.
	
	Now (cf.~Lemma \ref{lem:betaepsilonexponential}) define $\Gamma:{\mathcal X}(T)\rightarrow{\mathcal X}(T)$ by
		\begin{equation}
			\Gamma(\psi)(t,x)=u_0\left(\left(\Psi^\e_t\right)^{-1}(x)\right)+\int_0^t\operatorname{div}_{V^\e(s)}D\left(\psi(s)_\eta\right)\left(\left(\Psi^\e_{t-s}\right)^{-1}(x)\right)\ ds
		\end{equation}
	where $\Psi^\e_t$ is the one-parameter family of diffeomorphisms generated by $JD\beta^\e_\e$. Recall that $\Psi^\e_t$ takes the varifold $V^\e_0$ to $V^\e_t$.

	We compute
		\begin{equation}
			\begin{aligned}
				\lVert\Gamma(\psi_2)-\Gamma(\psi_1)\rVert&=\left\lVert\int_0^t\operatorname{div}_{V^\e(s)}D\left((\psi_2(s)-\psi_1(s))_\eta\right)\left(\left(\Psi^\e_{t-s}\right)^{-1}(x)\right)\ ds\right\rVert\\
				&\leq\sup_{t\in [0,T]} \exp(-2C(\eta)t) C(\eta)\int_0^t \lVert \psi_2(s)-\psi_1(s)\rVert_{L^\infty(V^\e(s))}\ ds\\
				&\leq \sup_{t\in [0,T]} \exp(-2C(\eta)t) C(\eta)\int_0^t \lVert \psi_2-\psi_1\rVert \exp(2C(\eta)s)\ ds\\
				&= \frac{1}{2}\left\lVert \psi_2-\psi_1\right\rVert \sup_{t\in [0,T]} \left(1-\exp(-2C(\eta)t)\right)\leq \frac{1}{2}\left\lVert \psi_2-\psi_1\right\rVert 
			\end{aligned}
		\end{equation}
	where the first inequality is obtained by differentiating on the mollifier $\phi_\eta$ and using the definition  $C(\eta)=\sup \lVert \phi_\eta\rVert_{C^2}$. The second inequality follows from the definition (\ref{defn-of norm}).
	Thus $\Gamma:{\mathcal X}(T)\rightarrow{\mathcal X}(T)$ is a contraction. In the standard way we conclude that the sequence $\psi_k$ defined by
		\begin{equation}
			\begin{cases}
				\psi_0(t,x)=u_0\left(\left(\Psi_t^\e\right)^{-1}(x)\right)\\
				\psi_{k+1}=\Gamma(\psi_k)
			\end{cases}
		\end{equation}
	converges to a unique solution of \eqref{IVP}, defined on $[0,T]$. Moreover none of the choices involved depended on $T$, so we can extend the solution to all $t\in[0,\infty)$.
	
\end{proof}

\bigskip

For any lagrangian varifold $V$  define the mollified lagrangian varifold $\phi_\sigma*V$ by setting, for any $A\subseteq LGr $, $$(\phi_\sigma* V)(A)=\int_A \phi_\sigma(x)dV_x(S)d\lVert V\rVert_x.$$

For each $x\in{\mathbb R}^{2n}$, define the mollified divergence operator $\operatorname{div}_{V_\sigma (x)}$ by 
	\begin{equation}
		\operatorname{div}_{V_\sigma (x)}X=\int_{LGr _x}\tr_S\left(\proj_SDX\right)d(\phi_\sigma*V)_x(S)
	\end{equation}
This operator is the integral over the entire fiber of the Grassmann bundle of the divergence of $X$ computed on each lagrangian plane $S$.\\

We can estimate the derivatives of $\operatorname{div}_{V_\sigma}X$ as follows:
\begin{lem}
	\begin{equation*}
		\lVert \operatorname{div}_{V_\sigma}X\rVert_{C^1} \leq \lVert \phi_\sigma\rVert_{C^1}\lVert X\rVert_{C^1}+\lVert X\rVert_{C^2}
	\end{equation*}
\end{lem}

\begin{lem}\label{lem:unique-smooth-solution}
	The initial value problem
		\begin{equation}\label{smoothIVP}
			\begin{cases}\frac{\partial}{\partial t}u(t,x)=\operatorname{div}_{V_\sigma^\e(t) (x)}Du_\eta(t,x)\\
			u(0,x)=\phi_\sigma \star u_0(x)\end{cases}
		\end{equation}
	for $u_0\in C^\infty({\mathbb R}^{2n})$ has a unique solution in $C^1({\mathbb R};C^\infty({\mathbb R}^{2n}))$.
\end{lem}
\begin{proof}
	We will show that the equation has a unique solution in $C^1({\mathbb R};C^{k,\alpha}({\mathbb R}^{2n}))$ for each $k \in \mathbb{N}$, $\alpha > 0$; then uniqueness will give existence and uniqueness in $C^\infty$.
	
	For fixed $k$, we regard the equation as an ordinary differential equation on $C^{k,\alpha}({\mathbb R}^{2n})$ and use the method of Picard. For fixed $T>0$, use $\mathcal{X}(T)$ to denote $C^0([0,T];C^{k,\alpha}({\mathbb R}^{2n}))$ equipped with the norm
		\begin{equation}
			\lVert \psi\rVert=\sup_{t\in[0,T]}\exp(-2C_k(\eta,\sigma)t)\lVert \psi(t)\rVert_{C^{k,\alpha}({\mathbb R}^{2n})}
		\end{equation}
	where $C_k(\eta,\sigma)=\lVert \phi_\eta \rVert_{C^{k+2}}+\lVert \phi_\sigma \rVert_{C^k}$.

	Now define $\Gamma:\mathcal{X}(T)\rightarrow\mathcal{X}(T)$ by
		\begin{equation}
			\Gamma(\psi)(t,x)=u_0(x)+\int_0^t\operatorname{div}_{V_\sigma^\e(s)}D\left( \psi(s)_\eta \right)\ ds
		\end{equation}
	Then we compute

	\begin{equation}
			\begin{aligned}
				\lVert\Gamma(\psi_2)-\Gamma(\psi_1)\rVert&=\left\lVert\int_0^t\operatorname{div}_{V^\e(s)}D\left((\psi_2(s)-\psi_1(s))_\eta \right)\left(\left(\Psi^\e_{t-s}\right)^{-1}(x)\right)\ ds\right\rVert\\
				&\leq\sup_{t\in [0,T]} \exp(-2C_k(\eta ,\sigma)t) C_k(\eta ,\sigma)\int_0^t \lVert \psi_2(s)-\psi_1(s)\rVert_{C^k({\mathbb R}^{2n})}\ ds\\
				&\leq \sup_{t\in [0,T]} \exp(-2C_k(\eta)t) C_k(\eta)\int_0^t \lVert \psi_2-\psi_1\rVert \exp(2C_k(\eta,\sigma)s)\ ds\\
				&= \frac{1}{2}\left\lVert \psi_2-\psi_1\right\rVert \sup_{t\in [0,T]} \left(1-\exp(-2C_k(\eta,\sigma)t)\right)<\frac{1}{2}\left\lVert \psi_2-\psi_1\right\rVert 
			\end{aligned}
		\end{equation}
	so that $\Gamma$ is a contraction mapping, hence has a unique fixed point in $\mathcal{X}(T)$, which is the solution of (\ref{smoothIVP}).
	
\end{proof}

\bigskip

Consider a subset $A \subset LGr(\R^{2n})$. Then:
$$
A = \{ (x, S): x \in U \subset \R^{2n}, S \in F_x \subset \pi^{-1}(x) \}
$$
Define the set $JA$ where $J$ is an orthogonal complex structure on $\R^{2n}$ by:
$$
JA = \{ (x, JS): x \in U \subset \R^{2n}, S \in F_x \subset \pi^{-1}(x) \}
$$
Let $V$ be a varifold on $LGr(\R^{2n})$. Then $A$ is $V$-measurable if and only if $JA$ is $V$-measurable. We define the varifold $JV$ by:
$$
JV(A) = V(JA).
$$

We introduce the notation, for each $\sigma>0$, the differential operator
		\begin{equation*}
			K_\sigma = \operatorname{div}_{V_\sigma^\e(t) (x)} \circ D 
		\end{equation*}
For each 	 $\sigma, \rho >0$ we define the differential operator 
\begin{equation*}
			K_{\sigma, \rho} = \operatorname{div}_{V_\sigma^\e(t) (x)} \circ D + \rho \operatorname{div}_{JV_\sigma^\e(t) (x)} \circ D
		\end{equation*}	
The symbol of $K_\sigma$ is
$$
\sigma(K_\sigma)(\xi) = \int_{LGr _x} |\xi_S|^2 d(\phi_\sigma*V)_x(S),
$$		
where $\xi_S$ denotes the projection of $\xi \in \R^{2n}$ onto the lagrangian $n$-plane $S$.
The symbol of $K_{\sigma, \rho}$ is
$$
\sigma(K_{\sigma, \rho})(\xi) = \int_{LGr _x} |\xi_S|^2 + \rho |\xi_{JS}|^2  d(\phi_\sigma*V)_x(S).
$$
It follows that for $\rho > 0$ and  $\Omega \subset \supp (\phi_\sigma)$ open that on $\Omega$ the differential operator $K_{\sigma, \rho}$ is strictly elliptic.
The Garding inequality gives:
\begin{equation}
\label{equ:Garding-ineq}
-\big( K_{\sigma, \rho} u, u \big)_{L^2(\R^{2n})} + || u ||^2_{L^2(\R^{2n})} \geq c || u ||^2_{H^1(\Omega)},
\end{equation}		
where the constant $c=c(\rho) > 0$. The $L^2$ inner product is defined using the Lebesgue measure. Hence,
\begin{equation}
-\big( K_{\sigma, \rho} u, u \big)_{L^2(\R^{2n})} \geq c || u ||^2_{H^1(\Omega)} -  || u ||^2_{L^2(\R^{2n})}.
\end{equation}		
Therefore, for all $\rho > 0$,
\begin{equation}
- \big( K_{\sigma, \rho} u, u \big)_{L^2(\R^{2n})} \geq  -  || u ||^2_{L^2(\R^{2n})}.
\end{equation}		
Hence, letting $\rho \to 0$,
\begin{equation}
\label{equ:Garding-ineq-2}
\big( K_{\sigma} u, u \big)_{L^2(\R^{2n})} \leq  || u ||^2_{L^2(\R^{2n})}.
\end{equation}	
Apply the inequality (\ref{equ:Garding-ineq-2}) to the function $u^{\frac{p}{2}}$ where $p$ is a positive even integer. Then: 
\begin{equation}
\label{equ:Garding-ineq-3}
\big( K_{\sigma} (u^{\frac{p}{2}}), u^{\frac{p}{2}} \big)_{L^2(\R^{2n})} \leq  c || u^{\frac{p}{2}} ||^2_{L^2(\R^{2n})} = c \lVert u\rVert^p_{L^p(\R^{2n})}.
\end{equation}

\bigskip

\begin{lem}
	For any $f\in C^\infty_c({\mathbb R}^{2n})$, any $p$, we have
		\begin{equation*}
			K_\sigma f^p=pf^{p-1}K_\sigma f+p(p-1)f^{p-2}\lvert\n_\sigma f\rvert^2
		\end{equation*}
\end{lem}
\begin{proof}
	We compute:
		\begin{equation}
			\begin{aligned}
			K_\sigma f^p(x)&=\int_{LGr _x}\tr_S(D^2 f^p)d (V_\sigma)_x(S)\\
			&=\int_{LGr _x}\tr_S(pf^{p-1}(x)D^2 f+p(p-1)f^{p-2}(x)Df\otimes Df)d (V_\sigma)_x(S)\\
			&=pf^{p-1}(x)K_\sigma f(x)+p(p-1)f^{p-2}(x)\lvert\n_\sigma f\rvert^2(x)
			\end{aligned}
		\end{equation}

\end{proof}

\bigskip

Applying the lemma to (\ref{equ:Garding-ineq-3}) for $p \geq 2$ we derive:
\begin{equation}
\frac{p}{2} \left( u^{\frac{p}{2} -1} K_{\sigma}  u, u^{\frac{p}{2}} \right)_{L^2(\R^{2n})} \leq   \lVert u \rVert^p_{L^p(\R^{2n})}.
\end{equation}
Hence
\begin{equation}
\label{equ:Garding-ineq-4}
\frac{p}{2} \big(  K_{\sigma}  u, u^{p-1} \big)_{L^2(\R^{2n})} \leq   \lVert u\rVert^p_{L^p(\R^{2n})}.
\end{equation}

Recall the operator $L_\eta$ for $\eta >0$ defined by:
\begin{equation}
\label{equ:smoothing-operator}
L_\eta(u) = \phi_\eta \star u
\end{equation}
Then $L_\eta$ is a pseudo-differential operator. We can, of course, consider the  differential operator $K_\sigma$ as a pseudo-differential operator and therefore the composition $K_\sigma \circ L_\eta$ is a pseudo-differential operator. To make this precise we follow Hormander [Ho] and define the space of symbols $S^m_{\rho, \d} (\Omega)$ where $\Omega$ is an open set in $\R^{2n}$ and $0 \leq \rho, \d \leq 1$:
\begin{eqnarray*}
\label{equ:symbol-space}
S^m_{\rho, \d} (\Omega) &=&\big  \{ p \in C^\infty( \Omega \times \R^{2n}): \mbox{for any compact}   \; K \subset \Omega,   \mbox{and any multiindices} \; \a, \b  \; \\
&&\mbox{there exists a constant} \; C_{K, \a, \b} \geq 0 \; \mbox{such that for all} \; x \in K, \xi \in \R^{2n} \\ 
&& \lVert \p^\b_x \p^\a_\xi p(x, \xi) \lVert \leq C_{K, \a, \b} (1+ \lvert\xi\rvert)^{m -\rho\lvert\a\rvert + \d \lvert\b\rvert} \big \}
\end{eqnarray*}
A pseudo-differential operator  with symbol in $S^m_{\rho, \d}$ is lies in the space $OPS^m_{\rho, \d}$. Then $K_\sigma$ lies in $OPS^2_{1, 0}$ and $L_\eta$ lies in $OPS^{-\infty}$. 
By Proposition \ref{prop:bound-on-moll-op} the operators $L_\eta$ as operators on $L^p( V )$ are uniformly bounded  for $0 < \eta \leq 1$ and therefore from (\ref{equ:Garding-ineq-4}) for $0 < \eta \leq 1$ there is a constant independent of $\eta$ such that:
\begin{equation}
\label{equ:Garding-ineq-5}
\frac{p}{2} \big(  (K_{\sigma} \circ L_\eta)  u, u^{p-1} \big)_{L^2(\R^{2n})} \leq  c  \lVert u\rVert^p_{L^p(\R^{2n})}.
\end{equation}

We rewrite the  initial value problem (\ref{smoothIVP}) using this notation.
\begin{equation}\label{smoothIVP-2}
			\begin{cases}\frac{\partial}{\partial t}u(t,x)= K_\sigma \circ L_\eta u(t,x)\\
			u(0,x)=\phi_\sigma \star u_0(x)\end{cases}
		\end{equation}

\begin{prop}
The unique smooth solution $u = u(t, x)$ of (\ref{smoothIVP-2}) satisfies for every positive even integer $p$:
\begin{equation}
\frac{d}{d t} \lVert u\rVert_{L^p} \leq c \lVert u\rVert_{L^p}
\end{equation}
where the constant is independent of $p$. Hence,
\begin{equation}
\sup_{[0,T], x  \in \R^{2n}} |u(t, x)| \leq C \sup_ {x  \in \R^{2n}} |\phi_\sigma \star u_0(x)|.
\end{equation}
\end{prop}	

\begin{proof}
For even $p \geq 2$, we compute. The inequality uses (\ref{equ:Garding-ineq-5})
		\begin{equation}
			\begin{aligned}
				\frac{d}{dt}\lVert u\rVert_{L^p}^p &=\frac{d}{dt} \left(u^{\frac{p}{2}},u^{\frac{p}{2}}\right)\\
				&=2\frac{p}{2}\left(u^{\frac{p}{2}-1}\frac{\partial u}{\partial t},u^{\frac{p}{2}}\right)\\
				&=p\left(u^{\frac{p}{2}-1}(K_\sigma \circ L_\eta) u,u^{\frac{p}{2}}\right)\\
				&=p\left((K_\sigma \circ L_\eta) u,u^{p-1}\right)\\
				&\leq 2 c \lVert u\rVert_{L^p}^p. 
			\end{aligned}
		\end{equation}
From this inequality we have:
\begin{equation}
\label{equ:Lp-space-inequality}
\frac{d}{dt}\lVert u\rVert_{L^p} \leq 2 c \lVert u\rVert_{L^p}
\end{equation}	
Applying Gronwall's inequality [E] to (\ref{equ:Lp-space-inequality}) yields for each $p$:
\begin{equation}
\lVert u(t) \rVert_{L^p} \leq C \lVert u(0) \rVert_{L^p} \leq C \lVert u(0) \rVert_{L^\infty}
\end{equation}
The result follows.
\end{proof}

\bigskip
 
 We have shown:
 
 \begin{thm}
 For each $\sigma > 0$ the unique smooth solutions of  (\ref{smoothIVP-2}) satisfy the estimate:
\begin{equation}
\label{equ:sup-bound}
 \lVert u(\eta)(x,t)\rVert_{L^\infty} \leq C  \lVert u_0(x)\rVert_{L^\infty}
\end{equation}
 where the constant is independent of $\eta$.  
 \end{thm} 
 
 \bigskip

\begin{thm}
\label{thm:sup-bound-for-u}
There is a constant $C > 0$ such that for each $\e > 0$ and $\eta > 0$ the unique solution $u$ of (\ref{IVP}) satisfies the estimate:
$$
\lVert u(x,t) \rVert_{L^\infty({V^\e})} \leq C \lVert u_0(x))\rVert_{L^\infty({V^\e})}
$$
\end{thm}

\begin{proof}
Fix $\e >0$ and $\eta > 0$. Consider a sequence $\sigma_i \to 0$. Let $\{ u_i \}$ be solutions of (\ref{smoothIVP-2})) for $\sigma_i$. Choose a subsequence (that we will continue to denote $\{ u_i \}$) that converges in $L^\infty(\R^{2n})$ to a function $u$. Then $u$ satisfies the  bound (\ref{equ:sup-bound}). We claim that $u$ is a solution of  (\ref{IVP}). Therefore,  the unique solution of (\ref{IVP}) satisfies the bound (\ref{equ:sup-bound}).

For each $i$ the smooth function $u_i$ satisfies the equation:
$$
\frac{\p}{\p t} u_i = \div_{V_{\sigma_i}} D L_\eta u_i.
$$
Hence for any test function $\varphi$ we have:
\begin{eqnarray*}
\int \frac{\p}{\p t} u_i \; \varphi \; d\lVert V_{\sigma_i}\rVert &=& \int  \div_{V_{\sigma_i}} D L_\eta u_i \; \varphi \; d\lVert V_{\sigma_i}\rVert \\
&=& - \int  D L_\eta u_i \cdot \nabla_{\sigma_i} \; \varphi \; d\lVert V_{\sigma_i}\rVert - \int  D L_\eta u_i \cdot H_{\sigma_i} \; \varphi \; d\lVert V_{\sigma_i}\rVert \\
&=& - \int   \nabla_{\sigma_i} L_\eta u_i \cdot \nabla_{\sigma_i}  \varphi \; d\lVert V_{\sigma_i}\rVert - \int  D L_\eta u_i \cdot H_{\sigma_i} \; \varphi \; d\lVert V_{\sigma_i}\rVert \\
&=&  \int    L_\eta u_i \; \div_{V_{\sigma_i} } \; \nabla_{\sigma_i}  \varphi \; d\lVert V_{\sigma_i}\rVert +  \int  L_\eta u_i  \; H_{\sigma_i} \cdot \nabla_{\sigma_i}  \varphi \; d\lVert V_{\sigma_i}\rVert \\
&& - \int  D L_\eta u_i \cdot H_{\sigma_i} \; \varphi \; d\lVert V_{\sigma_i}\rVert \\
\end{eqnarray*}
As $\sigma_i \to 0$ by the dominated convergence theorem and Brakke's orthogonality of $H$ we have:
$$
 \int  L_\eta u_i  \; H_{\sigma_i} \cdot \nabla_{\sigma_i}  \varphi \; d\lVert V_{\sigma_i}\rVert \to 0
$$
Therefore letting $\sigma_i \to 0$ we have:
\begin{eqnarray*}
&&\int \frac{\p}{\p t} u \; \varphi \; d\lVert V\rVert \\
&=& \int    L_\eta u \; \div_{V } \nabla_V \varphi  \; d\lVert V\rVert  - \int  D L_\eta u \cdot H \; \varphi \; d\lVert V\rVert \\
&=& -\int    D L_\eta u \cdot  \nabla_V \varphi \; d\lVert V\rVert - \int  L_\eta u \;  H \cdot \nabla_V \varphi \; d\lVert V\rVert  - \int  D L_\eta u \cdot H \; \varphi \; d\lVert V\rVert \\
&=& \int   \div_{V } D L_\eta u \;  \varphi \; d\lVert V\rVert+ \int D L_\eta u \cdot H \; \varphi \; d\lVert V\rVert   - \int  D L_\eta u \cdot H \; \varphi \; d\lVert V\rVert \\
&=& \int   \div_{V } D L_\eta u \;  \varphi \; d\lVert V\rVert
\end{eqnarray*}
Since this is true for any test function $\varphi$ it follows that $u$ is a solution of (\ref{IVP}). The result follows.
\end{proof}

\begin{cor}
 There is a constant $A >0$ such that for any $\e >0 $:
 $$
 \sup_{t \in [0,T]} \sup_{ V^\e(t)} |\b^\e| < A
 $$ 
\end{cor}

\begin{proof}
On the $\e$-flow apply Theorem \ref{thm:sup-bound-for-u} with initial condition $u_0 = \beta$ and $\eta = \e$.
\end{proof}

\bigskip

\section{  Control on volume and $\beta$ along the  $\e$-flows}

In Section \ref{sec:epsilon-flows} we have, for each fixed $\e > 0$, produced a flow $V^\e(t)$ of Maslov index zero, integral rectifiable lagrangian varifold satisfying $(H_2)$. We would like to extract a limit flow $V(t)$ by taking a limit as $\e \to 0$ for each $t$ of the varifolds $V^\e(t)$. Moreover we wish extract a limit that inherits the properities of the $V^\e(t)$. However the estimates derived in Section  \ref{sec:epsilon-flows} blow-up as $\e \to 0$. In Section \ref{sec:sup-bounds} we showed that the lagrangian angles satisfy sup bounds. In this section we exploit the weak derivatives of the $\b^\e(t)$ to derive decay estimates for both the volumes $\lVert V^\e(t) \rVert$ and the lagrangian angles $\b^\e(t)$ along the $\e$-flows.

\bigskip

Let $V$ be a integer lagrangian varifold with Maslov index zero. Let $\b \in L^\infty(V)$ denote the lift of the lagrangian angle with weak derivative $B \in L^2(V)$.  The mollified lagrangian angle is:
\begin{equation}
			\beta_\e(x)=\frac{\int \phi_\e(x-y)\beta(y)d\lVert V\rVert_y}{\int \phi_\e(x-y)  d\lVert V\rVert_y +\e \lVert V\rVert}
\end{equation}
We have noted above that, for fixed $\e$,  $\beta_\e$ is globally well-defined and smooth.

\begin{thm}\label{thm:errorestimate}
Suppose $V$ is an integral lagrangian  $n$-varifold with Maslov index zero satisfying $(H_2)$, $n \geq 2$. Suppose that the lagrangian angle $\b \in L^\infty(V)$ has weak derivative $B \in L^2(V)$. Let $H$ denote the mean curvature on $V$. Then if $\b_\e$ is the mollified lagrangian angle:
\begin{equation}
\label{equ:final-formula}
\n_x \b_\e(x) = B_\e(x) -\b_\e(x) H_\e(x) + E(x, \e),
\end{equation}
where
$$
B_\e(x) = \frac{\int  \phi_\e(x-y)B(y) d\lVert V\rVert_y}{\int \phi_\e(x-y)  d\lVert V\rVert_y +\e \lVert V \rVert}, \;\; H_\e(x) =\frac{\int  \phi_\e(x-y)H(y) d\lVert V\rVert_y}{\int \phi_\e(x-y)  d\lVert V\rVert_y +\e \lVert V \rVert},
$$
and
\begin{equation}
\label{equ:error-estimate}
 |E(x, \e)|^2 < c_1 \e^{-2} \frac{ \int_{B(x, \e) \times Gr(n,2n)} |S -T^n_x(V)|^2 dV(\xi, S)  }{ {\int_{B(x, \e)} \phi_\e(x-y) d\lVert V\rVert_y}  }
 \end{equation}
where $c_1$ is a  constant depending  on  $||\b||_{L^\infty(V)}$. 
\end{thm}

\begin{proof}
We begin by computing:
\begin{eqnarray}
\label{equ:rough-formula-nabla-beta}
&&\n_x  \b_\e(x) \\  \nonumber
&&= \frac{\int \n_x  \phi_\e(x-y)\b(y) d\lVert V\rVert_y}{\int \phi_\e(x-y)  d\lVert V\rVert_y +\e \lVert V \rVert} -\b_\e(x) \frac{\int \n_x  \phi_\e(x-y) d\lVert V\rVert_y}{\int \phi_\e(x-y)  d\lVert V\rVert_y +\e \lVert V \rVert} \\  \nonumber
&&= \frac{-\int \n_y  \phi_\e(x-y)\b(y) d\lVert V\rVert_y}{\int \phi_\e(x-y)  d\lVert V\rVert_y +\e \lVert V \rVert} + \frac{\int \big( \n_x \phi_\e(x-y) +  \n_y \phi_\e(x-y) \big)\b(y) d\lVert V\rVert_y}{\int \phi_\e(x-y)  d\lVert V\rVert_y +\e \lVert V \rVert } \\  \nonumber
&&+ \b_\e(x) \frac{\int \n_y  \phi_\e(x-y) d\lVert V\rVert_y}{\int \phi_\e(x-y)  d\lVert V\rVert_y +\e \lVert V \rVert} -\b_\e(x) \frac{\int \big( \n_x  \phi_\e(x-y) +\n_y  \phi_\e(x-y)\big)d\lVert V\rVert_y}{\int \phi_\e(x-y)  d\lVert V\rVert_y +\e \lVert V \rVert}
\end{eqnarray}

First observe:
\begin{equation}
\label{equ:compute-B}
\frac{-\int \n_y  \phi_\e(x-y)\b(y) d\lVert V\rVert_y}{\int \phi_\e(x-y)  d\lVert V\rVert_y +\e \lVert V \rVert} =  \frac{\int  \phi_\e(x-y)B(y) d\lVert V\rVert_y}{\int \phi_\e(x-y)  d\lVert V\rVert_y +\e \lVert V \rVert} = B_\e(x)
\end{equation}

Fix $x \in \R^{2n}$ and let $E$ be an arbitrary fixed vector in $\R^{2n}$. Using the first variation formula we have:
\begin{eqnarray} \nonumber
\Big(\int \n_y  \phi_\e(x-y)  d\lVert V\rVert_y \Big) \cdot E &=& \int \n_y  \phi_\e(x-y) \cdot E  \;  d\lVert V\rVert_y \\ \nonumber
 &=&  \int \div_V (\phi_\e(x-y)  E) \;  d\lVert V\rVert_y\\ \nonumber
&=&-\int \phi_\e(x-y)  H \cdot E \;  d\lVert V\rVert_y \\ 
&=&-\Big( \int \phi_\e(x-y)  H  \;  d\lVert V\rVert_y \Big) \cdot E\\ \nonumber
\end{eqnarray}
Thus,
\begin{equation}
\label{equ:compute-H}
\frac{\int \n_y  \phi_\e(x-y)   d\lVert V\rVert_y}{\int \phi_\e(x-y)  d\lVert V\rVert_y +\e \lVert V \rVert} = -\frac{ \int \phi_\e(x-y)  H d\lVert V\rVert_y}{\int \phi_\e(x-y)  d\lVert V\rVert_y +\e \lVert V \rVert} = -H_\e(x)
\end{equation}

We define:
\begin{eqnarray} \label{equ:define-E}
&&E(x, \e) \\ \nonumber
&& = \frac{\int \big( \n_x \phi_\e(x-y) +  \n_y \phi_\e(x-y) \big)\b(y) d\lVert V\rVert_y}{\int \phi_\e(x-y)  d\lVert V\rVert_y +\e  \lVert V \rVert }  -\b_\e(x) \frac{\int \big( \n_x  \phi_\e(x-y) +\n_y  \phi_\e(x-y)\big)d\lVert V\rVert_y}{\int \phi_\e(x-y)  d\lVert V\rVert_y +\e \lVert V \rVert} 
\end{eqnarray}

Using (\ref{equ:compute-B}), (\ref{equ:compute-H}) and (\ref{equ:define-E}) in (\ref{equ:rough-formula-nabla-beta}) yields (\ref{equ:final-formula}). It remains to estimate the error function $E(x, \e)$.

We consider the  $n$-plane $S \in Gr(n, 2n)$ as the orthogonal projection $\R^{2n} \to S$ and write:
\begin{eqnarray*}
&&\big| \int \big( \n_x \phi_\e(x-y) +  \n_y \phi_\e(x-y) \big)\b(y) d\lVert V\rVert_y \big| \\ 
&\leq& \big| \int_{ B(x, \e) } \big( \n_x \phi_\e(x-y) +  \n_y \phi_\e(x-y) \big)\b(y) d\lVert V\rVert_y  \big| \\
&\leq& \big|  \int_{ B(x, \e) } \big( T^n_x(V) -T^n_y(V) \big)(D \phi_\e)(x-y) \b(y) d\lVert V\rVert_y \big|
\end{eqnarray*}
Note that on $B(x, \e)$:
$$
D \exp \big( \frac{\e^2}{|x|^2 - \e^2} \big) =  \exp \big( \frac{\e^2}{|x|^2 - \e^2} \big) \big(\frac{-2 \e^2 x}{ ( |x|^2- \e^2 )^2} \big).
$$
Therefore there is a constant  $K > 0$ such that:
$$
| D \phi_\e(x-y) | \leq  (\phi_\e(x-y))^{\frac{1}{2}} \frac{K}{\e}  .
$$

Using the  H\"older inequality,
\begin{eqnarray*}
&&|\int_{ B(x, \e) }  \big( T^n_x(V) -T^n_y(V) \big)(D \phi_\e)(x-y) \b(y)  d\lVert V\rVert_y| \\ 
&\leq & \int_{ B(x, \e) }  | T^n_x(V) -T^n_y(V)| \e^{-1}  K (\phi_\e(x-y))^{\frac{1}{2}} | \b(y)|  d\lVert V\rVert_y \\
&\leq&  \e^{-1} K \Big( \int_{B(x, \e) \times Gr(n,2n)} |S -T^n_x(V)|^2 dV(\xi, S)  \Big)^{\frac{1}{2}}  \bigg(\int_{B(x, \e)} \phi_\e(x-y) d\lVert V\rVert_y \bigg)^{\frac{1}{2}} ||\b ||_{L^\infty({ V } )}  \\
\end{eqnarray*}
Thus,
\begin{eqnarray*}
&&\frac{|\int_{ B(x, \e) }  \big( T^n_x(V) -T^n_y(V) \big)(D \phi_\e)(x-y) \b(y)  d\lVert V\rVert_y}{\int_{\R^{2n}} \phi_\e(x-y)  d\lVert V\rVert_y + \e \lVert V \rVert}  \\
&& \leq  \e^{-1} K  \Big( \int_{B(x, \e) \times Gr(n,2n)} |S -T^n_x(V)|^2 dV(\xi, S)  \Big)^{\frac{1}{2}} \bigg( {\int_{B(x, \e)} \phi_\e(x-y) d\lVert V\rVert_y}  \bigg)^{-\frac{1}{2}} ||\b ||_{L^\infty({V } )} \\
&& \leq   C \e^{-1}   \Big( \int_{B(x, \e) \times Gr(n,2n)} |S -T^n_x(V)|^2 dV(\xi, S)  \Big)^{\frac{1}{2}}  \bigg( {\int_{B(x, \e)} \phi_\e(x-y) d\lVert V\rVert_y}  \bigg)^{-\frac{1}{2}},
\end{eqnarray*}
where the constant $C$  depends on   $||\b ||_{L^\infty({V } )}$.

By a similar argument:
\begin{eqnarray*}
 &&\frac{|\int \big( \n_x \phi_\e(x-y) +  \n_y \phi_\e(x-y) \big) d\lVert V\rVert_y}{\int \phi_\e(x-y)  d\lVert V\rVert_y +\e \lVert V \rVert} \\ 
 &\leq&   C \e^{-1}   \Big( \int_{B(x, \e) \times Gr(n,2n)} |S -T^n_x(V)|^2 dV(\xi, S)  \Big)^{\frac{1}{2}}\bigg( {\int_{B(x, \e)} \phi_\e(x-y) d\lVert V\rVert_y}  \bigg)^{-\frac{1}{2}} .
\end{eqnarray*}
for  a constant $C$. Since $\sup | \b_\e | \leq ||\b ||_{L^\infty({V } )}$ we conclude from (\ref{equ:define-E}):
\begin{equation}
\label{equ:estimate-E}
|E(x, \e)| \leq  2 C \e^{-1}   \Big( \int_{B(x, \e) \times Gr(n,2n)} |S -T^n_x(V)|^2 dV(\xi, S)  \Big)^{\frac{1}{2}}\bigg( {\int_{B(x, \e)} \phi_\e(x-y) d\lVert V\rVert_y}  \bigg)^{-\frac{1}{2}},
\end{equation}
where the constant $C$  depends on   $||\b ||_{L^\infty({V } )}$. 
\end{proof}

\bigskip

Let $\{ V^\e(t): t \geq 0 \}$ be an $\e$-flow. Recall that for each $t>0$, $V^\e(t)$  is an integral lagrangian  $n$-varifold with Maslov index zero satisfying $(H_2)$.  We will denote mollification using the notation:
$$
\phi_\eta \star_{V^\e} H^\e = (H^\e)_\eta
$$
Note that mollification is independent of scale in the measure so that 
$$
\phi_\eta \star_{V^\e} H^\e = \phi_\eta \star_{\l V^\e} H^\e,
$$ 
for any scalar $\l > 0$.

Set 
$$
m_\e(t) = ||V^\e(t)||,
$$
the total mass of the measure $V^\e(t)$.  Observe that since each $V^\e(t)$ is generated by a one-parameter family of diffeomorphisms, we have that $m_\e(t)>0$ for each $\e,t$.

Normalize the measure $\lVert V^\e(t)\rVert$ to the probability measure $\frac{\lVert V^\e(t)\rVert}{m_\e(t)}$. Set:
$$
\l_\e(t) = || H^\e(t) ||_{L^2(\frac{V^\e(t)}{m_\e(t)})}
$$

We will need the following lemmas to estimate the error term (\ref{equ:define-E}).

\begin{lem}
\label{lem:sequence-Allard}
Given a sequence $\{ \e_i \}$ with $\e_i \to 0$ suppose $\lim_{i \to \infty} \l_{\e_i}(t)$ exists and $\lim_{i \to \infty} \l_{\e_i}(t) > 0$, possibly $= \infty$, then the rescaled varifolds $\bar{V}^{\e_i} = \frac{V^{\e_i}}{\l_{\e_i} m_{\e_i}}$ have:
\begin{enumerate}
\item uniformily bounded mass, \\
\item satisfy $(H_2)$ and  the bound:
$$
|| {\bar H}^{\e_i} ||_{L^2(\bar{V}^{\e_i})} = 1.
$$
\end{enumerate}
Therefore there is a subsequence, that we continue to denote $\{ \e_i \}$, such that,
$$
\bar{V}^{\e_i} \to \bar{V}^0,
$$
as varifolds, $\bar{V}^0$ is a lagrangian varifold satisfying $(H_2)$.
\end{lem}

\begin{proof}
The result follows from Theorem \ref{thm:compact-Hp-varifolds}.
\end{proof}

Observe that the error term $E(x, \e)$ is independent of the rescalings of the varifold $V^\e$.

\begin{lem}
\label{lem:dom-conv-to-E}
Given a sequence $\{ \e_i \}$ with $\e_i \to 0$ such that  $\lim_{i \to \infty} \l_{\e_i}(t)$ exists and $\lim_{i \to \infty} \l_{\e_i}(t) > 0$ (possibly $= \infty$) then there is a subsequence, which we continue to denote  $\{ \e_i \}$, such that:
$$
\lim_{ \e_i \to 0}  \int | E(x, \e_i) |^2 d \lVert{\bar{V}^{\e_i}}\rVert =0.
$$
\end{lem}

\begin{proof}
Introduce the function:
\begin{eqnarray} \label{equ:define-E2}
&&E(x, r,  \e) \\ \nonumber
&& = \frac{\int \big( \n_x \phi_r(x-y) +  \n_y \phi_r(x-y) \big)\b(y) d\lVert V^\e\rVert_y}{\int \phi_r(x-y)  d\lVert V^\e\rVert_y +r ||V^\e|| }  -\b_r(x) \frac{\int \big( \n_x  \phi_r(x-y) +\n_y  \phi_r(x-y)\big)d\lVert V^\e\rVert_y}{\int \phi_r(x-y)  d\lVert V^\e\rVert_y +r ||V^\e||} 
\end{eqnarray}
The function $E(x, r,  \e)$ is also independent of the rescalings of the varifold $V^\e$. Observe that:
$$
E(x, \e,  \e) = E(x, \e).
$$
The estimate (\ref{equ:error-estimate}) becomes:
\begin{equation}
\label{equ:error-estimate2}
 |E(x, r, \e)|^2 < c_1 r^{-2} \frac{ \int_{B(x, r) \times Gr(n,2n)} |S -T^n_x(V^\e)|^2 dV^\e(\xi, S)  }{ {\int_{B(x, r)} \phi_r(x-y) d\lVert V^\e_y\rVert}  }
 \end{equation}

For $r < 1$ and $n \geq 2$  we have:
\begin{eqnarray*}
r^{-2}    \int_{B(x, r) \times Gr(n,2n)} |S -T^n_x(V^\e)|^2 dV^\e(\xi, S) &\leq& r^{-2}    \int_{B(x, r) \times Gr(n,2n)} 2 ( |S|^2 + |T^n_x(V^\e)|^2 ) dV^\e(\xi, S)  \\
 &\leq&  4 r^{-2}    \int_{B(x, r)} d\lVert V^\e\rVert \\
  &\leq& 4   r^{-n}    \int_{B(x, r)} d\lVert V^\e\rVert  \\
\end{eqnarray*}
Thus,
$$
 | E(x,r, \e) |^2 \leq  c_1  \frac{r^{-n}    \int_{B(x, r)} d\bar{V}^{\e} }{\int_{B(x, r)} \phi_r(x-y) d\lVert \bar{V}^{\e}\rVert }
$$
For $0 < r \leq 1$ and $x \in \supp \lVert V^\e\rVert$, set:
$$
f(x, r, \e) = \frac{r^{-n}    \int_{B(x, r)} d\lVert{\bar V}^\e\rVert }{\int_{B(x, r)} \phi_r(x-y) d\lVert{\bar V}^\e\rVert_y}
$$

Fix $\e$. Then $f(x, r, \e)$ is continuous for $x \in \supp \lVert V^\e\rVert $ and $r  \in(0, 1]$ and for each $x \in \supp \lVert V^\e\rVert$:
$$
\lim_{r \to 0} f(x, r, \e) = 1
$$
Thus there exists a constant $\L=\L(\e)<\infty$ such that:
$$
\sup_{\supp V^\e \times (0, 1]} f(x, r, \e) <\L
$$
It follows that
\begin{equation}
\label{equ:sup-bound-E}
\sup_{\supp V^\e \times (0, 1]} | E(x, r, \e) |^2 \leq  c_1 \L.
\end{equation}
On the other hand by Theorem \ref{thm:tilt-excess-B-M} pointwise for a.e. point $x$ we have:
$$
\lim_{r \to 0}  | E(x,r, \e) |^2 = 0.
$$
Therefore by the dominated convergence theorem for each $\e > 0$:
$$
\lim_{r \to 0} \int | E(x,r, \e) |^2 d\lVert{\bar V}^\e\rVert_x =0.
$$
In particular, given a sequence $\{ r_j \}$ with $r_j \to 0$ we have for each $\e$:
$$
\lim_{r_j \to 0} \int | E(x,r_j, \e) |^2 d\lVert{\bar V}^\e\rVert_x =0.
$$

Given a sequence $\{ \e_i \}$ with $\e_i \to 0$ such that Lemma \ref{lem:sequence-Allard} applies to give $\bar{V}^{\e_i} \to \bar{V}^0$ we can use the previous reasoning to conclude:
$$
\lim_{r_j \to 0} \int | E(x,r_j, \e_i) |^2 d\lVert{\bar V}^{\e_i}\rVert_x =0,
$$
and

$$
\lim_{r_j \to 0} \int | E(x,r_j, 0) |^2 d\lVert{\bar V}^{0}\rVert_x =0,
$$
where $E(x,r_j,0)$ represents the error computed with respect to the varifold ${\bar V}^0$, that is,
\begin{align*}
&E(x,r_j,0)=\\
&\frac{\int \left(\nabla_x \phi_{r_j}(x-y)+\nabla_y\phi_{r_j}(x-y)\right)\beta(y)d\lVert{\bar V}^0\rVert_y}{\int\phi_{r_j}(x-y)d\lVert{\bar V}^0\rVert_y+r_j ||{\bar V}^0||}-\beta_{r_j}(x)\frac{\int \left(\nabla_x \phi_{r_j}(x-y)+\nabla_y\phi_{r_j}(x-y)\right)d\lVert{\bar V}^0\rVert_y}{\int\phi_{r_j}(x-y)d\lVert{\bar V}^0\rVert_y+r_j ||{\bar V}^0||}
\end{align*}
Therefore,
$$
\lim_{j \to \infty} \lim_{i \to \infty}  \int | E(x,r_j, \e_i) |^2 d{\bar V}^{\e_i}_x = \lim_{i \to \infty} \lim_{j \to \infty}  \int | E(x,r_j, \e_i) |^2 d{\bar V}^{\e_i}_x =0.
$$
Taking a diagonal sequence and setting $r_i = \e_i$ gives the result.
\end{proof}

\begin{lem}
\label{lem:sequence-Allard2}
Given a sequence $\{ \e_i \}$ with $\e_i \to 0$ suppose $\lim_{i \to \infty} \l_{\e_i}(t) =0$  then the rescaled varifolds $\hat{V}^{\e_i} = \frac{V^{\e_i}}{ m_{\e_i}}$ have:
\begin{enumerate}
\item uniformily bounded mass, \\
\item satisfy $(H_2)$ and the bound:
$$
|| {\hat H}^{\e_i} ||_{L^2(\hat{V}^{\e_i})} \leq 1.
$$
\end{enumerate}
Therefore there is a subsequence, that we continue to denote $\{ \e_i \}$, such that,
$$
\hat{V}^{\e_i} \to \hat{V}^0,
$$
as varifolds, $\hat{V}^0$ is a lagrangian varifold with vanishing mean curvature.
\end{lem}

\begin{proof}
Same as the proof of Lemma \ref{lem:sequence-Allard}.
\end{proof}

\begin{lem}
\label{lem:dom-conv-to-E2}
Given a sequence $\{ \e_i \}$ with $\e_i \to 0$ such that  $\lim_{i \to \infty} \l_{\e_i}(t)=0$ then there is a subsequence, which we continue to denote  $\{ \e_i \}$, such that:
$$
\lim_{ \e_i \to 0}  \int | E(x, \e_i) |^2 d \lVert{\hat{V}^{\e_i}}\rVert_x =0.
$$
\end{lem}

\begin{proof}
Same as the proof of Lemma \ref{lem:dom-conv-to-E}.
\end{proof}

Using the results and estimates of Section \ref{sec:epsilon-flows} the geometric quantities along the $\e$-flow  are Lipschitz in time on $[0,T]$ and continuous in $\e$ for $\e > 0$.
For a fixed $s \in [0,T]$ we consider three cases:
\begin{enumerate}
\item $\limsup_{\e \to 0}   \l_\e(s) > 0$, possibly $=\infty$.\\
\item $\limsup_{\e \to 0}   \l_\e(s) =0$ and $\liminf_{\e \to 0} m_\e(s) < \infty$\\
\item $\limsup_{\e \to 0}   \l_\e(s) =0$ and $\liminf_{\e \to 0} m_\e(s) = \infty$\\
\end{enumerate}

The main theorem of this section is:

\begin{thm}\label{thm:volumedecrease}
Let $\{ V^\e(t) : t \geq 0 \}$ be an $\e$-flow. For each $t > 0$, $V^\e(t)$  is an integral lagrangian  $n$-varifold with Maslov index zero satisfying $(H_2)$.  Suppose that the lagrangian angle $\b_t \in L^\infty(V^\e(t))$ has weak derivative $B_t \in L^2(V^\e(t))$. 

Fix $s \in [0,T]$. In cases (1) and (2) there is a sequence $\{ \e_i \}$ with $\e_i \to 0$ such that:
$$
\lim_{\e_i \to 0} \frac{d}{dt} \lVert V^{\e_i}(t) \rVert_{|_{t = s}} < 0
$$
In case (3) there is a sequence $\{ \e_i \}$ with $\e_i \to 0$ such that:
$$
\lim_{\e_i \to 0} \frac{d}{dt} \lVert V^{\e_i}(t) \rVert_{|_{t = s}} = 0
$$

In all cases  there is a sequence, that we continue to denote $\{ \e_i \}$, such that for each $s \in [0,T]$  and any $\sigma> 0$ there is an integer $I$ such that for $i >I$:
$$
\frac{d}{dt} \lVert V^{\e_i}(t) \rVert_{|_{t = s}} < \sigma
$$
Moreover $\sigma \to 0$ as $I \to \infty$.
 \end{thm}

The following Lemmas are essential to the proof of Theorem \ref{thm:volumedecrease}.

\begin{lem}
\label{thm:mollified-mean-curv1}
Given a sequence $\{ \e_i \}$ with $\e_i \to 0$ there is a subsequence, that we will continue to denote $\{ \e_i \}$, such that:
\begin{equation}
\label{equ:mollified-mean-curv1}
\lim_{\e_i \to 0} \int \big(\frac{H^{\e_i}}{\l_{\e_i}} \big)_{\e_i} \cdot \frac{H^{\e_i}}{\l_{\e_i}} d \big( \frac{V^{\e_i}}{m_{\e_i}} \big) = \int X \cdot X dV \neq 0,
\end{equation}
where $V$ is a Radon measure and  $X \in L^2(V)$ is vector-valued.
\end{lem}

\begin{proof}
\noindent  {\bf (i)} The set of Radon measures $\big\{ \frac{\lVert V^\e\rVert}{m_\e} \big\}$ is bounded and hence there is a sequence $\{ {\e_i} \}$ with ${\e_i} \to0$ such that  $\big\{ \frac{\lVert V^{\e_i}\rVert}{m_{\e_i}} \big\}$  converges weakly to a Radon measure $V$.

\medskip

\noindent  {\bf (ii)} Since $|| \frac{H^{\e}}{\l_{\e}} ||_{L^2(\frac{V^{\e}}{m_{\e}})} = 1$, it follows that for  ${\eta} > 0$ there is a constant $C(\eta)$ with $C(\eta) \to 1$ as $\eta \to 0$  such that:
$$
|| \big(\frac{H^{\e}}{\l_{\e}} \big)_{\eta} ||_{L^2(\frac{V^{\e}}{m_{\e}})} < C(\eta).
$$ 
Note that if $A$ is a set with $\lVert V^\e\rVert (A) =0$ then $\frac{H^{\e}}{\l_{\e}}(x) = 0$ for $x \in A$. That is, the vector-valued function $\frac{H^{\e}}{\l_{\e}}$ is supported on the support of the measure $\lVert V^\e\rVert$. Hence for any $\eta >0$:
$$
\int \big| \big( \frac{H^{\e}}{\l_{\e}} \big)_{\eta} {\big|}^2 dV \leq \int \big| \big( \frac{H^{\e}}{\l_{\e}} \big)_{\eta} {\big|}^2 d\big( \frac{\lVert V^{\e}\rVert}{m_{\e}} \big) < C(\eta).
$$
Therefore the vector-valued functions $ \big(\frac{H^{\e}}{\l_{\e}} \big)_{\eta} \in L^2(V)$ and satisfy $|| \big(\frac{H^{\e}}{\l_{\e}} \big)_{\eta} ||^2_{ L^2(V)} < C(\eta)$ . Hence there is a sequence  $\{ {\e_i} \}$ with $\e_i \to 0$ such that $ \big(\frac{H^{\e_i}}{\l_{\e_i}} \big)_{\eta}$  converges weakly in $L^2(V)$ to a vector-valued function $X_\eta \in L^2(V)$.

In particular, the vector-valued functions $ \big(\frac{H^{\e}}{\l_{\e}} \big)_\e$ are bounded in $L^2(V)$. Therefore, there is a sequence $\{ {\e_i} \}$ with $\e_i \to 0$ such that  $ \big(\frac{H^{\e_i}}{\l_{\e_i}} \big)_{\e_i}$  converges weakly in $L^2(V)$ to a vector-valued function $X \in L^2(V)$.

\medskip

\noindent {\bf (iii)} Combining (i) and (ii) there are sequences $\e_i$ and $\d_i$ such that for any $\eta > 0$:
$$
\lim_{{\e_i} \to 0, \; {\d_i} \to 0} \big( \frac{H^{\e_i}}{\l_{\e_i}} \big)_{\e_i} d \big( \frac{\lVert V^{\d_i}\rVert}{m_{\d_i}} \big) = X dV.
$$
and
$$
\lim_{{\e_i} \to 0, \; {\d_i} \to 0} \big( \frac{H^{\e_i}}{\l_{\e_i}} \big)_{\eta} d \big( \frac{\lVert V^{\d_i}\rVert}{m_{\d_i}} \big) = X_\eta dV.
$$
Taking a diagonal sequence we have:
$$
\lim_{{\e_i} \to 0} \big(\frac{H^{\e_i}}{\l_{\e_i}} \big)_{\e_i} d\big( \frac{\lVert V^{\e_i}\rVert}{m_{\e_i}} \big) = X dV.
$$
and
$$
\lim_{{\e_i} \to 0} \big(\frac{H^{\e_i}}{\l_{\e_i}} \big)_{\eta} d\big( \frac{\lVert V^{\e_i}\rVert}{m_{\e_i}} \big) = X_\eta dV.
$$

\medskip

\noindent {\bf (iv)} Using (ii) and (iii) above we have that there are sequences $\{ {\e_i} \}$ and $\{ {\d_i} \}$ such that:
$$
\lim_{{\e_i} \to 0,  \; {\d_i} \to 0} \big( \frac{H^{\e_i}}{\l_{\e_i}} \big)_{\e_i} \big( \frac{H^{\d_i}}{\l_{\d_i}} \big)_\eta d \big( \frac{\lVert V^{\d_i}\rVert}{m_{\d_i}} \big) = X \cdot X_\eta dV.
$$
Taking a diagonal sequence we have:
$$
\lim_{{\e_i} \to 0} \int \big( \frac{H^{\e_i}}{\l_{\e_i}} \big)_{\e_i} \big( \frac{H^{\e_i}}{\l_{\e_i}} \big)_\eta d \big( \frac{\lVert V^{\e_i}\rVert}{m_{\e_i}} \big) = \int X \cdot X_\eta dV.
$$

\medskip

\noindent {\bf (v)} From (ii) we have that $X_\eta \in L^2(V)$ and that there is a constant $C(\eta)$  with $C(\eta) \to 1$ as $\eta \to 0$ such that $|| X_\eta ||_{L^2(V)} < C(\eta)$. Hence, for small $\eta$, there is a constant $C$ independent of $\eta$ such that $|| X_\eta ||_{L^2(V)} < C$. Then there is  a sequence $\{ {\eta_i} \}$ and a vector-valued function $Y \in L^2(V)$ such that $\{ X_{\eta_i} \}$ converges weakly in $L^2(V)$ to $Y$ as ${\eta_i} \to 0$. Since $\big(\frac{H^\e}{\l_\e} \big)_\eta$ goes to 
$\frac{H^\e}{\l_\e}$ strongly in $L^2(V)$ as $\eta \to 0$, it follows that:
$$
\lim_{{\e_i} \to 0} \int \big( \frac{H^{\e_i}}{\l_{\e_i}} \big)_{\e_i}   \frac{H^{\e_i}}{\l_{\e_i}} d \big( \frac{\lVert V^{\e_i}\rVert}{m_{\e_i}} \big) = \int X \cdot Y dV.
$$

\medskip

\noindent {\bf (vi)} From (iii) there is a sequence $\{ {\d_i} \}$ such that:
$$
\lim_{{\d_i} \to 0} \big(\frac{H^{\d_i}}{\l_{\d_i}} \big)_\eta d\big( \frac{\lVert V^{\d_i}\rVert}{m_{\d_i}} \big) = X_\eta dV.
$$
Taking the diagonal sequence $\eta = \d_i$ we have that 
$$
\lim_{\eta \to 0} X_\eta = X.
$$
Hence $Y = X$. We conclude that:
$$
\lim_{{{\e_i}} \to 0} \int \big(\frac{H^{\e_i}}{\l_{\e_i}} \big)_{\e_i} \cdot \frac{H^{\e_i}}{\l_{\e_i}} d \big( \frac{V^{\e_i}}{m_{\e_i}} ) = \int X \cdot X dV
$$

\medskip

\noindent {\bf (vii)} We have:
\begin{eqnarray*}
|| X ||_{L^2(V)} & = & \lim_{ {\e_i} \to 0} || \big(\frac{H^{\e_i}}{\l_{\e_i}} \big)_{\e_i} ||^2_{L^2( \frac{V^{\e_i}}{m_{\e_i}} )} \\
& = & \lim_{ {\e_i} \to 0} ||  \frac{H^{\e_i}}{\l_{\e_i}}  ||^2_{L^2 ( \frac{V^{\e_i}}{m_{\e_i}} )} \\
& = & 1
\end{eqnarray*}
Hence, $\int X \cdot X dV \neq 0$.

This completes the proof of (\ref{equ:mollified-mean-curv1}). 
\end{proof}

\begin{lem}
\label{cor:mollified-mean-curv1}
Given a sequence $\{ \e_i \}$ with $\e_i \to 0$ there is a subsequence, that we will continue to denote $\{ \e_i \}$, such that:
$$
\liminf_{{\e_i} \to 0}  \int \big( \frac{JH^{\e_i}}{\l_{\e_i}} \big)_{\e_i} \cdot \frac{JH^{\e_i}}{\l_{\e_i}} d \big( \frac{\lVert V^{\e_i}\rVert}{m_{\e_i}} ) = \int JX \cdot JX dV \neq 0,
$$
where $J$ is the complex structure on $\R^{2n}$.
\end{lem}

\bigskip

\begin{lem}
\label{thm:mollified-mean-curv2}
Given a sequence $\{ \e_i \}$ with $\e_i \to 0$ there is a subsequence, that we will continue to denote $\{ \e_i \}$, such that:
\begin{equation}
\label{equ:mollified-mean-curv2}
\lim_{{\e_i} \to 0} \int ( \b^{\e_i})_{\e_i} \big(\frac{ H^{\e_i}}{\l_{\e_i}} \big)_{\e_i} \cdot \frac{JH^{\e_i}}{\l_{\e_i}} d \big( \frac{\lVert V^{\e_i}\rVert}{m_{\e_i}} ) = \int\tilde{\b} X \cdot JX dV = 0,
\end{equation}
where $V$ and  $X \in L^2(V)$ are the same measure and vector-valued function as in Theorem \ref{thm:mollified-mean-curv1} and $\tilde{\b} \in L^\infty(V)$.
\end{lem}

\begin{proof}
The proof is similar to the proof of Lemma \ref{thm:mollified-mean-curv1} except that the smooth function $( \b^\e)_\e$ must be taken into account. Observe that $|( \b^\e)_\e|  < A $ for all $\e$ and therefore there is a sequence  ${\e_i}$  such that $\{ ( \b^{\e_i})_{\e_i} \}$ converges in $L^\infty(V)$ to $\tilde{\b} \in {L^\infty(V)}$ with $||\tilde{\b}||_{L^\infty(V)} < A$.
Using the proof of Lemma \ref{thm:mollified-mean-curv1} there is a sequence ${\e_i}$ and vector-valued functions $X, X_\eta \in L^2(V)$ such that:
$$
\big(\frac{ H^{\e_i}}{\l_{\e_i}} \big)_{\e_i}  \rightharpoonup X, 
$$
 and
$$
\big(\frac{ H^{\e_i}}{\l_{\e_i}} \big)_\eta  \rightharpoonup X_\eta ,
$$
weakly in $L^2(V)$. Therefore there are sequences $\{ \e_i \}$,  $\{ {\t_i} \}$ and $\{ {\d_i} \}$ such that:
$$
\lim_{{\e_i} \to 0,  \; {\t_i} \to 0, \; {\d_i} \to 0} \big( \frac{H^{\e_i}}{\l_{\e_i}} \big)_{\e_i} \big( \frac{JH^{\e_i}}{\l_{\e_i}} \big)_\eta (\b^{\t_i})_{\t_i} d \big( \frac{\lVert V^{\d_i}\rVert}{m_{\d_i}} \big) = \tilde{\b} X \cdot JX_\eta dV.
$$
Taking a diagonal sequence we have:
$$
\lim_{{\e_i} \to 0} \int \big( \frac{H^{\e_i}}{\l_{\e_i}} \big)_{\e_i} \big( \frac{JH^{\e_i}}{\l_{\e_i}} \big)_\eta (\b^{\e_i})_{\e_i} d \big( \frac{\lVert V^{\e_i}\rVert}{m_{\e_i}} \big) = \int \tilde{\b} X \cdot JX_\eta dV.
$$
Using $\lim_{\eta \to 0} X_\eta = X$, the result follows. 
\end{proof}

\bigskip

\begin{lem}
\label{thm:mollified-mean-curv3}
Given a sequence $\{ \e_i \}$ with $\e_i \to 0$ there is a subsequence, that we will continue to denote $\{ \e_i \}$, such that:
\begin{equation}
\label{equ:mollified-mean-curv3}
\lim_{{\e_i} \to 0} \int \big(\frac{\b^{\e_i} H^{\e_i}}{\l_{\e_i}} \big)_{\e_i} \cdot \frac{JH^{\e_i}}{\l_{\e_i}} d \big( \frac{\lVert V^{\e_i}\rVert}{m_{\e_i}} \big) = \int \hat{\b} X \cdot JX dV = 0,
\end{equation}
where $V$ and  $X \in L^2(V)$ are the same measure and vector-valued function as in Theorem \ref{thm:mollified-mean-curv1} and $\hat{\b} \in L^\infty(V)$.
\end{lem}

\begin{proof}
Using Theorem \ref{thm:boundonbeta} the set $\{ \big(\frac{\b^\e H^\e}{\l_\e} \big)_\e d \big( \frac{\lVert V^\e\rVert}{m_\e} ) \}$ of Radon measures is bounded and therefore there is a sequence $\e_i$ such that $\{ \big(\frac{\b^{\e_i} H^{\e_i}}{\l_{\e_i}} \big)_{\e_i} d \big( \frac{V^{\e_i}}{m_{\e_i}} ) \}$ converges weakly to the vector-valued signed measure $\tau$. We have already shown, in the proof of Lemma \ref{thm:mollified-mean-curv1}, that there is a sequence $\{ \e_i \}$ and a vector-valued functions $X  \in L^2(V)$  such that:
$$
\big(\frac{ H^{\e_i}}{\l_{\e_i}} \big)_{\e_i}  \rightharpoonup X, 
$$
weakly in $L^2(V)$. Therefore the sequence of the Radon measures $\{ \big(\frac{ H^{\e_i}}{\l_{\e_i}} \big)_{\e_i} d \big( \frac{V^{\e_i}}{m_{\e_i}} )\big \}$ converges weakly as Radon measures to the measure $\mu = X dV$. The measure $\tau$ is absolutely continuous with respect to the measure $\mu$ and therefore there is an integrable function $\hat{\b}$ such that:
$$
\tau = \hat{\b} \mu = \hat{\b} X dV
$$
Observe that $| \big(\frac{\b^\e H^\e}{\l_\e} \big)_\e|  < A |\big(\frac{ H^\e}{\l_\e} \big)_\e|$ for all $\e$ and therefore $||\hat{\b}||_{L^\infty(V)} < A$.
Also from the proof of Lemma \ref{thm:mollified-mean-curv1} there is a a sequence $\{ \e_i \}$ and  vector-valued functions $X_\eta  \in L^2(V)$  such that:
$$
\big(\frac{ H^{\e_i}}{\l_{\e_i}} \big)_\eta  \rightharpoonup X_\eta, 
$$
weakly in $L^2(V)$. Therefore, as in the previous proofs,  we conclude that:
$$
\lim_{{\e_i} \to 0} \int \big( \frac{\b^{\e_i} H^{\e_i}}{\l_{\e_i}} \big)_{\e_i} \big( \frac{JH^{\e_i}}{\l_{\e_i}} \big)_\eta  d \big( \frac{\lVert V^{\e_i}\rVert}{m_{\e_i}} \big) = \int \hat{\b} X \cdot JX_\eta dV.
$$
Using $\lim_{\eta \to 0} X_\eta = X$, the result follows. 
\end{proof}

\bigskip

\noindent{\it Proof of Theorem \ref{thm:volumedecrease}:}

Denote the $\e$-flow beginning with $V$ by $\{ V^\e(t) \}$, where $V^\e(t)$ is the varifold along the $\e$-flow at time $t$. The flow vector field at time $t$ is $J D \b_\e$, where $\b$ is the lagrangian angle along $V^\e_t$ and $\b_\e$ is its mollification. Denote the mean curvature vector field along the $\e$-flow by $H^\e$. 

	We have by the first variation,
		\begin{eqnarray} \nonumber
\frac{d}{dt} \lVert V^\e_t \rVert  &=& -\int JD \b_\e \cdot H^\e d\lVert V^\e\rVert = \int D \b_\e \cdot JH^\e d\lVert V^\e\rVert \\  \nonumber
&=& \int \n \b_\e \cdot JH^\e d\lVert V^\e\rVert \\  \nonumber
 &=& \int\big(B_\e -\b_\e (H^\e)_\e + E(-, \e) \big) \cdot JH^\e d\lVert V^\e\rVert \\  \nonumber
\label{equ:deriv-of-volume}
 &=&-\int (JH^\e)_\e \cdot JH^\e d\lVert V^\e\rVert + \int \big( (\b H^\e)_\e -\b_\e (H^\e)_\e + E(-, \e) \big) \cdot JH^\e d\lVert V^\e\rVert \\
		\end{eqnarray}
The final equality follows using Theorem \ref{thm:harvey-lawson-varifold}. 
Clearly, if $H^\e(t) = 0$, or equivalently $\l_\e(t) =0$, then
$$
\frac{d}{dt} \lVert V^\e(t) \rVert = 0.
$$

Fix $t$. In Case (1) multiply (\ref{equ:deriv-of-volume}) by ${\l_{\e}}^{-2}(t) m_{{\e}}^{-1}(t) > 0$ to get:
\begin{eqnarray*}
&&\l_{{\e}}^{-2} m_{{\e}}^{-1} \frac{d}{dt} \lVert V^{{\e}}(t) \rVert \\
&&= -\int (\frac{JH^{{\e}}}{\l_{{\e}}})_{{\e}} \cdot \frac{JH^{{\e}}}{\l_{{\e}}} d( \frac{\lVert V^{{\e}}\rVert}{m_{{\e}}}) + \int \big( (\b \frac{JH^{{\e}}}{\l_{{\e}}})_{{\e}} -\b_{{\e}} (\frac{JH^{{\e}}}{\l_{{\e}}})_{{\e}} + \frac{E(-, {{\e}})}{\l_{{\e}}} \big) \cdot \frac{JH^{{\e}}}{\l_{{\e}}} d( \frac{\lVert V^{{\e}}\rVert}{m_{{\e}}})
\end{eqnarray*}
Choose a sequence $\{ \e_i \}$ such that  $\lim_{\e_i \to 0}  {\l_{\e_i}(t)} > 0$.
To estimate the term:
$$
\int \frac{E(-, {{\e}})}{\l_{{\e}}} ) \cdot \frac{JH^{{\e}}}{\l_{{\e}}} d( \frac{\lVert V^{{\e}}\rVert}{m_{{\e}}}),
$$
use the H{\"o}lder inequality to give:
$$
\big| \int \frac{E(-, {{\e}})}{\l_{{\e}}}  \cdot \frac{JH^{{\e}}}{\l_{{\e}}} d( \frac{\lVert V^{{\e}}\rVert}{m_{{\e}}}) \big| \leq || \frac{E(-, {{\e}})}{\l_{{\e}}} ||_{L^2( \frac{V^\e}{m_\e})} ||  \frac{H^\e}{\l_\e} ||_{L^2( \frac{V^\e}{m_\e})}
$$
Then
$$
 || \frac{E(-, {{\e_i}})}{\l_{{\e_i}}} ||^2_{L^2( \frac{V^{\e_i}}{m_{\e_i}})} = \int | E(x, \e_i) |^2 d\lVert {\bar V}^{\e_i}\rVert_x
 $$
For a suitable subsequence of $\{ \e_i \}$ this term goes to zero as ${\e_i} \to 0$ by Lemma \ref{lem:dom-conv-to-E}.
Applying Lemma \ref{cor:mollified-mean-curv1}, Lemma \ref{thm:mollified-mean-curv2} and Lemma \ref{thm:mollified-mean-curv3} and choosing a subsequence (that we continue to denote $\{ \e_i \}$) as ${\e_i} \to 0$ the right hand side goes to $-\int X \cdot X dV < 0$. Therefore since ${\l_{\e_i}}^{-2}(t) m_{\e_i}^{-1}(t) > 0$ we conclude that:
$$
\lim_{\e_i \to 0} \frac{d}{dt} \lVert V^{\e_i}(t) \rVert < 0.
$$

Fix $t$. In Case (2) we rewrite (\ref{equ:deriv-of-volume}) as
\begin{eqnarray} \label{equ:deriv-of-volume-scaled}
&&\frac{d}{dt} \ln \big( m_\e(t) \big) \\ \nonumber
&& = -\int (JH^\e)_\e \cdot JH^\e d \big( \frac{\lVert V^\e\rVert}{m_\e} \big) + \int \big( (\b H^\e)_\e -\b_\e (H^\e)_\e + E(-, \e) \big) \cdot JH^\e d \big( \frac{\lVert V^\e\rVert}{m_\e} \big)
\end{eqnarray}
Choose a sequence $\{ \e_i \}$ such that $\lim_{\e_i \to 0}  {\l_{\e_i}(t)} = 0$ (or equivalently, $\lim_{\e_i \to 0} || H^{\e_i}(t) ||^2_{L^2(\frac{V^{\e_i}(t)}{m_{\e_i}(t)})} = 0$) and such that $\lim_{\e_i \to 0}  {m_{\e_i}(t)} < \infty$
From this it follows that
$$
\lim_{\e_i \to 0} \int (JH^{\e_i})_{\e_i} \cdot JH^{\e_i} d \big( \frac{V^{\e_i}}{m_{\e_i}} \big) =0,
$$
$$
 \lim_{\e_i \to 0}  \int \big( (\b H^{\e_i})_{\e_i} -\b_{\e_i} (H^{\e_i})_{\e_i}  \big) \cdot JH^{\e_i} d \big( \frac{V^{\e_i}}{m_{\e_i}} \big) = 0.
$$ 
 The remaining term can be estimated using the Holder inequality and (\ref{lem:dom-conv-to-E2}). We have:
\begin{eqnarray*}
\big| \int  E(-, {\e_i})  \cdot JH^{\e_i} d \big( \frac{V^{\e_i}}{m_{\e_i}} \big) \big| &\leq& || E(-, {\e_i})  ||_{L^2 \big(\frac{V^{\e_i}(t)}{m_{\e_i}(t)}\big)} || H^{\e_i}  ||_{L^2 \big(\frac{V^{\e_i}(t)}{m_{\e_i}(t)}\big)} \\
&\leq & \big( \int | E(x, \e_i) |^2 d \big( \frac{\lVert V^{\e_i}(t)\rVert}{m_{\e_i}(t)}\big) \big)^{\frac{1}{2}}    {\l_{\e_i}}
\end{eqnarray*}  
This term goes to zero since both $\int | E(x, \e_i) |^2 d\hat{V}^{\e_i}$ and  ${\l_{\e_i}}$ go to zero as $\e_i \to 0$. Since $\lim_{\e_i \to 0}  {m_{\e_i}(t)} < \infty$ it follows that 
$$
\lim_{{\e_i} \to 0} \frac{d}{dt}  \lVert V^{{\e_i}}(t) \rVert =0.
$$

Fix $t$. In Case (3) choose a sequence $\{ \e_i \}$ such that $\lim_{\e_i \to 0}  {\l_{\e_i}(t)} = 0$ (or equivalently, $\lim_{\e_i \to 0} || H^{\e_i}(t) ||^2_{L^2(\frac{V^{\e_i}(t)}{m_{\e_i}(t)})} = 0$). Repeat the analysis of Case (2) to conclude:
$$
\lim_{{\e_i} \to 0} \frac{d}{dt} \ln \big( \lVert V^{{\e_i}}(t) \rVert \big)=0.
$$

Fix $t$. All three cases show that there is a sequence $\{ \e_i \}$ such that:
\begin{equation}
\label{equ:ln-volume-limit}
\lim_{{\e_i} \to 0} \frac{d}{dt} \ln \big( \lVert V^{{\e_i}}(t) \rVert \big) \leq 0.
\end{equation}
Let ${\cal S}$ be a countable dense subset of $[0,T]$. Choosing a diagonal subsequence we can conclude that there is a sequence $\{ \e_i \}$ with $\e_i \to 0$ such that for all $t \in {\cal S}$ (\ref{equ:ln-volume-limit}) holds. Using the continuity of the right hand side of (\ref{equ:deriv-of-volume-scaled}) in $\e$ and $t$ we conclude that (\ref{equ:ln-volume-limit}) holds for all $t \in [0, T]$.
Hence given $\sigma >0$ for all $t \in [0,T]$ there exists an integer $I$ such that for $i > I$ we have:
\begin{equation}
\label{equ:ln-volume-est}
 \frac{d}{dt} \ln \lVert V^{\e_i}(t) \rVert  < \sigma.
\end{equation}
Integrating (\ref{equ:ln-volume-est}) over $[0,T]$ we conclude the volume (or mass) estimate for $i > I$:
\begin{equation}
\label{equ:volume-est}
 \lVert V^{\e_i}(t) \rVert  < \exp(\sigma T) ||V(0)||
\end{equation}

In light of the volume estimate (\ref{equ:volume-est}) we can return to Case (3) above and observe that it does not occur. Hence we conclude that for all $t \in [0,T]$ there is a sequence $\{ \e_i \}$ with $\e_i \to 0$ such that:
$$
\lim_{{\e_i} \to 0} \frac{d}{dt}   \lVert V^{{\e_i}}(t) \rVert  \leq 0.
$$

\bigskip

\begin{cor}\label{cor:constructlimitflow}
For each $t \in [0, T]$ there is a sequence  $\{\e_i \}$ with $\e_i \to 0$ such that the lagrangian rectifiable varifolds $V^{\e_i}(t)$ coverge weakly as Radon measures to the  lagrangian rectifiable varifold $V(t)$ with:
$$
 \rVert V(t) \rVert = \liminf_{\e \to 0}  \rVert V^{\e}(t) \rVert.
$$
The volume  $\rVert V(t) \rVert$ is a non-increasing, lower semi-continuous function of $t$ such that for a.e. $t$:
$$
\frac{d}{dt} \lVert V(t) \rVert \leq 0.
$$
\end{cor}

\begin{rem}
It could be the case that $\lVert V^\e(t)\rVert\rightarrow 0$ as $\e\rightarrow 0$; in this case $\lVert V(s)\rVert=0$ for all $s\in [0,T]$ with $s\geq t$.
\end{rem}

\bigskip

\begin{thm}\label{thm:betacontrol}
Under the same hypotheses as Theorem \ref{thm:volumedecrease}, for each even integer $p \geq 2$,
on a fixed interval $[0, T]$, given any $\sigma> 0$ there is an $\e_0 > 0$ such that for $\e < \e_0$:
$$
\frac{d}{dt} \lVert \beta^{\e}(t) \rVert_{L^p(V^{\e}(t))} < \sigma
$$
In particular, for any $\e < \e_0$ and any $s, t \in [0, T]$ with $t > s$:
\begin{equation}
\label{equ:beta-lipschitz}
 \frac{\lVert \beta^{\e}(t) \rVert_{L^p(V^{\e}(t))}  -   \lVert \beta^{\e}(s) \rVert_{L^p(V^{\e}(s))}} {t-s} < \sigma..
\end{equation}
Moreover $\sigma \to 0$ as $\e_0 \to 0$.
\end{thm}

\begin{proof}
	 For each $\e$-flow,
		\begin{equation}
			\begin{aligned}
			\frac{d}{dt}\int (\beta^\e)^p d{\lVert V^\e\rVert}=&-\int p (\beta^\e)^{p-1} \div JJD(\beta^\e)_\e  d{\lVert V^\e\rVert} + \int(\beta^\e)^p \div JD(\beta^\e)_\e  d{\lVert V^\e\rVert}\\
			=&p\int(\beta^\e)^{p-1}\div D(\beta^\e)_\e  d{\lVert V^\e\rVert}+\int(\beta^\e)^p\div JD(\beta^\e)_\e  d{\lVert V^\e\rVert}
			\end{aligned}
		\end{equation} 
	We now use Propositions \ref{divergenceformula} and \ref{powerrule} to trade the divergence operator for weak derivatives:
		\begin{equation}
			\begin{aligned}
			p\int(\beta^\e)^{p-1}\div D(\beta^\e)_\e  d{\lVert V^\e\rVert}&=-p\int \left[(p-1)(\beta^\e)^{p-2}B^\e -(p-2)(\beta^\e)^{p-1}H^\e \right]\cdot D(\beta^\e)_\e d{\lVert V^\e\rVert}\\
			\int(\beta^\e)^p\div JD(\beta^\e)_\e d{\lVert V^\e\rVert}&=-\int\left[p(\beta^\e)^{p-1}B^\e -(p-1)(\beta^\e)^pH^\e \right]\cdot JD(\beta^\e)_\e d{\lVert V^\e\rVert}
			\end{aligned}
		\end{equation} 
	Now we apply the Harvey-Lawson formula $B^\e =\nabla \beta^\e +\beta^\e  H^\e =-JH^\e +\beta^\e H^\e$ in each term:
		\begin{eqnarray*}
			&&p\int(\beta^\e)^{p-1}\div D(\beta^\e)_\e  d{\lVert V^\e\rVert}\\
			&&=-p\int \left[(p-1)(\beta^\e)^{p-2}\left(-JH^\e +(\beta^\e) H^\e \right)-(p-2)(\beta^\e)^{p-1}H^\e \right]\cdot D(\beta^\e)_\e  d{\lVert V^\e\rVert}\\
			&&=-p\int \left[-(p-1)(\beta^\e)^{p-2}JH^\e +(\beta^\e)^{p-1}H^\e \right]\cdot D(\beta^\e)_\e  d{\lVert V^\e\rVert}\\
			&&=-p(p-1)\int(\beta^\e)^{p-2}H^\e \cdot JD(\beta^\e)_\e  d{\lVert V^\e\rVert}-p\int(\beta^\e)^{p-1}H^\e \cdot D(\beta^\e)_\e  d{\lVert V^\e\rVert}\\
			\end{eqnarray*}
			\begin{eqnarray*}
			&&\int(\beta^\e)^p\div JD(\beta^\e)_\e  d{\lVert V^\e\rVert}\\
			&&=-\int\left[p(\beta^\e)^{p-1}\left(-JH^\e +(\beta^\e) H^\e \right)-(p-1)(\beta^\e)^pH^\e \right]\cdot JD(\beta^\e)_\e  d{\lVert V^\e\rVert}\\
			&&=-\int\left[-p(\beta^\e)^{p-1}JH^\e +(\beta^\e)^pH^\e \right]\cdot JD(\beta^\e)_\e  d{\lVert V^\e\rVert}\\
			&&=p\int (\beta^\e)^{p-1}H^\e \cdot D(\beta^\e)_\e  d{\lVert V^\e\rVert}-\int(\beta^\e)^pH^\e \cdot JD(\beta^\e)_\e  d{\lVert V^\e\rVert}
			\end{eqnarray*}
	Adding these terms together we obtain
	\begin{equation}
		\frac{d}{dt}\int(\beta^\e)^pd{\lVert V^\e\rVert}=-\int\left(p(p-1)(\beta^\e)^{p-2}+(\beta^\e)^p\right)H^\e \cdot JD(\beta^\e)_\e  d{\lVert V^\e\rVert}
	\end{equation}

	Now by Brakke's perpendicularity result, we can write $H^\e \cdot J\nabla(\beta^\e)_\e $ in this integrand, so that we obtain 
	\begin{equation}
		\frac{d}{dt}\int(\beta^\e)^pd{\lVert V^\e\rVert}=-\int\left(p(p-1)(\beta^\e)^{p-2}+(\beta^\e)^p\right)H^\e \cdot J\nabla(\beta^\e)_\e  d{\lVert V^\e\rVert}
	\end{equation}
	Both terms are of the form $\int (\beta^\e)^q H^\e \cdot J\nabla(\beta^\e)_\e  d{\lVert V^\e\rVert}$.
	
	\begin{eqnarray*}
			&&\int(\beta^\e)^qH^\e \cdot J\nabla(\beta^\e)_\e  d{\lVert V^\e\rVert}\\
			&&=\int(\beta^\e)^q H^\e \cdot J\left((B^\e)_\e  -(\beta^\e)_\e  (H^\e)_\e  +E\right)d{\lVert V^\e\rVert}\\
			&&=\int(\beta^\e)^qH^\e \cdot J\left((\nabla (\beta^\e)+(\beta^\e H^\e )_\e -(\beta^\e)_\e  (H^\e)_\e +E\right)d{\lVert V^\e\rVert}\\
			&&=\int(\beta^\e)^q H^\e \cdot (H^\e)_\e  d{\lVert V^\e\rVert}-\int(\beta^\e)^q JH^\e \cdot (\beta^\e H^\e)_\e -(\beta^\e)_\e  (H^\e)_\e+ E )d{\lVert V^\e\rVert}
	\end{eqnarray*}

Hence,
\begin{eqnarray} \nonumber
&&\frac{d}{dt}\int(\beta^\e)^pd{\lVert V^\e\rVert}\\ \nonumber
&&=-\int\left(p(p-1)(\beta^\e)^{p-2}+(\beta^\e)^p\right)H^\e \cdot (H^\e)_\e d{\lVert V^\e\rVert} \\  \nonumber
\label{equ:integral-bound-H}
&& + \int \left(p(p-1)(\beta^\e)^{p-2}+(\beta^\e)^p\right) JH^\e \cdot (\beta^\e H^\e )_\e -(\beta^\e)_\e (H^\e)_\e + E )d{\lVert V^\e\rVert}\\ 
\end{eqnarray}

Using (\ref{equ:integral-bound-H}) if $H^\e(t) = 0$ then
$$
\frac{d}{dt} \lVert \beta^\e(t) \rVert_{L^p(V^\e(t))} = 0
$$
Using the uniform sup bound on $\b^\e(t)$ and (\ref{equ:integral-bound-H}) the argument proceeds exactly as in the proof of Theorem \ref{thm:volumedecrease}.
The result follows.

\end{proof}

\begin{cor}
\label{cor:beta-lipschitz}
Under the same hypotheses as Theorem \ref{thm:betacontrol}, the conclusions hold for $\lVert \beta^{\e}(t) \rVert_{L^\infty(V^{\e}(t))}$. In particular, given any $\sigma> 0$ there is an $\e_0 > 0$ such that for $\e < \e_0$ and any $s, t \in [0, T]$ with $t > s$:
\begin{equation}
\label{equ:beta-lipschitz}
 \frac{\lVert \beta^{\e}(t) \rVert_{L^\infty(V^{\e}(t))}  -   \lVert \beta^{\e}(s) \rVert_{L^\infty(V^{\e}(s))}} {t-s} < \sigma..
\end{equation}
Moreover $\sigma \to 0$ as $\e_0 \to 0$.
\end{cor}

\begin{cor}
\label{cor:norm-beta-decreasing}
For each $t \in [0, T]$ there is a function $\b(t) \in L^\infty(V(t))$ that is the scalar lift of the lagrangian angle on $V(t)$. The norms $\lVert \beta(t) \rVert_{L^\infty(V(t))}$ are non-increasing, lower semi-continuous functions of $t$ and for a.e. $t$
$$
\frac{d}{dt} \lVert \beta(t)\rVert_{L^\infty(V(t))} \leq 0
$$
\end{cor}

\begin{proof}
The existence of the scalar lifts of the lagrangian angle follows from Theorem \ref{thm:convergence-beta}. The non-increasing, lower semi-continuous statement and derivative bound follow from Corollary \ref{cor:beta-lipschitz}
\end{proof}

\bigskip

\section{Extending the Flow and First Variation Control}
In this section we show how to extend the lagrangian varifold  $\{ V(t) \}$, constructed above on a compact time interval $[0,T]$, to all $t>0$. We show that for a.e. $t>0$ the lagrangian varifold $V(t)$ satisfies $(H_2)$ and has Maslov index zero. Unfortunately we do not have uniform bounds on the mean curvature of $V(t)$, however, Theorem \ref{thm:good-times-MC-bounds} and Theorem \ref{thm:H2-bounded-infinity} below are partial replacements.

\bigskip

Recall that $\l_{\e}(t) = || {H^{\e}}(t) ||_{L^2(\frac{V^{\e}(t)}{m_\e(t)})}$. Set:
$$
{\cal S}_M = \{ t \in [0,T]: \limsup_{\e \to 0} \l_{\e}(t) > M \}
$$
$$
{\cal S}_\infty = \{ t \in [0,T]: \limsup_{\e \to 0} \l_{\e}(t) = \infty \}
$$

\begin{thm}
\label{thm:first-bound-MC} Suppose we have $\inf_{[0,T]} \lVert V(t)\rVert\geq \alpha>0$. 
The one dimensional Lebesgue measure of ${\cal S}_M$, denoted $m({\cal S}_M)$ satisfies the bound:
$$
m({\cal S}_M) \leq M^{-2} \a^{-1} 2 || V(0) ||,
$$

 Consequently, 
 $$
 m({\cal S}_\infty) = 0.
 $$
\end{thm}

\begin{proof}
First assume that for all $\e >0$ and $t \in [0,T]$ that $|| H^\e(t) ||_{L^2(V^\e(t))} \neq 0$. Using the uniform sup bound on $\b^\e$ and the behavior of mollification as $\e \to 0$ we have for each $t \in [0,T]$ that there exists an $\e(t) > 0$ such that for $0 < \e < \e(t)$:

\begin{eqnarray} \nonumber
&&\big| \int \Big( \big(\b^\e \frac{H^\e(t)}{\l_\e(t)}\big)_\e -(\b^\e)_\e \big( \frac{H^\e(t)}{\l_\e(t)}\big)_\e + E(-, \e) \Big) \cdot J\big(\frac{H^\e(t)}{\l_\e(t)} \big) d\big(\frac{\lVert V^\e(t)\rVert}{m_\e(t)}\big) \big| \\
\label{equ:scaled-inequality}
&&\leq \frac{1}{2} \int \big(J\big(\frac{H^\e(t)}{\l_\e(t)}\big)\big)_\e \cdot J\big(\frac{H^\e(t)}{\l_\e(t)} \big) d\big(\frac{\lVert V^\e(t)\rVert}{m_\e(t)}\big)
\end{eqnarray}
Using the continuity in $\e> 0$ and $t$ of $H^\e(t)$ and $\l_\e(t)$ we can find an $\overline{\e}$ such that the  inequality (\ref{equ:scaled-inequality}) holds for any $0 < \e < \overline{\e}$ and all $t \in [0,T]$. Hence for all $t \in [0,T]$ and all $0 < \e < \overline{\e}$ we have:
\begin{eqnarray} \nonumber
&&\big| \int \Big( \big(\b^\e H^\e(t) \big)_\e -(\b^\e)_\e \big( H^\e(t) \big)_\e + E(-, \e) \Big) \cdot J\big(H^\e(t) \big) d\lVert V^\e(t)\rVert \big| \\
\label{equ:unscaled-inequality}
&&\leq \frac{1}{2} \int \big(J\big(H^\e(t) \big)\big)_\e \cdot J\big(H^\e(t) \big) d\lVert V^\e(t)\rVert
\end{eqnarray}
This inequality includes the cases when $|| H^\e(t) ||_{L^2(V^\e(t))} = 0$. Hence we have for all $t \in [0,T]$ and all $0 < \e < \overline{\e}$:
$$
 \frac{1}{2} \int \big(JH^\e(t) \big)_\e \cdot JH^\e(t) d\lVert V^\e(t)\rVert \leq - \frac{d}{dt} \lVert V^{\e}(t) \rVert
$$
 Integrate to derive  for $0 < \e < \overline{\e}$:
\begin{equation}
\label{equ:MC-bound}
 \int_0^T  \frac{1}{2} \int \big(JH^\e(t) \big)_\e \cdot JH^\e(t) d\lVert V^\e(t)\rVert dt \leq - \int_0^T  \frac{d}{dt} \lVert V^{\e}(t) \rVert dt < || V^\e(0) ||.
\end{equation}
It follows that for for $0 < \e < \overline{\e}$:
\begin{equation}
\label{equ:inequ-on-small-set}
 \int_{{\cal S}_M}  \int \big(JH^\e(t) \big)_\e \cdot JH^\e(t) d\lVert V^\e(t)\rVert dt  < 2 || V^{\e}(0) ||.
\end{equation}
Set:
$$
\l_\e({{\cal S}_M}) = \inf_{t \in {{\cal S}_{M}}} \l_{\e}(t).
$$
$$
m_\e({{\cal S}_M}) = \inf_{t \in {{\cal S}_{M}}} m_{\e}(t).
$$
Then, for a countable dense subset  ${{\cal D}_M}$ of ${{\cal S}_M}$ there is a sequence $\{ \e_i \}$ with $\e_i \to 0$ such that for $t \in {{\cal D}_M}$, $\lim_{i \to \infty} \l_{\e_i}(t) > M$. From the continuity of $\l_\e(t)$ in $t$ it follows that for all $t \in {{\cal S}_M}$:
$$
\lim_{i \to \infty} \l_{\e_i}(t) > M.
$$ 
Therefore as $\e_i \to 0$ we have $\l_{\e_i}({{\cal S}_M}) \geq M$ and $m_{\e_i}({{\cal S}_M}) \geq \a$.

Consider (\ref{equ:inequ-on-small-set}) with $\e = \e_i$. Then multiply both sides by $\l_{\e_i}({{\cal S}_M})^{-2} m_{\e_i}({{\cal S}_M})^{-1}$ to get:
\begin{eqnarray} \nonumber
\label{equ:inequ-on-small-set-scaled}
&&\int_{{\cal S}_{M}} \frac{\l^2_{\e_i}(t)}{\l^2_{\e_i}({{\cal S}_M})} \frac{m_{\e_i}(t)}{m_{\e_i}({{\cal S}_M})}    \Big(  \int (\frac{JH_t^{\e_i}}{\l_{\e_i}(t)})_{\e_i} \cdot \frac{JH(t)^{\e_i}}{\l_{\e_i}(t)} d( \frac{\lVert V^{\e_i}(t)\rVert}{m_{\e_i}(t)})  \Big) dt  \\ [.3cm]
 \label{equ:inequ-on-small-set-scaled}
 && < \l_{\e_i}({{\cal S}_M})^{-2} m_{\e_i}({{\cal S}_M})^{-1} 2 || V^{\e_i}(0) ||.
\end{eqnarray}
Clearly for all $t \in {{\cal S}_M}$
$$
\frac{\l^2_{\e_i}(t)}{\l^2_{\e_i}({{\cal S}_M})} \frac{m_{\e_i}(t)}{m_{\e_i}({{\cal S}_M})} \geq 1
$$
Therefore as $\e_i \to 0$ the left hand side of (\ref{equ:inequ-on-small-set-scaled}) is greater than or equal to 
$$ 
\int_{{\cal S}_M} \int JX \cdot JX dV dt = \int_{{\cal S}_M} dt = m({{\cal S}_M}).
$$
The limit of the right hand side is bounded above by
$$
M^{-2} \a^{-1} 2 || V(0) ||.
$$
Therefore
$$
m({{\cal S}_M}) \leq M^{-2} \a^{-1} 2 || V(0) ||.
$$
The result follows.
\end{proof}

\bigskip

Set:
$$
{\cal T}_M = \{ t \in [0,T]: \limsup_{\e \to 0} \l_{\e}(t) \leq M \}
$$

\begin{thm}
\label{thm:bdd-first-variation}
For each $t \in {\cal T}_M$ the lagrangian varifold $V(t)$:
\begin{enumerate}
\item satisfies $(H_2)$. \\
\item the generalized mean curvature of $V(t)$, $H(t) \in L^2(V(t))$ satisfies the bound $$ \lVert H(t) \rVert_{L^2(V(t))} \leq M$$ \\
\item the lagrangian angle $\b(t) \in L^\infty(V(t))$ has weak derivative $B(t) \in L^2(V(t))$ and therefore $V(t)$ has Maslov index zero.
\end{enumerate}
\end{thm}

\begin{proof}
The first two items follow from Allard's compactness theorem for varifolds  Theorem \ref{thm:compact-Hp-varifolds}. The last item follows from Theorem \ref{thm:convergence-beta} and Corollary \ref{cor:norm-beta-decreasing}.
\end{proof}

\bigskip

Note that the first and last items in Theorem \ref{thm:bdd-first-variation} are independent of the value of $M$ and that as $M \to \infty$, $m({\cal S}_M) \to 0$. Therefore:

\begin{cor}\label{cor:aezeroMaslov}
For almost every  $t \in [0, T]$ the lagrangian rectifiable varifold $V(t)$ satisfies the conclusions (1) and (3) of Theorem \ref{thm:bdd-first-variation}.
\end{cor}

The limit flow $\{ V(t) \}$ has the property that for almost every $t \in [0,T]$ each lagrangian varifold satisfies $(H_2)$ and has Maslov index zero. However there are no uniform bounds on $ \lVert H(t) \rVert_{L^2(V(t))}$. This failure is partial rectified by the following result:

\begin{thm}
\label{thm:good-times-MC-bounds}
Suppose  $\inf_{[0,T]} \lVert V(t)\rVert\geq \alpha$. Then:
$$
\int_0^T \lVert H(t) \rVert^2_{L^2(V(t))} dt \leq \frac{4}{\a} \lVert V(0) \rVert^2.
$$
\end{thm}

\begin{proof}
From the proof of Theorem \ref{thm:first-bound-MC}  there is an ${\bar \e} > 0$ such that for each $t \in [0,T]$ and $0 < \e < {\bar \e}$:
$$
 \int (JH^\e(t))_\e \cdot JH^\e(t) d\lVert V^\e(t)\rVert < - 2 \frac{d}{dt} \lVert V^{\e}(t) \rVert
$$
Using the notation $\l_\e(t)$ and $m_\e(t)$ as defined  above we have:
$$
m_\e(t)  (\l_\e(t))^2 \int (\frac{JH^\e(t)}{\l^\e(t)})_\e \cdot \frac{JH^\e(t)}{\l_\e(t)} d( \frac{\lVert V^\e(t)\rVert}{m_\e(t)}) < - 2 \frac{d}{dt} \lVert V^{\e}(t)t \rVert
$$
Now $m_\e(t) \geq \lVert V(t)\rVert \geq \a$. Therefore:
$$
 (\l_\e(t))^2 \int (\frac{JH^\e(t)}{\l_\e(t)})_\e \cdot \frac{JH^\e(t)}{\l_\e(t)} d( \frac{\lVert V^\e(t)\rVert}{m_\e(t)}) < - \frac{2}{\a} \frac{d}{dt} \lVert V^{\e}(t) \rVert
$$
Recall that as $\eta \to 0$, $(\frac{JH^\e(t)}{\l_\e(t)})_\eta \to \frac{JH^\e(t)}{\l_\e(t)}$ strongly in $L^2(V^{\e}(t))$. Therefore for any $t$ and $\e$, as $\eta \to 0$:
$$
\int (\frac{JH^\e (t)}{\l_\e(t)})_\eta \cdot \frac{JH^\e(t)}{\l_\e(t)} d( \frac{\lVert V^\e(t)\rVert}{m_\e(t)}) \to 1.
$$
It follows that  for any $t \in [0, T]$ and $\e \in [0, 1]$ there is an $\eta_0 > 0$ such that for $\eta < \eta_0$:
$$
\int (\frac{JH^\e(t)}{\l_\e(t)})_\eta \cdot \frac{JH^\e(t)}{\l_\e(t)} d( \frac{\lVert V^\e(t)\rVert}{m_\e(t)}) > \frac{1}{2}.
$$
In particular, there is an $\e_0 > 0$ such that for $\e < \e_0$:
$$
\int (\frac{JH^\e(t)}{\l_\e(t)})_\e \cdot \frac{JH^\e(t)}{\l_\e(t)} d( \frac{\lVert V^\e(t)\rVert}{m_\e(t)}) > \frac{1}{2}.
$$
Hence  for any $t \in [0, T]$ and $\e < \min(\e_0, {\bar \e})$:
\begin{eqnarray} \nonumber
\label{equ:MCintegral-ineq}
\frac{1}{2}  (\l^\e(t))^2 &<&  (\l_\e(t))^2 \int (\frac{JH^\e(t)}{\l_\e(t)})_\e \cdot \frac{JH^\e(t)}{\l_\e(t)} d( \frac{\lVert V^\e(t)\rVert}{m_\e(t)}) \\ 
&<& - \frac{2}{\a} \frac{d}{dt} \lVert V^{\e}(t) \rVert \\ \nonumber
\end{eqnarray}
Thus when $\e < \min( \e_0, \bar{\e})$ we can 
integrate (\ref{equ:MCintegral-ineq}) to derive:
$$
\int_0^T  (\l_\e(t))^2 dt < \frac{4}{\a} \lVert V^{\e}(0) \rVert = \frac{4}{\a} \lVert V(0) \rVert
$$
Recall that for each $\e > 0$ the masses $m_\e(t) = ||V^\e(t)||$ are bounded independent of $t \in [0,T]$ by a constant $A_\e$ and that as $\e \to 0$, $A_\e \to ||V(0)||$.
Since 
$$
\int_0^T  (\l_\e(t))^2 dt = \int_0^T  \int | H^\e(t)|^2 d( \frac{\lVert V^\e(t)\rVert}{m_\e(t)}) dt =   \int_0^T m_\e(t)^{-1} || H^\e(t) ||^2_{L^2(V^\e(t))} dt.
$$
It follows that,
$$
\int_0^T  || H^\e(t) ||^2_{L^2(V^\e(t))} dt < A_\e \int_0^T  (\l_\e(t))^2 dt   < \frac{4}{\a} \lVert V(0) \rVert
$$
Let $\{ \e_i \}$ be a sequence  such that  $A_{\e_i} \to ||V(0)||$ as $\e_i \to 0$. Then:
\begin{eqnarray*}
\int_0^T  \lVert H(t) \rVert^2_{L^2(V(t))} dt &\leq& \limsup_{\e_i \to 0} A_{\e_i} \frac{4}{\a} \lVert V(0) \rVert \\
&\leq& ||V(0)||  \frac{4}{\a} \lVert V(0) \rVert \\
&=& \frac{4}{\a} \lVert V(0) \rVert^2
\end{eqnarray*}
\end{proof}

\bigskip

\begin{thm}
Given a lagrangian integer varifold $V_0$ that satisfies $(H_2)$ with Maslov index zero  there exists a lagrangian flow $\{ V(t) \}$ with $V(0) = V_0$ that exists for all time $t >0$.
\end{thm}

\begin{proof}
By the constructions of the previous sections given an interval $[0,T]$ there is a lagrangian flow $\{ V(t) \}$ with $V(0) = V_0$ on $[0,T]$. If $V(t_0) = 0$ for some $t_0 \in[0,T]$ then extend the flow by the zero varifold. 
Otherwise there is a constant $\a > 0$ such that $|| V(t) || \geq \a > 0$ for all $t \in [0,T]$. Theorem \ref{thm:good-times-MC-bounds} applies. It follows that there is a $t_0 \in [ \frac{T}{2}, T]$ such that $V(t_0)$ is a lagrangian integer varifold  with Maslov index zero satisfying $(H_2)$. Apply our construction with initial  varifold $V(t_0)$ to construct a lagrangian flow on $[ t_0, t_0 + T]$. Iterating, the result follows.
\end{proof}

\begin{thm}
\label{thm:H2-bounded-infinity}
Given any lagrangian flow $\{ V(t) \}$, if $|| V(t) || \geq \a > 0$ for $t \in [0,T_0]$ then
\begin{equation}
\int_0^{T_0} \lVert H(t) \rVert^2_{L^2(V(t))} dt \leq \frac{4}{\a} \lVert V(0) \rVert^2.
\end{equation}
In particular, if $\{ V(t) \}$ is a lagrangian flow with $|| V(t) || \geq \a > 0$ for all $t > 0$ then
\begin{equation}
\int_0^{\infty} \lVert H(t) \rVert^2_{L^2(V(t))} dt \leq \frac{4}{\a} \lVert V(0) \rVert^2.
\end{equation}
\end{thm}

We conclude:

\begin{thm}
	The flow $\{V(t)\}$ can be extended for all $t>0$ while preserving the properties:
		\begin{enumerate}
			\item Each $V(t)$ is a lagrangian integer $n$-varifold.
			\item $\lVert V(t)\rVert$ is a nonincreasing, lower semicontinuous function of $t$.
			\item For each $V(t)$ the lagrangian angle, $\beta(t)$, admits a scalar-valued lift so that $\lVert \beta(t)\rVert_{L^\infty(V(t))}$ is a nonincreasing, lower semicontinuous function of $t$.
			\item For almost every $t>0$, $V(t)$ satisfies $(H_2)$.
	                  \item For almost every $t>0$, $\beta(t)$ has a weak derivative in $B(t)\in L^2(V(t))$, i.e.~$V(t)$ has Maslov index zero.
		\end{enumerate}
\end{thm}

\bigskip

\section{Lagrangian Currents in a Calabi-Yau manifold}

Now we will show how to adapt our construction of the hamiltonian flow of lagrangian varifolds to a flow of integral cycles in a Calabi-Yau manifold. Observe that the definitions of weak derivative and Maslov index zero are sensible in this context.

\subsection{The varifold flow on a Calabi-Yau manifold}\label{sec:manifoldcase} First consider a lagrangian varifold $V$ with $H\in L^2(V)$ and Maslov index zero in a closed Calabi-Yau manifold $(N,\omega, J)$.  We  define 
$$
\phi_\e(x,y)=
\begin{cases}
C \e^{-n} \exp \bigg( \frac{\e^2}{\dist(x,y)^2 - \e^2} \bigg) \;\; \mbox{if} \; \dist(x,y) < \e, \\
0   \;\; \mbox{if} \; \dist(x,y) \geq \e. \\
\end{cases}
$$
where $\operatorname{dist}$ is the Riemannian distance with respect to the Calabi-Yau metric on $N$. We will assume that $\e$ is always less than the injectivity radius of $N$. We mollify according to 
	\begin{equation}
		f_\e(x)=\frac{\int f(y)\phi_\e(x,y)d\lVert V\rVert y}{\int \phi_\e(x,z)+\e d\lVert V\rVert z}
	\end{equation}
Since $\epsilon$ is smaller than the injectivity radius of $N$, $\phi_\e$ and hence $f_\e$ is smooth.

If the ball $B_\e(x)$ has compact closure in a convex coordinate neighborhood $U \subset \R^{2n}$, then alternately we could use the euclidean distance $|x-y|$ between $x, y \in  B_\e(x)$. However in ${\bar U}$ the euclidean distance and the distance in the Calabi-Yau metric are equivalent. The constructions and estimates in previous sections all carry over to this mollifier up to fixed constants. 

In particular, we have the following theorem:
\begin{thm}\label{thm:varifold-flow-CY}
	Given any Lagrangian varifold $V$ with $H\in L^2(V)$ and Maslov index zero in a closed Calabi-Yau manifold $(N,\omega, J)$, there is a one-parameter family $V(t)$ of Lagrangian varifolds in $N$, defined for all $t>0$, with $V(0)=V$ and 
		\begin{enumerate}
			\item Each $V(t)$ is a lagrangian integer varifold.
			\item $\lVert V(t)\rVert$ is a nonincreasing, lower semicontinuous function of $t$.
			\item Each $V(t)$ admits a scalar lift of the lagrangian angle, $\beta(t)$, so that $\lVert \beta(t)\rVert_{L^\infty(V(t))}$ is a nonincreasing, lower semicontinuous function of $t$.
			\item For almost every $t>0$, $V(t)$ satisfies $(H_2)$.
			\item For almost every $t>0$, $\beta(t)$ has a weak derivative $B(t)\in L^2(V(t))$, i.e.~$V(t)$ is Maslov index zero.
		\end{enumerate}\end{thm}

\bigskip

\subsection{The flow of lagrangian cycles}

The basic object we will use is the integer multiplicity current as described in \cite{s} and originally introduced by Federer and Fleming \cite{ff}. As described in \cite{s} the integer multiplicity currents are obtained by assigning an orientation to the approximate tangent space $T_xV$ of an integer multiplicity rectifiable varifold $V$. This allows the integration of $n$-forms along the varifold making the varifold into a current. We will always require that the associated  integer multiplicity rectifiable varifold $V$ is lagrangian and satisfies $(H_2)$.

\begin{defn} Let $T$ be an integer multiplicity current. If the rectifiable varifold $V_T$ associated to $T$ is lagrangian and satisfies $(H_2)$, we will call $T$ a {\em lagrangian integer multiplicity current} or briefly a  {\em lagrangian current}. If $\partial T=0$, we will call $T$ a {\em lagrangian integer multiplicity cycle} or briefly a  {\em lagrangian cycle}.
\end{defn}

The essential compactness result for lagrangian currents and cycles is a consequence of results of Federer-Fleming \cite{ff} and Allard \cite{a}.  Proofs can be found in \cite{s} and \cite{white}.  We recall the definition of mass and weak convergence of currents:

\begin{defn}
	Let $T$ is an $n$-current in the open set $U$.  Denote the set of $n$-forms with support in $U$ by $\L^n(U)$.  The mass of the current $T$, denoted $M(T)$ is given by:
$$
M(T) = \sup_{ |\o| \leq 1, \o \in \L^n(U)} T(\o).
$$
More generally for any open $W \subset U$ we set:
$$
M_W(T) = \sup_{ |\o| \leq 1, \o \in \L^n(U), \supp \o \subset W} T(\o).
$$
We say a sequence $\{ T_i \}$ converges weakly to a current $T$, denoted $T_i \rightharpoonup T$, if $\lim T_i(\o) = T(\o)$ for all $\o \in \L^n(U)$
\end{defn}

\begin{thm}
\label{thm:current-compactness}
Let $U \subset \R^{2n}$ be open. If $\{ L_i \}$ is a sequence of lagrangian currents (cycles) supported in $U$ with:
$$
\sup_{i \geq 1} ( M_W(L_i) + M_W(\p L_i) ) < \infty \;\;\; \mbox{for all} \; W \subset U \;\; \mbox{with} \; \bar{W} \subset U,
$$
then there is a lagrangian current (cycle) $L$ supported in $U$ and a subsequence $\{ L_{i_j} \}$ such that $L_{i_j} \rightharpoonup L$.
\end{thm}

\begin{proof}
The proof follows from Theorems 27.3 and 42.7 in \cite{s} and Allard's integral compactness theorem in \cite{a}. The lagrangian condition is easily seen to be preserved under weak convergence of currents.
\end{proof}

\begin{defn}\label{defn:zero-Maslov-current} An lagrangian current (cycle) with associated lagrangian varifold that has Maslov index zero will be called a {\em lagrangian Maslov index zero current (cycle)}.
\end{defn}

\bigskip

Given a lagrangian Maslov index zero cycle $I$ whose associated varifold $V$ satisfies $(H_2)$, the construction of the $\e$-flows in Section \ref{sec:epsilon-flows} can be applied to the associated lagrangian varifold to produce, for each $\e>0$, a one-parameter family of hamiltonian diffeomorphisms $\Psi^\e(t)$, hence a one-parameter family of lagrangian Maslov index zero cycles $I^\e(t)=\Psi^\e(t)_\#I$. We define, analogously to Corollary \ref{cor:constructlimitflow}, $I(t)$ to be the cycle with 
	\begin{equation}
		M(I(t))=\liminf_{\e\rightarrow 0}M(I^\e(t))
	\end{equation}
where we use the compactness theorem \ref{thm:current-compactness} to ensure that there is such a limit.

In the manifold case we have:

\begin{thm}
Let $M$ be a compact Calabi-Yau manifold. Let $\L \in H_n(M, \R)$ be a lagrangian homology class (a class that can be represented by a lagrangian cycle). Suppose that $\L$ can be represented by a lagrangian Maslov index zero cycle. Then the lagrangian cycles $\{I^\e(t)\}$ of the $\e$-flows and the lagrangian cycles $\{I_t \}$ of the limit lagrangian flow all have Maslov index zero and all represent $\L$. 
\end{thm}

\begin{proof}
$I^\e(t)$ is related to  $I$ by a one-parameter family of hamiltonian diffeomorphisms, hence represents the same homology class. Moreover weak convergence of currents preserves homology class.
\end{proof}

\begin{cor}
Suppose the lagrangian homology class $\L \in H_n(M, \R)$  is non-trivial. Set 
$$
m(\L) = \inf_T M(T),  
$$
where $T$ is a cycle (not necessarily lagrangian) that represents $\L$. Then $m(\L) =\a > 0$ is a lower bound on the mass of all the varifolds  $\{V^\e(t) \}$ of the $\e$-flows and  $\{V(t) \}$ of the  limit lagrangian flow.
\end{cor}


\begin{cor}
	If $[I]\neq 0\in H_n(M,{\mathbb R})$, then there is a sequence of times $t_j\rightarrow \infty$ so that $I(t_j)$ converges to a stationary current, which we may take to have Maslov index zero. 	
\end{cor} 
\begin{proof}
	By Theorem \ref{thm:H2-bounded-infinity}, we have:
		\begin{equation}
			\int_0^\infty \lVert H\rVert^2_{L^2(V(t))} dt< \frac{4}{\a} \lVert V(0) \rVert^2.
		\end{equation}
	so there is a sequence of times $t_j\rightarrow \infty$ with $\lVert H\rVert_{L^2(V(t_j))}\rightarrow 0$. 
	
	By Theorem \ref{thm:varifold-flow-CY} we may choose the sequence $t_j$ so that each $I(t_j)$ has Maslov index zero. By Theorem \ref{thm:varifold-flow-CY} there are uniform bounds on $\beta(I(t_j))$. Since $\lVert H\rVert_{L^2(V(t_j))}\rightarrow 0$, we obtain uniform bounds on the weak derivative $B(I(t_j))$.
	
	The corresponding $I(t_j)$ have uniformly bounded mass, so by the compactness Theorems \ref{thm:convergence-beta} and \ref{thm:current-compactness}, passing to a subsequence we obtain some limit cycle $I(\infty)$ and associated Maslov index zero varifold $V(\infty)$ with $\lVert H\rVert^2_{L^2( V(\infty))}=0$.
\end{proof}

Denote the limit stationary cycle by $I(\infty)$

\bigskip

\section{The structure of $I(\infty)$}

In this section we prove:
 
	\begin{thm}\label{thm:convergence}
		$I(\infty)$ is the sum of finitely many special lagrangian integral cycles. 
	\end{thm}

A smooth lagrangian submanifold is special lagrangian if and only if it is stationary (and in this case it has vanishing Maslov class); we have the following generalization of this fact to our setting.
\begin{prop}
\label{prop:sum-SL}
	Suppose $I$ is a lagrangian current with finite mass and no boundary in the ball $B_{2R}(0)$, such that the corresponding varifold $V$ is stationary and has Maslov index zero. Then there are  special Lagrangian currents, $L_1,\ldots, L_N$, with $I=L_1+\cdots+L_N$ in $B_{2R}(0)$.
\end{prop}

Since $V$ is stationary, by Allard's regularity theorem there are an open dense set of $p\in \supp\lVert V\rVert$ such that for each $p$  there is some $r>0$ such that in $B_r(p)$, $\supp\lVert V\rVert$ is  smooth and hence special lagrangian. In particular we can select such a $p$ with $B_R(p)\subset B_{2R}(0)$. Call the phase of this special lagrangian $\theta_1$. The idea is to show that $V_1=\beta^{-1}(\theta_1)$, where $\b$ is the lagrangian angle of $V$, contains some definite positive amount of the mass of $V$. However to make sense of $\b^{-1}$ we need $\b$ to be an ambient  Lipschitz function. Therefore we mollify $\b$ to construct a smooth ambient function $\b_\eta$.

\begin{lem}\label{lem:coarea}
		For almost every $s\in{\mathbb R}$, 
			\begin{equation*}
				\limsup_{\eta\rightarrow 0}{\mathcal H}^{n-1}(\supp\lVert V\rVert\cap \beta_\eta^{-1}(s))=0
			\end{equation*}
	\end{lem}
	\begin{proof}
		For each $\eta>0$, we may apply the coarea formula for rectifiable sets, namely
			\begin{equation}
				\int_{-\infty}^\infty{\mathcal H}^{n-1}(\supp\lVert V\rVert \cap \beta_\eta^{-1}(s))ds=\int_{\supp\lVert V\rVert}\lvert \nabla\beta_\eta\rvert d{\mathcal H}^n
			\end{equation}
		The conclusion follows from the error estimate on $\beta_\eta$:
			\begin{equation}
				\begin{aligned}
					\int_{\supp\lVert V\rVert}\lvert \nabla \beta_\eta\rvert d{\mathcal H}^n&\leq \int \lvert\nabla\beta_\eta\rvert d\lVert V\rVert\\
					&=\int\lvert B_\eta-\beta_\eta H_\eta+E \rvert d\lVert V\rVert \rightarrow 0
					\end{aligned}
			\end{equation}
		where (appealing to Theorem \ref{thm:harvey-lawson-varifold}) $B_\eta$ and $\beta_\eta H_\eta$ are both zero because $V$ is stationary. The remaining term satisfies
		\begin{equation}
		\int \lvert E \rvert d\lVert V\rVert \leq  \Big( \int \lvert E \rvert^2 d\lVert V\rVert \Big)^{\frac{1}{2}} \Big( \int d\lVert V\rVert \Big)^{\frac{1}{2}}
		 \end{equation}
		However, as $\eta \to 0$
		\begin{equation} 
		\int \lvert E \rvert^2 d\lVert V\rVert \to 0,
		  \end{equation}   
		by Lemma \ref{lem:sequence-Allard2}. The result follows.
		
	\end{proof}

\noindent{\it Proof of Proposition \ref{prop:sum-SL}}	

Consider, for each $\eta,\e>0$,
		\begin{equation}
			L_{\eta,\e}=I\lrcorner\left( \beta_\eta^{-1}(\theta_1-\e,\theta_1+\e)\cap B_{2R}\right)
		\end{equation}
	Each $L_{\eta,\e}$ an integral current whose support lies in the support of $V$.  We seek a limit of  $L_{\eta,\e}$ as $\eta,\e\rightarrow 0$ in the sense of integral currents and in the sense of varifolds. The masses of the $L_{\eta,\e}$ are dominated by the mass of $V$. Note that by hypothesis, $I$ is without boundary.  Thus the only part of $\partial L_{\eta,\e}$ in $B_{2R}$, has support in 
		\begin{equation}
			\supp\lVert V\rVert \cap \beta_\eta^{-1}(\theta_1\pm \e)\cap B_{2R}
		\end{equation}
	Now we can apply Lemma \ref{lem:coarea} to choose $\eta_k,\e_k\rightarrow 0$ for which 
		\begin{equation}
			\limsup_k {\mathcal H}^{n-1}\left(\supp \lVert V\rVert\cap\beta_{\eta_k}^{-1}(\theta_1\pm \e_k)\cap B_{2R}\right)=0
		\end{equation}
	In particular, there is a sequence of the $L_{\eta,\e}$ converging to some $L_1$, an integral lagrangian current without boundary in $B_{2R}$, with $\supp L_1\subset \supp\lVert V\rVert$. Let $V_1$ be the associated lagrangian varifold to $L_1$ then we  have $\beta=\theta_1$ $V_1$-a.e.. Thus, $V_1$ is stationary and we can  apply the monotonicity formula about $p$ with radii $r$ and $R$:
	\begin{equation}
		\frac{V_1\left(B_{2R}\right)}{R^n}\geq \frac{V_1\left(B_{R}(p)\right)}{R^n}\geq \frac{V_1\left(B_{r}(p)\right)}{r^n}=\frac{V\left(B_r(p)\right)}{r^n}\geq \omega_n.
	\end{equation}
	The last inequality follows because $p$ is a regular point of $V$. This inequality shows that $L_1$ accounts for at least $\omega_nR^n$ of the mass of $V$. Hence we can repeat the argument for some $p'\in \supp( I-L_1)$ to find $L_2$, and so on until the mass of $V$ is exhausted.

\begin{rem}
	The above proof is inspired by the proof of Proposition 5.1 of \cite{nzm}.
\end{rem}

To establish Theorem \ref{thm:convergence}, we need only note that $V(\infty)$ is stationary; then we may apply Proposition \ref{prop:sum-SL} to  $I(\infty)$ obtain the desired result.

\begin{rem}
 In the decomposition $I(\infty)=L_1+\cdots+L_N$, The phases of the special lagrangians may differ. In particular, the configuration of special lagrangians may not be minimizing.
\end{rem}

\bigskip

\section{Collapse in the euclidean case}

Classical compact mean curvature flows in Euclidean space must encounter singularities in finite time, but these singularities may not be caused by global volume collapse. Here we show that lagrangian varifold flow exhibits the same behavior, in fact collapsing to zero mass in finite time.

\begin{thm}\label{thm:collapse}
	If $\supp V$ lies in a ball of radius $R$, then $\supp V(t)$ lies in a ball of radius $\sqrt{R^2-2nt}$.
\end{thm}
\begin{proof}
	Consider the function $f(x,t)=\lvert x\rvert^2+2nt$. Compute, for each $V^\e(t)$,
		\begin{eqnarray} \nonumber
			&&\frac{d}{dt} \lVert V^\e(t)\rVert (f(t))  \\ \nonumber
			&&=\int 2x\cdot JD\beta_\e + 2n +f(x,t) \div JD\beta_\e d\lVert V^\e(t)\rVert x\\ \nonumber
			&&=\int 2x\cdot JD\beta_\e + 2n -\nabla \lvert x\rvert^2 \cdot JD\beta_\e - f(x,t)H^\e\cdot JD\beta_\e d\lVert V^\e(t)\rVert x	\\
		\end{eqnarray} 
	For almost every $t$,  as $\e\rightarrow 0$, $JD\beta_\e\rightarrow H$. For such $t$, we have
		\begin{equation}
			\begin{aligned}
			\frac{d}{dt} \lVert V^\e(t)\rVert (f(t))\rightarrow \int 2x\cdot H + 2n - f\lvert H\rvert^2 d\lVert V(t)\rVert	
			\end{aligned}
		\end{equation} 
	Now an elementary computation shows that $\div x=n$, so we have
	\begin{equation}
		  \int n\ d\lVert V(t)\rVert x=\int \div x\ d\lVert V(t)\rVert x = -\int H\cdot x\ d\lVert V(t)\rVert x	
	\end{equation}
	hence
	\begin{equation}
			\begin{aligned}
			\frac{d}{dt} \lVert V^\e(t)\rVert (f(t))\rightarrow \int - f\lvert H\rvert^2 d\lVert V(t)\rVert	\leq 0
			\end{aligned}
		\end{equation} 
	with equality holding only if $V(t)$ is stationary or the zero varifold.
	
	Thus for small enough $\e$, $\lVert V^\e(t)\rVert (f(t))$ is nonincreasing in $t$. Letting $\e\rightarrow 0$, we have $\lVert V(t)\rVert (f(t))$ nonincreasing in $t$. \\
	
	Similarly, for any $p\geq 2$, we have
		\begin{equation}
			\frac{d}{dt} \lVert V^\e(t)\rVert (f(t)^p)\rightarrow \int pf^{p-1}\left(2x\cdot H+2n\right)-f^p\lVert H\rVert^2d\lVert V(t)\rVert
		\end{equation}
	and
	\begin{equation}
		\begin{aligned}
		  - \int f^{p-1}H\cdot x\ d\lVert V(t)\rVert_x&=\int \div \left(f^{p-1} x\right)\ d\lVert V(t)\rVert_x\\
		  & = \int (p-1)f^{p-2}\nabla f\cdot x+f^{p-1}n\ d\lVert V(t)\rVert_x
		 \end{aligned}
	\end{equation}
	and $\nabla f\cdot x = 2x\cdot x^\top\geq 0$, so that
	\begin{equation}
			\frac{d}{dt} \lVert V^\e(t)\rVert (f(t)^p)\rightarrow  \int -4p(p-1)f^{p-2}x\cdot x^\top-f^p\lvert H\rvert^2d\lVert V(t)\rVert_x\leq 0
		\end{equation}
	Letting $\e \rightarrow 0$, we conclude that $\lVert f(t)\rVert_{L^p( V(t))}\leq \lVert f(0)\rVert_{L^p( V(0))}$. Letting $p\rightarrow \infty$, we conclude that $\lVert f(t)\rVert_{L^\infty(V(t))}\leq \lVert f(0)\rVert_{L^\infty(V(0))}$. \\
	
	But since $f$ is continuous, we have that for  $x\in\supp V(t)$, $$\lvert x\rvert^2+2nt=f(t)\leq \max f(t)\leq \max f(0)\leq R^2.$$
	
\end{proof}

\begin{rem}
	The proof of Theorem \ref{thm:collapse} follows that of Brakke \cite{b}. It is almost --- but not precisely --- a barrier argument. The rate of collapse $R^2-2nt$ is slower than that ($R^2-2(2n-1)t$) determined by mean curvature flow applied to the $2n-1$-sphere bounding the ball of radius $R$ which contained the support of the initial varifold.
\end{rem}

\begin{cor}
	For any initial varifold with compact support, the flow $V(t)$ becomes the zero varifold in finite time.
\end{cor}

\begin{cor}
	Hyperplanes are barriers for the lagrangian varifold flow.
\end{cor}

By a similar argument one has the following result, which controls the rate of collapse.
\begin{thm}
	If the support of $V$ is disjoint from the ball of radius $R$, then the support of $V(t)$ is disjoint from the ball of radius $\sqrt{R^2-2nt}$.
\end{thm}
\bigskip

\section{Conclusions}

\subsection{Topological ramifications}

Finally we state the main topological result of the paper.
\begin{thm}\label{thm:maintopological}
	Let $M$ be a closed Calabi-Yau manifold and $\Lambda \in H_n(M,{\mathbb R})$. If $\Lambda$ can be represented by a lagrangian cycle with Maslov index zero and $H\in L^2$, then $\Lambda = \alpha_1+\cdots+\alpha_k$, for some $\alpha_i\in H_n(M,{\mathbb R})$, each of which can be represented by a special lagrangian cycle.
\end{thm}

\begin{proof}
	Let $I\in \Lambda$ be a cycle with Maslov index zero. Then $I(t)\in\Lambda$ exists for all time and we may extract a sequential limit $I(\infty)\in\Lambda$ which is the sum of special lagrangian cycles. 
\end{proof}

\begin{cor}
 Let $N$ be a closed Calabi-Yau manifold. 
If an integral lagrangian homology class $\a \in H_n(N; \Z)$ can be represented by an immersed lagrangian submanifold with vanising Maslov class, then $\a = \a_1 + \dots + \a_k$ where each $\a_i \in H_n(N; \Z)$ is a lagrangian homology class that can be represented by a special lagrangian current. The phases of the calibrating $n$-forms may be different for each $i=1, \dots, k$. 
 \end{cor} 
 \begin{proof}
 By Proposition \ref{prop:immersedZM}, the immersion representing $\a$ is a cycle with Maslov index zero. Theorem \ref{thm:maintopological} applies.
 \end{proof}
 
 We also have the following result, which resolves the Thomas-Yau conjecture in the weak setting:
 \begin{thm}
 \label{cor:Thomas-Yau}
 If $\Sigma$ is an embedded lagrangian submanifold with vanishing Maslov class, then there is a mass-decreasing flow of lagrangian currents starting from $\Sigma$ and converging (in infinite time) in a sum $L_1+\cdots+L_N$ of special lagrangian cycles.
 \end{thm}

Let $(N, \omega)$ be a closed Calabi-Yau manifold of complex dimension $n$ with K\"ahler form $\omega$. There are two reasonable definitions of lagrangian homology class. The one used in the introduction: A class in $H_n(N, \Z)$ is called {\it a lagrangian homology class} if it can be represented by a simplex consisting of simplices with all $n$-simplices lagrangian. The topological definition:  A class in $\a \in H_n(N, \Z)$ is called {\it a lagrangian homology class} if $\a \cap [\omega] = 0$. See \cite{w2} for results relating these two definitions. The results on lagrangian homology stated in the introduction apply with either definition. Denote the subspace of lagrangian homology classes by $LH_n(N, \Z)$. Let ${\cal S} \subset LH_n(N, \Z)$ be the subspace of the lagrangian homology that is generated by the special lagrangian cycles with any phase. 
 
 \begin{thm}
 \label{thm:homology-image}
Let $N$ be a closed Calabi-Yau manifold of complex dimension $n$. 
If an integral lagrangian homology class $\a \in H_n(N; \Z)$ can be represented by the image of a smooth map $f: M \to N$, where $M$ is a simply connected closed $n$-manifold then $\a \in {\cal S}$. In particular, the image of the Hurewicz homomorphism $\pi_n(N) \to H_n(N, \Z)$ in the lagrangian homology lies in ${\cal S}$.
 \end{thm} 
 
 \begin{proof}
 Suppose the image of a smooth map $f: M \to N$ represents a lagrangian homology class (for either definition). Then $f^*([\omega]) =0$. Using the h-principle \cite{l} \cite{g} $f$ is homotopic to a lagrangian immersion $\ell: M \to N$. Perturbing the lagrangian immersion we can suppose it has a at worst finite number of isolated double points. Since $M$ is simply connected it follows that  $\ell(M)$  is a lagrangian cycle with Maslov index zero and $H \in L^2$. The result follows.
 \end{proof}
 
  \begin{cor}
  \label{cor:K3}
Let $N$ be a simply connected closed Calabi-Yau manifold of complex dimension $n$ with $2 \leq n \leq 6$. Then all lagrangian homology classes lie in  ${\cal S}$. 
   \end{cor} 
   
   \begin{proof}
In the case $n=2$ every homology class in $H_2(N, \Z)$ lies in the image of the Hurewicz homomorphism. The result then follows, in this case, directly from Theorem \ref{thm:homology-image}. 

We can assume $n \geq 3$. We will show that if the class  $\a \in H_n(N; \Z)$ can be represented by the image of a smooth map $f: M \to N$ then $\a$ can be represented by the image of a smooth map ${\Tilde f}: {\Tilde M} \to N$ where $\Tilde M$ is simply-connected. First we observe that the map $f$ can be assumed to be an immersion with, at worst, double points. If $M$ is not simply-connected there is a homotopically non-trivial curve $\g \subset M$. We can assume that $f(\g)$ is imbedded in $N$ and therefore spans an imbedded disc in $N$. We write $\imath: D \to N$ for this imbedded disc and note that $\p D = \g$ and $f = \imath$ along $\g$. In particular, $f \cup \imath: M \cup D \to N$ is continuous and represents $\a \in H_n(N, \Z)$. Let $U \subset M$ be a tubular neighborhood of $\g$. Attach a $2$-handle $D^2 \times D^{n-1}$ to $M \cup D$ by gluing $\p D^2 \times D^{n-1}$ to  $U$ with $D^2 \times {0}$ identified with $D$. Denote the resulting space by $S$. Set  $\Tilde M = M \setminus U \cup D^2 \times \p D^{n-1}$. Then $\Tilde M$ is an $n$-manifold. There is a deformation retract $r$ of $S$ onto $M \cup D$. Therefore there is a continuous map $S \to N$ given by the composition of $r$ and $f \cup \imath$. Denote the restriction of this map to $\Tilde M$ by $\Tilde f$. Then ${\Tilde f}( {\Tilde M})$ represents $\a$ in homology. The map ${\Tilde f}: {\Tilde M} \to N$ is continuous. By Whitney approximation we can approximate ${\Tilde f}$ by a smooth map and hence by an immersion.  Iterate this process until there are no non-trivial curves $\g$ in $M$. 

If $2 \leq n \leq 6$ then by R. Thom's work on the Steenrod problem \cite{th} every class in $H_n(N, \Z)$ can be represented by the image of a smooth map $M^n \to N$.  The result now follows in the  cases $3 \leq n \leq 6$ from Theorem \ref{thm:homology-image}.
 \end{proof}
   
It is not unreasonable to:

  \begin{conj}
 Let $N$ be a simply connected closed Calabi-Yau manifold of any complex dimension. Then all lagrangian homology classes lie in  ${\cal S}$.   
   \end{conj}
   
   This conjecture should be contrasted with the related but quite different problem of characterizing the classes in $LH_n(N, \Z)$ that can be represented by a special lagrangian variety with fixed phase. This is a version of a special lagrangian Hodge conjecture. For this problem it is known that there are lagrangian classes in a $K3$ surface that cannot be represented a special lagrangian variety with fixed phase \cite{w1}. Though by Corollary \ref{cor:K3} all lagrangian classes in any $K3$ surface lie in ${\cal S}$ and therefore can be represented by a sum of special lagrangian varieties possibly with differing phases.

\bigskip

\subsection{Agreement with lagrangian mean curvature flow}
If $\left(\Sigma(t)\right)_{t\in[0,T)}$ is a lagrangian mean curvature flow and $\Sigma(0)$ has vanishing Maslov class, so does each $\Sigma(t)$ (see e.g.~\cite{nzm}). In this situation the lagrangian angle $\beta(t)$ of each time-slice $\Sigma(t)$ can be locally extended to an ambient smooth function (still denoted $\beta(t)$) so that $\Sigma(t)$ is generated by the hamiltonian motion $JD\beta(t)$. In this case, the GMT estimates we rely on all have improved versions coming from the fact that the $\Sigma(t)$ satisfy uniform $C^{2,\alpha}$ bounds on compact time intervals, so it can be shown that $JD\beta^\e(t)\rightarrow JD\beta(t)$ uniformly as $\e\rightarrow 0$.

In particular, we have the following
\begin{thm}
If $\left(\Sigma(t)\right)_{t\in[0,T)}$ is a lagrangian mean curvature flow with vanishing Maslov class, the lagrangian varifold flow $V(t)$ starting from the integer-rectifiable varifold corresponding to $\Sigma(0)$ is unique and $V(t)$ is the integer-rectifiable varifold corresponding to  $\Sigma(t)$ for each $t\in[0,T)$.
\end{thm}

Classical mean curvature flow in general is expected to encounter singularities in finite time. Our construction, on the other hand, exists for all $t>0$.  Thus we may think of the lagrangian varifold flow as  a way to extend classical lagrangian mean curvature flow past singularities, in the case of vanishing Maslov class.\\

We also have a local agreement theorem:
\begin{thm}
Suppose $U\subset N$ is an open set and $V(t)$ is a lagrangian varifold flow so that for each $t\in[\alpha,\omega]$, $V(t)|_U$ is the integer-rectifiable varifold corresponding to the intersection of some embedded lagrangian submanifold $\Sigma(t)$ with $U$. $\Sigma(t)$ satisfy mean curvature flow on $U\times[\alpha,\omega]$.
\end{thm}

\bigskip

\subsection{Nonuniqueness}
Because the construction of the lagrangian flow relies on compactness theorems from geometric measure theory, the lagrangian flow is highly nonunique. In particular:

Consider a lagrangian flow $\{ V(t) : t \geq 0 \}$. Suppose at time $t=a > 0$ the lagrangian varifold $V(a)$ has Maslov index zero and satisfies $(H_2)$. Then we can construct a new lagrangian flow $\{ W(t) : t \geq 0 \}$ with $W(0) = V(a)$. We do not know if the lagrangian flows $\{ V(t) : t \geq a \}$ and $\{ W(t) : t \geq 0 \}$ coincide. Suppose that at time $t=b > a$ the lagrangian varifold $V(b)$ is stationary and therefore consists of components each with constant lagrangian angle. The varifold $V(b)$ is not, in general, a minimizer of volume. Therefore there may be a time $t=c$ such that for $t > c$ the volume of the lagrangian varifolds $W(t)$ is less than the volume of $V(b)$.

\end{document}